\documentclass[a4paper,11pt]{amsart}
\usepackage[all]{xy}
\usepackage{graphicx,amsmath,amssymb,amsthm,amscd,mathrsfs}

\DeclareMathOperator{\supp}{supp}

\DeclareMathOperator{\pr}{pr}

\DeclareMathOperator{\ord}{ord}

\newtheorem{thm}{Theorem}[section]
\newtheorem{df}[thm]{Definition}
\newtheorem{prop}[thm]{Proposition}

\newtheorem{cor}[thm]{Corollary}
\newtheorem{rem}[thm]{Remark}

\newtheorem{exa}[thm]{Example}
\title[Moduli of connections, quivers, and integrable 
deformations]{Moduli spaces of meromorphic connections, quiver varieties, and 
integrable deformations}
\begin{document}
\title[]{Moduli spaces of meromorphic connections, quiver varieties, and 
integrable deformations}
\date{}
\author{Kazuki Hiroe}
\email{kazuki@josai.ac.jp}
\address{Department of Mathematics, Josai University\\ 
1-1 Keyakidai Sakado-shi Saitama 350-0295 JAPAN.}

\maketitle
\begin{abstract}
	This is a note in which we first 
	review symmetries of moduli spaces of stable
	meromorphic 
connections on trivial vector bundles over the Riemann sphere,
and next discuss symmetries 
of their integrable deformations as an application.
We shall give an interpretation of 
the list of 4-dimensional Painlev\'e type equations given by 
Kawakami-Nakamura-Sakai from our classification of 4-dimensional 
moduli spaces of connections.
In the study of the symmetries, a realization 
of the moduli spaces as quiver varieties is given and 
plays an essential role.
\end{abstract}
\section*{Introduction}
This note is designed to review symmetries of moduli spaces of meromorphic 
connections on trivial vector bundles over $\mathbb{P}^{1}$ and 
give an application of these symmetries for their isomonodromic deformations.
In the series of works by Okamoto  \cite{Okam}, it was clarified that 
Painlev\'e equations have affine 
Weyl group symmetries. After these pioneering works, many studies 
of symmetries of Painlev\'e type equations are successfully developed
in connection with the algebraic geometry, 
representation theory of affine Lie algebras and so on (see 
Noumi and Yamada \cite{NY}, Sakai \cite{Sak0}, Sasano \cite{Sas},
Boalch \cite{Bo}
and their references for instance).
On the other hand, 
the recent work of Kawakami, Nakamura and Sakai \cite{KNS} suggests
that many known Painlev\'e type equations 
are 
uniformly obtained from isomonodormic deformations of linear ordinary differential 
equations.
In this note, inspired by their work, we shall introduce a study
of symmetries of isomonodromic deformations from those of 
moduli spaces of meromorphic connections.

In the first section 
we shall explain some relationship among moduli spaces 
$\mathfrak{M(\mathbf{B})}$ (see \S \ref{moduli} for the 
definition) of 
stable meromorphic connections on trivial bundles with at most 
unramified 
irregular singularities over $\mathbb{P}^{1}$,
quiver varieties $\mathfrak{M}_{\lambda}^{\text{reg}}(\mathsf{Q},\alpha)$ 
(see \S \ref{Quiver}) and integrable deformations.
One of  the most remarkable facts which will be introduced there is the 
following.
\begin{thm}[Theorem 5.14 in \cite{H2}]
	For any moduli space $\mathfrak{M(\mathbf{B})}$,
		there exists a quiver $\mathsf{Q}$, a dimension vector 
		$\alpha$ and a complex parameter $\lambda$ such that 
		we have an open embedding 
		\[
			\mathfrak{M}(\mathbf{B})\hookrightarrow 
			\mathfrak{M}_{\lambda}^{\text{reg}}(
			\mathsf{Q},\alpha).
		\]
\end{thm}
Here the embedding is an isomorphism if the number of irregular singular points
is less than or equal to 1.
This embedding is found  
by Crawley-Boevey \cite{C} when 
all singular points are regular singular, by 
Boalch \cite{Boarx} and  Yamakawa with the author 
\cite{HY} when only one singular point is irregular singular and the others 
are regular singular, and 
by the author \cite{H2} for general cases.

From the geometry of quiver varieties, the above embedding leads us to a natural 
proof of the fact:
\begin{thm}[Corollary 5.15 in \cite{H2}]
	The moduli space $\mathfrak{M}(\mathbf{B})$ has a 
		structure as a connected complex symplectic manifold.
\end{thm}
Compare with the results by Boalch \cite{Boa1} and  
Inaba-Saito \cite{IS}, etc. 
This fact should be essential for the theory of isomonodromic deformations
because it is believed that isomonodromic deformations should have 
descriptions as Hamiltonian systems over the moduli spaces.

Also we can determine the condition under which the moduli spaces are nonempty.
\begin{thm}[Corollary 7.13 in \cite{H2}]
	The explicit necessary and sufficient condition for 
		$\mathfrak{M}(\mathbf{B})\neq \emptyset$
		is determined 
		by means of the root system of the quiver $\mathsf{Q}$.
\end{thm}
The above fact is also called a {\em additive Deligne-Simpson problem},
see Kostov \cite{Kos}, Crawley-Boevey \cite{C}, 
Boalch \cite{Boarx}, Yamakawa with the author \cite{HY} and the author 
\cite{H2}, etc.

Further, there is another advantage 
of this realization of moduli spaces as quiver varieties, which shall 
be explained in the second section.
Namely, quiver varieties naturally have Weyl group actions generated by 
the reflection functors. 
Thus we can translate these Weyl group
symmetries of quiver varieties to those of 
the moduli spaces.
On the other hand, the moduli spaces themselves have symmetries generated 
by middle convolutions. Thus  
we shall compare these symmetries of middle convolutions 
and of quivers. 
Afterward, the orbits under these 
Weyl groups are in our interest.
It will be shown that a kind of a finiteness of the fundamental 
domains under these actions:
\begin{thm}
	If we fix the dimension of moduli spaces, then 
		there exist only finite numbers of fundamental 
		spectral types  (see Definition \ref{spectype}) of moduli spaces.
\end{thm}
This was proved by Oshima \cite{O} when all singularities are regular singular 
and by Oshima and the author \cite{HO} for general cases.
We also give the classifications of fundamental 
spectral types of dimension 2 and 4 
for example.

Finally, we shall translate the symmetries of moduli spaces into those of 
isomonodromic deformations. Namely,
Haraoka and Filipuk \cite{HF} showed that 
deformation equations of isomonodromic deformations of 
linear Fuchsian differential equations are invariant under 
middle convolutions. This theory is generalized for 
isomonodormic deformations of irregular singular differential equations by 
Boalch \cite{Bo} and Yamakawa \cite{Y2}.
As a corollary of these facts, 
it will be shown that there exist fundamental domains of isomonodromic 
deformations under the action of middle convolutions. Moreover 
we shall see the finiteness of  spectral types of 
fundamental isomonodromic deformations of a fixed dimension
and the classifications of them for lower dimensional cases.
As an application of these, it will be shown that the spectral types 
appeared in the paper of Kawakami-Nakamura-Sakai \cite{KNS} in which 
they considered confluences of 4-dimensional isomonodromic deformations 
is the complete 
list of fundamental spectral types of dimension 4, 
and moreover we classify their Weyl group
symmetries explicitly.

We should mention related previous works.
In \cite{Yam}, Yamakawa gave another realization of 
moduli spaces as generalized quiver varieties associated 
with non-symmetric root systems and clarified the symmetries 
of moduli spaces from his quivers.
In \cite{Bo},
Boalch gave a formulation of isomonodromic deformations 
as Hamiltonian systems on quiver varieties and studied the symmetries
of them when the meromorphic connections have
only one irregular singular point and an arbitrary number of regular 
singular points.
In this note we do not deal with Hamiltonian equations of isomonodromic 
deformations directly. 
However our study of the symmetries of isomonodromic 
deformations of meromorphic connections with general unramified 
irregular singularities should induce the symmetries of their Hamiltonian 
equations. Also, our study after these works will shed a light over the 
importance of quiver varieties for the study of isomonodromic deformations.
\section{Moduli spaces of meromorphic connections, 
quiver varieties and integrable deformations}
For a commutative ring $R$, $M(n,R)$ denotes the set of $n\times n$
matrices with coefficients in $R$ and $\mathrm{GL}(n,R)\subset M(n,R)$ consists
of invertible elements.
The sheaves of holomorphic functions and meromorphic functions 
on a complex manifold $X$ are written by  $\mathcal{O}_{X}$ and 
$\mathcal{M}_{X}$ 
respectively. In particular when $X=\mathbb{P}^{1}$, we write $\mathcal{O}=
\mathcal{O}_{\mathbb{P}^{1}}$ and $\mathcal{M}=\mathcal{M}_{\mathbb{P}^{1}}$
for short.
Let us denote the ring of convergent (resp. formal) power series of $z$
by $\mathbb{C}\{z\}$ (resp. $\mathbb{C}[\![z]\!]$). 
Their total quotient fields are written by $\mathbb{C}\{\!\{z\}\!\}$ and 
$\mathbb{C}(\!(z)\!)$ respectively.

In this section, we first give a quick review of meromorphic connections 
on vector bundles over $\mathbb{P}^{1}$ and their relation with systems 
of first order linear ordinary differential equations defined on 
$\mathbb{P}^{1}$. Next we introduce moduli spaces of stable meromorphic 
connections on trivial vector bundles over $\mathbb{P}^{1}$ and 
give their realizations as quiver varieties. Finally, we consider 
integrable deformations of connections in these moduli spaces.
\subsection{Gauge equivalences of differential equations}
We recall gauge transformations of systems of first order linear 
ordinary differential equations defined locally on $\mathbb{P}^{1}$ and 
moreover recall Hukuhara-Turrittin-Levelt normal forms of
local differential equations under formal gauge transformations.

Let $U$ be an open subset of $\mathbb{P}^{1}$ and $z$ a local coordinate on $U$.
\begin{df}[gauge transformation]\normalfont
For a linear differential equation 
\[
	\frac{d}{dz}Y=AY
\]
with $A\in M(n,\mathcal{M}(U))$ and $X\in \mathrm{GL}(n,\mathcal{M}(U))$, 
we define a new differential equation
$\frac{d}{dz}\tilde{Y}=B\tilde{Y}$ by 
\[
	B:=XAX^{-1}+\left(\frac{d}{dz}X\right)X^{-1}.
\]
We call $B$ the {\em meromorphic gauge transformation} of $A$ by $X$
and write $B=:X[A]$.
In particular if $X\in \mathrm{GL}(n,\mathcal{O}(U))$, we say 
the {\em holomorhic gauge transformation}.
\end{df}
Here we note that if a vector $Y$ is a solution of  
$\frac{d}{dz}Y=AY$ then $\tilde{Y}=XY$ is a solution of $\frac{d}{dz}\tilde{Y}
=B\tilde{Y}$ for $B=X[A]$.

Let us take $a\in U$ and choose a local coordinate $z$ which is zero at $a$. 
Then the stalks $\mathcal{O}_{a}$ and $\mathcal{M}_{a}$ at $a$ 
can be identified with 
$\mathbb{C}\{z\}$  and $\mathbb{C}\{\!\{z\}\!\}$. 
We can similarly define holomorphic and meromorphic gauge 
transformations of a local differential equation 
$\frac{d}{dz}Y=AY$ with $A\in M(n,\mathcal{M}_{a})$.
In this case, we can moreover define formal gauge transformations,
namely we say $X[A]$ is the {\em formal holomorphic gauge transformation} 
of $A$ by $X$ 
if $X\in \mathrm{GL}(n,\mathbb{C}[\![z]\!])$ and 
{\em formal meromorphic gauge transformation}  if 
$X\in \mathrm{GL}(n,\mathbb{C}(\!(z)\!))$.

For a local differential equation $\frac{d}{dz}Y=AY$ with 
$A\in M(n,\mathbb{C}(\!(z)\!))$, it is known that there exists 
a normal form under the formal meromorphic gauge transformations as 
follows.

Let $\mathcal{P}:=\bigcup_{s\in \mathbb{Z}_{>0}}\mathbb{C}(\!(z^{\frac{1}{s}})\!)$, the 
field of Puiseux series. 

\begin{df}[Hukuhara-Turrittin-Levelt normal form]\normalfont
By {\em Hukuhara-Turrittin-Levelt normal form} or {\em HTL normal form}
for short, we mean an element of the form
\[
	\mathrm{diag}\left(q_{1}(z^{-\frac{1}{s}})I_{n_{1}}+R_{1},\ldots,
	q_{m}(z^{-\frac{1}{s}})I_{n_{m}}+R_{m}
	\right)z^{-1}\in M(n,\mathbb{C}(\!(z^{\frac{1}{s}})\!))\subset 
M(n, \mathcal{P})
\]
where $q_{i}(t)\in t\,\mathbb{C}[t]$ satisfying $q_{i}\neq q_{j}$
if $i\neq j$, and $R_{i}\in M(n_{i},\mathbb{C})$ with $n_1+\cdots n_{m}=n$.
\end{df}
For an HTL normal form $H\in M(n,\mathcal{P})$, we call $H_{\text{irr}}:=
H-\mathrm{pr}_{\text{res}}(H)z^{-1}$ the {\em irregular part} of $H$.
Here we denote the coefficient matrix of $z^{-1}$ in $A\in M(n,\mathcal{P})$
by $\mathrm{pr}_{\text{res}}(A)$.

The following is a fundamental fact of the local formal theory of 
differential equations with irregular singularity.

\begin{thm}[Hukuhara-Turrittin-Levelt, see \cite{W} for instance]
	For any $A\in M(n,\mathbb{C}(\!(z)\!))$, there exist an integer 
	$r\in \mathbb{Z}_{>0}$ and 
	$X\in 
	\mathrm{GL}(n,\mathbb{C}(\!(z^{\frac{1}{r}})\!))$ such that 
	$X[A]$ is an HTL normal form in $M(n,\mathbb{C}(\!(z^{\frac{1}{r}})\!))$.
	We call this $X[A]$ the {\em normal form} of $A$.
\end{thm}

In the above theorem, we may assume the field extension is minimal, i.e.,
\[
	r=\mathrm{min}\left\{
		s\,\middle|\, \text{normal form }X[A]\in
		M(n,\mathbb{C}(\!(z^{\frac{1}{s}})\!))
	\right\}.
\]
If two HTL normal forms $H,H'\in M(n,\mathbb{C}(\!(z^{\frac{1}{r}})\!))$ are 
	normal forms of an $A\in M(n,\mathbb{C}(\!(z)\!))$, then  
	there exists $g\in \mathrm{GL}(n,\mathbb{C})$ such that 
	\begin{align*}
		g^{-1}H_{\text{irr}}g&=H'_{\text{irr}},&
	g^{-1}\mathrm{exp}(2\pi\sqrt{-1}k\,\mathrm{pr}_{\text{res}}(H))g&=
	\mathrm{exp}(2\pi\sqrt{-1}k\,\mathrm{pr}_{\text{res}}(H'))
\end{align*}for some 
	integer $k\ge 1$, see Theorem 6.3 in \cite{BV} for example.

\subsection{Meromorphic connections}
Let us recall the notion of meromorphic connections and 
see their relationship with differential equations.
For $f=\sum_{i>-\infty}^{\infty}a_{i}z^{i}
\in \mathbb{C}(\!(z)\!)$, the {\em order} is  
\[
	\mathrm{ord}(f):=\mathrm{min}\{i\mid a_{i}\neq 0\}.
\]
If $f=0$, we formally put $\mathrm{ord}(f)=\infty$.
For a meromorphic function $f$ locally defined near $a\in \mathbb{P}^{1}$, 
we denote the germ of $f$ at 
$a$ by $f_{a}$. 
We may see $f_{a}\in \mathbb{C}\{\!\{z_{a}\}\!\}\subset 
\mathbb{C}(\!(z_{a})\!)$ by setting $z_{a}=z-a$
if $a\in \mathbb{C}$ and $z_{a}=1/z$ if $a=\infty$ where we take 
$z$ as the standard coordinate of $\mathbb{C}$.
Then define 
\[
	\mathrm{ord}_{a}(f):=\mathrm{ord}(f_{a}).
\]
For a meromorphic 1-form $\omega$ defined on $\mathbb{P}^{1}$, the order 
$\mathrm{ord}_{a}(\omega)$ can be defined as follows. 
Set $U_{1}=\mathbb{P}^{1}\backslash 
\{\infty\}$ and $U_{2}=\mathbb{P}^{1}\backslash\{0\}$. 
Let $z_{i}$ be coordinates of $U_{i}$, $i=1,2$, such that 
$z_{1}(0)=z_{2}(\infty)=0$ and $z_{2}=1/z_{1}$ in $U_{1}\cap U_{2}$.
Then there exist meromorphic functions $f_{i}$ on $U_{i}$ such that 
\[
	\omega =f_{i}\,dz_{i}
\]
on $U_{i}$ for $i=1,2$. Here we note that 
\[
	z_{1}^{2}f_{1}(z_{1})=-f_{2}(1/z_{1})
\]
for  $z_{1}\in U_{1}\cap U_{2}=\mathbb{C}\backslash\{0\}$.
Then define
\[
	\mathrm{ord}_{a}(\omega):=\ord_{a}(f_{i})
\]
for $a\in U_{i}$, $i=1,2$.
 
Let us fix a collection of  points ${a_{0},\ldots,a_{p}}\in \mathbb{P}^1$
 and set
$S:=k_{0}a_{0}+\cdots +k_{p} a_{p}$ as an effective
divisor with $k_0,\ldots,k_{p}>0$.
For $a\in \mathbb{P}^{1}$ 
let $S(a)$ be the coefficient of $a$ in $S$, i.e.,
\[
	S(a):=
	\begin{cases}
	k_{i}&\text{if }a=a_{i}\text{ for }i=0,\ldots,p,\\
	0&\text{otherwise.}
	\end{cases}
\]
For an open set $U\subset \mathbb{P}^{1}$ we define 
$\Omega_{S}(U)$ to be the set of all meromorphic 1-forms $\omega$
on $U$ satisfying $\mathrm{ord}_a(\omega)\ge -S(a)$ for any $a\in U$.
This correspondence defines the sheaf $\Omega_{S}$ by the natural
restriction mappings.

Let $\mathcal{E}$ be a locally free sheaf of rank $n$ on $\mathbb{P}^{1}$,
namely a sheaf of $\mathcal{O}$-modules satisfying 
that for any $a\in \mathbb{P}^{1}$
there exists an open neighbourhood $V\subset \mathbb{P}^{1}$ such that 
$\mathcal{E}|_{V}\cong \mathcal{O}^{n}|_V$.
We may sometimes regard $\mathcal{E}$ as a holomorphic vector bundle over 
$\mathbb{P}^{1}$.
\begin{df}[meromorphic connection]\normalfont
A {\em meromorphic connection} is a pair $(\mathcal{E},\nabla)$ 
of a locally free sheaf $\mathcal{E}$ and a morphism 
$\nabla\colon \mathcal{E}\rightarrow \mathcal{E}\otimes \Omega_{S}$ 
of sheaves of $\mathbb{C}$-vector spaces satisfying
\[
	\nabla(fs)=df\otimes s+f\otimes \nabla(s)
\]
for all $f\in \mathcal{O}(U)$, $s\in \mathcal{E}(U)$
and open subsets $U\subset \mathbb{P}^{1}$.
\end{df}
Let $U\subset \mathbb{P}^{1}$ be an open subset which gives a local 
trivialization of $\mathcal{E}$ and $z$ a local coordinate of $U$.
Then if we fix an identification
$\mathcal{E}|_{U}\cong \mathcal{O}^{n}|_{U}$, 
we can write $\nabla=d-A\,dz$ by $A\in M(n,\mathcal{M}(U))$ on $U$. 
Note that if we write $\nabla=d-A'\,dz$ by another identification
$\mathcal{E}|_{U}\cong \mathcal{O}^{n}|_{U}$, then $A'$ can be obtained by 
a holomorphic gauge transformation of $A$, namely there exists 
$X\in \mathrm{GL}(n,\mathcal{O}(U))$ such that 
\[
	A'=X[A].
\]
Thus we may say that $(\mathcal{E},\nabla)$ defines a holomorphic gauge 
equivalent class of a local differential equation 
\[
	\frac{d}{dz}Y=AY
\]
on $U\subset \mathbb{P}^{1}$.

In particular, suppose that  $\mathcal{E}$ is {\em trivial}, i.e., 
$\mathcal{E}\cong \mathcal{O}^{n}$ and set $U_{1}=\mathbb{P}^{1}
\backslash\{\infty\}$ and $U_{2}=\mathbb{P}^{1}\backslash\{0\}$
as before.
Then if we fix a trivialization $\mathcal{E}\cong\mathcal{O}^{n}$, 
we have
$\nabla=d-A(z_{1})dz_{1}$ on $U_{1}$ with 
$A(z_{1})=(\alpha_{i,j}(z_{1}))_{i,j=1,\ldots,n}\in M(n,\mathbb{C}(z))$ 
satisfying
$\ord_{a}(\alpha_{i,j})\ge -S(a)$ for all $a\in U_{1}$.
Similarly on $U_{2}$ we have $\nabla=d-B(z_{2})dz_{2}$. 
Since $\mathcal{O}^{n}$ is trivial bundle, 
\[
	A(z_{1})dz_{1}=B(z_{2})dz_{2}\text{ on }U_{1}\cap U_{2}.
\]
Namely, 
\[
	B(z_{2})=-\frac{A(1/z_{2})}{z_{2}^{2}}.
\]
This is nothing but the coordinate exchange $t=\frac{1}{z}$ for a differential
equation
\[
	\frac{d}{dz}Y=A(z)Y\longmapsto
	-t^{2}\frac{d}{dt}Y=A(1/t)Y.
\]

Thus a meromorphic connection 
$(\mathcal{O}^n,\nabla)$, $\nabla=d-A\,dz$ on the trivial bundle $\mathcal{O}^{n}$
over $\mathbb{P}^{1}$ corresponds to 
a meromorphic differential equation 
$\frac{d}{dz}Y=AY$ with 
$A=(\alpha_{i,j})_{i,j=1,\ldots,n}\in M(n,\mathbb{C}(z))$ satisfying
$\ord_{a}(\alpha_{i,j})\ge -S(a)$ for all $a\in \mathbb{P}^{1}$,
and vice versa. 
\subsection{Quiver varieties}\label{Quiver}
In this subsection we shall introduce quiver varieties. 
\subsubsection{Complex symplectic manifold and symplectic reduction}
Before seeing the definition of quiver varieties, 
let us recall complex symplectic manifolds.
\begin{df}[complex symplectic manifold]\normalfont
	Let $M$ be an even dimensional complex manifold 
	and $\omega$ a closed nondegenerate holomorphic 2-form on  $M$.
	Then $\omega$ is called a {\em symplectic form} on $M$
	and the pair $(M,\omega)$ is called a {\em complex symplectic 
	manifold}.
\end{df}

\begin{exa}[cotangent bundle of $\mathbb{C}^{n}$]\normalfont
Let us see that the holomorphic cotangent bundle 
$T^{*}\mathbb{C}^{n}$ of
$\mathbb{C}^{n}$ is a complex symplectic manifold.
Since $T^{*}\mathbb{C}^{n}\cong \mathbb{C}^{n}\times (\mathbb{C}^{n})^{*}$,
the coordinate system $(z_{1},\ldots,z_{n})$ of $\mathbb{C}^{n}$ and 
the dual coordinate $(\xi_{1},\ldots,\xi_{n})$ of $(\mathbb{C}^{n})^{*}$
defines a coordinate system of $T^{*}\mathbb{C}^{n}$.
Then the 2-form of $T^{*}\mathbb{C}^{n}$  defined by 
\[
	\omega:=\sum_{i=1}^{n}dz_{i}\wedge d\xi_{i}
\]
is closed since $\omega=-d\alpha$,
\[
	\alpha:=\sum_{i=1}^{n}\xi_{i}\,dz_{i}.
\]
This $\omega$ is called the {\em canonical symplectic form} of 
$T^{*}\mathbb{C}^{n}$.
\end{exa}

Let us define a moment map on $(M,\omega)$. 
Suppose that $M$ has an action
of a complex Lie group $G$. Let $\mathfrak{g}$  be the Lie algebra of $G$ and 
$\mathfrak{g}^*$ the dual vector space of $\mathfrak{g}$.
Let us denote the exponential map by
\[
	\mathrm{exp}\colon \mathfrak{g}\longrightarrow G.
\]
Then any $\xi\in \mathfrak{g}$ defines the holomorphic 
vector field $\xi_{M}\in TM$
by the action of the one parameter subgroup
\[
	m\longmapsto \mathrm{exp}(-t\xi)\cdot m\,\quad m\in M,\,t\in\mathbb{R}.
\]
Then a moment map is defined as follows.
\begin{df}[moment map]\normalfont
	Let $(M,\omega)$ be a complex symplectic manifold with 
	a $G$-action as above. Then a $G$-equivariant map 
	\[
		\mu\colon M\longrightarrow \mathfrak{g}^{*}
	\]
	is called {\em moment map} if it satisfies
	\[
		d\langle \mu(\cdot),\xi\rangle(v)=
		\omega(v,\xi_M)
	\]
	for all $v\in \mathfrak{X}(M)$, vector fields 
	on $M$. Here $\langle\ ,\ \rangle\colon \mathfrak{g}^{*}\times
	\mathfrak{g}\rightarrow \mathbb{C}$ is the canonical 
	pairing and we regard $\langle \mu(\cdot),\xi\rangle\in C^{\infty}(
	M)$.
\end{df}

Let us assume that the action of $G$ is free and proper, namely stabilizers 
in $G$ of $m$ are trivial for all $m\in M$ and the action map
\[
	\begin{array}{ccc}
		G\times M&\longrightarrow &M\times M\\
		(g,m)&\longmapsto&(g\cdot m,m)
	\end{array}
\]
is a proper map.
In this case it is known that the orbit space $M/G$ becomes a manifold.
Let us take $\xi\in (\mathfrak{g}^{*})^{G}$, 
$G$-invariant under the coadjoint action, and consider $\mu^{-1}(\xi)$.
Then $\mu^{-1}(\xi)/G$ becomes a complex manifold and moreover has a 
symplectic form $\underline{\omega}$ defined as follows.
For $\underline{p}\in \mu^{-1}(\xi)/G$ and $\underline{u},\underline{v}
\in T_{\underline{p}}(\mu^{-1}(\xi)/G)$, take $p\in \pi^{-1}(\underline{p})$
, the inverse image of the projection $\pi\colon \mu^{-1}(\xi)\rightarrow 
\mu^{-1}(\xi)/G$, and $u,v\in T_{p}\mu^{-1}(\xi)$ so that 
$(\pi_{*})_{p}(u)=\underline{u}$ and $(\pi_{*})_{p}(v)=\underline{v}$.
Here $\pi_{*}\colon T\mu^{-1}(\xi)\rightarrow T(\mu^{-1}(\xi)/G)$ 
is the differentiation of $\pi$.
Then we define $\underline{\omega}_{\underline{p}}(\underline{u},
\underline{v}):=\omega_{p}(u,v)$. It can be shown that $\underline{\omega}$
is well-defined because $\mu$ is a moment map.

\begin{df}[symplectic reduction, Marsden-Weinstein reduction]\normalfont
	We call the 
	symplectic manifold $(\mu^{-1}(\xi)/G,\underline{\omega})$ 
	a {\em symplectic reduction} or {\em Marsden-Weinstein 
	reduction} of $(M,\omega)$ under the action of $G$.
\end{df}
Sometimes we drop the assumption that the action of $G$ is free and proper,
and call $\mu^{-1}(\xi)/G$ a symplectic reduction too though it may have 
singularities.
\subsubsection{Representations of quiver and quiver variety}
Now let us recall representations of quivers
 
\begin{df}[quiver]
\normalfont
	A \textit{quiver} $\mathsf{Q}
	=(\mathsf{Q}_{0},\mathsf{Q}_{1},s,t)$ is the
	quadruple consisting 
	of $\mathsf{Q}_{0}$, the set of \textit{vertices}, 
	and $\mathsf{Q}_{1}$,
        the set of \textit{arrows} connecting 
	vertices in $\mathsf{Q}_{0}$,
	and two maps $s,t\,\colon \mathsf{Q}_{1}\rightarrow \mathsf{Q}_{0}$,
	which associate to each arrow $\rho\in \mathsf{Q}_{1}$ its 
	{\em source} $s(\rho)\in \mathsf{Q}_{0}$ and 
	its {\em target} $t(\rho)\in \mathsf{Q}_{0}$ respectively.
\end{df}
\begin{df}[representation of quiver]
      \normalfont
      Let $\mathsf{Q}$ be a finite quiver, i.e., 
      $\mathsf{Q}_{0}$ and $\mathsf{Q}_{1}$ are finite sets. 
      A {\em representation} $M$ of $\mathsf{Q}$ is defined by the 
      following data:
      \begin{enumerate}
	      \item To each vertex $a$ in $\mathsf{Q}_{0}$,  a finite 
              dimensional $\mathbb{C}$-vector space $M_{a}$ is attached.
      \item To each arrow $\rho\colon a\rightarrow b$ in $\mathsf{Q}_{1}$,
              a $\mathbb{C}$-linear map 
              $x_{\rho}\colon M_{a}\rightarrow M_{b}$
              is attached.
      \end{enumerate}   
\end{df}      
     We denote the representation by 
      $M=(M_{a},x_{\alpha})_{a\in \mathsf{Q}_{0},\alpha\in \mathsf{Q}_{1}}$.
      The collection of integers defined by $\mathbf{dim\,}M
      =(\mathrm{dim}_{\mathbb{C}}M_{a})_{a\in \mathsf{Q}_{0}}$ 
      is called the {\em dimension vector} of 
      $M$.

For a fixed vector $\alpha\in (\mathbb{Z}_{\ge 0})^{\mathsf{Q}_{0}}$,
the representation space is 
\[
	\mathrm{Rep}(\mathsf{Q},V,\alpha)=\bigoplus_{\rho\in \mathsf{Q}_{1}}
	\mathrm{Hom}_{\mathbb{C}}(V_{s(\rho)},V_{t(\rho)}),
\]
where $V=(V_{a})_{a\in \mathsf{Q}_{0}}$ is a collection of finite dimensional 
$\mathbb{C}$-vector spaces with 
$\mathrm{dim}_{\mathbb{C}}V_{a}=\alpha_{a}$.
If $V_{a}=\mathbb{C}^{\alpha_{a}}$ for all $a\in \mathsf{Q}_{0}$, we simply
write 
$$\mathrm{Rep\,}(\mathsf{Q},\alpha)=
\bigoplus_{\rho\in \mathsf{Q}_{1}}
\mathrm{Hom}_{\mathbb{C}}(\mathbb{C}^{\alpha_{s(\rho)}},
\mathbb{C}^{\alpha_{t(\rho)}}).$$

To each $(x_{\rho})_{\rho\in \mathsf{Q}_{1}}
\in \mathrm{Rep}(\mathsf{Q},V,\alpha)$, the representation 
$(V_{a},x_{\rho})_{a\in \mathsf{Q}_{0},\rho\in \mathsf{Q}_{1}}$ associates.
Thus we may identify $(x_{\rho})_{\rho\in \mathsf{Q}_{1}}$ with
$(V_{a},x_{\rho})_{a\in \mathsf{Q}_{0},\rho\in \mathsf{Q}_{1}}$.
  
The representation space $\mathrm{Rep}(\mathsf{Q},V,\alpha)$ has 
an action of $\prod_{a\in \mathsf{Q}_{0}}\mathrm{GL}(V_{a})$.
For $(x_{\rho})_{\rho\in \mathsf{Q}_{1}}\in \mathrm{Rep\,}(\mathsf{Q},V,\alpha)$ and 
$g=(g_{a})\in \prod_{a\in \mathsf{Q}_{0}}\mathrm{GL}(V_{a})$,
then 
$g\cdot (x_{\rho})_{\rho\in \mathsf{Q}_{1}}\in \mathrm{Rep\,}(\mathsf{Q},V,\alpha)$ 
consists of 
$x'_{\rho}=g_{t(\rho)}x_{\rho}g_{s(\rho)}^{-1}
\in \mathrm{Hom}_{\mathbb{C}}(V_{s(\rho)},V_{t(\rho)})$.
   
Let  
$M=(M_{a},x^{M}_{\rho})_{a\in \mathsf{Q}_{0},\rho\in \mathsf{Q}_{1}}$
and 
$N=(N_{a},x^{N}_{\rho})_{a\in \mathsf{Q}_{0},\rho\in \mathsf{Q}_{1}}$ 
be representations of a quiver $\mathsf{Q}$.
Then $N$ is called the {\em subrepresentation} of M
if we have the following:
\begin{enumerate}
          \item There exists a direct sum decomposition
		  $M_{a}=N_{a}\oplus N'_{a}$ for each $a\in \mathsf{Q}_{0}$.
	  \item For each $\rho\colon a\rightarrow b\in \mathsf{Q}_{1}$,
              the equality $x^{M}_{\rho}|_{N_{a}}
              =x^{N}_{\rho}$ holds.
\end{enumerate}
In this case we denote $N\subset M$.
Moreover if  
\begin{itemize}
	      \item[3.] 
		      for each $\rho\colon a\rightarrow b\in \mathsf{Q}_{1}$,
              we have $x^{M}_{\rho}|_{N'_{a}}\subset N_{b}^{'}$,
\end{itemize}
then we denote $M=N\oplus N'$ as {\em direct sum} 
 where 
$N'
=(N'_{a},x^{M}_{\rho}|_{N_{a}^{'}})_{a\in \mathsf{Q}_{0},\rho\in \mathsf{Q}_{1}}$.

The representation $M$ is said to be {\em irreducible} if 
$M$ has no subrepresentations other than $M$ and $\{0\}$. 
Here $\{0\}$ is the representation of $\mathsf{Q}$ 
which consists of zero vector spaces and zero
linear maps.
On the other hand if any direct sum 
decomposition  $M=N\oplus N'$
satisfies either  $N=\{0\}$ or $N'=\{0\}$, then
$M$ is said to be {\em indecomposable}.

Let us recall the double of 
a quiver $\mathsf{Q}$.
\begin{df}[double quiver]
\normalfont
	Let $\mathsf{Q}=(\mathsf{Q}_{0},\mathsf{Q}_{1})$ be a finite quiver. 
	Then the {\em double quiver} $\overline{\mathsf{Q}}$ of $\mathsf{Q}$ is the 
        quiver obtained by 
        adjoining the reverse arrow $\rho^{*}\colon b\rightarrow a$ 
        to each arrow $\rho\colon a\rightarrow b$. Namely
	$\overline{\mathsf{Q}}:=(\overline{\mathsf{Q}}_{0}:=
	\mathsf{Q}_{0},\overline{\mathsf{Q}}_{1}:=
	\mathsf{Q}_{1}\cup \mathsf{Q}_{1}^{*})$ where 
	$\mathsf{Q}_{1}^{*}:=\{\rho^{*}\colon t(\rho)\rightarrow s(\rho)\mid 
	\rho \in \mathsf{Q}_{1}\}$. 
\end{df}

Let us note that for each $\rho \in \mathsf{Q}_{1}$ we can identify 
\[
	\mathrm{Hom}_{\mathbb{C}}(\mathbb{C}^{\alpha_{s(\rho)}},\mathbb{C}^{
\alpha_{t(\rho)}})^{*}\cong 
\mathrm{Hom}_{\mathbb{C}}(\mathbb{C}^{\alpha_{s(\rho^{*})}},\mathbb{C}^{
\alpha_{t(\rho^{*})}})
\]
by the trace pairing.
Thus the representation space $\mathrm{Rep}(\overline{\mathsf{Q}},\alpha)$ can be 
identified with the cotangent bundle
\[
	T^{*}\mathrm{Rep}(\mathsf{Q},\alpha)\cong \mathrm{Rep}(
	\overline{\mathsf{Q}},\alpha).
\]
In this case the canonical symplectic form is given by 
\[
	\omega(x,y)=\sum_{\rho\in \mathsf{Q}_{1}}(
	\mathrm{tr}(x_{\rho}y_{\rho^{*}})-\mathrm{tr}(x_{\rho^{*}}y_{\rho})).
\]
Thus we can see $\mathrm{Rep(\overline{Q},\alpha)}$ as a complex symplectic
manifold with the action of 
\[
	\mathbf{G}:=\prod_{a\in \mathsf{Q}_{0}}
	\mathrm{GL}(\alpha_{a},\mathbb{C}).
\]
Then the following map  is a moment map; 
\[
\mu_{\alpha}\colon \mathrm{Rep}(\overline{\mathsf{Q}},\alpha)\rightarrow 
\prod_{a\in \mathsf{Q}_{0}}M(\alpha_{a},\mathbb{C})
\]
whose images 
$(\mu_{\alpha}(x)_{a})_{a\in \mathsf{Q}_{0}}$ are given by
\[
	\mu_{\alpha}(x)_{a}=\sum_{
	\substack{\rho\in \mathsf{Q}_{1}\\ t(\rho)=a}}x_{\rho}
        x_{\rho^{*}}-
        \sum_{
	\substack{\rho\in \mathsf{Q}_{1}\\s(\rho)=a}}x_{\rho^{*}}
        x_{\rho}.
\]      
Now we are ready to define quiver varieties.
\begin{df}[quiver variety]\normalfont
Let us take
a collection of complex numbers
$\lambda=(\lambda_{a})\in \mathbb{C}^{\mathsf{Q}_{0}}$. 
Then a {\em quiver variety} is the affine quotient 
\[
	\mathfrak{M}_{\lambda}(\mathsf{Q},\alpha):=
	\mu^{-1}(\lambda)/\!/\mathbf{G}:=
	\mathrm{Specm\,}\mathbb{C}[\mu^{-1}(\lambda)]^{\mathbf{G}}.
\]
\end{df}
Here
$\mathbb{C}[\mu^{-1}(\lambda)]$ is 
the coordinate ring of $\mu^{-1}(\lambda)$. 

Since this variety might have singularities, 
we moreover consider the regular part
of this variety defined as follows.
\begin{df}\normalfont
We say that $x\in \mathrm{Rep}(\overline{\mathsf{Q}},\alpha)$ 
is {\em stable} if 
\begin{enumerate}
	\item the orbit $\mathbf{G}\cdot x$
is closed,
\item  stabilizer of $x$ in 
$\mathbf{G}/\mathbb{C}^{\times}$ 
is finite.
\end{enumerate}
\end{df}
Here we note that $\mathbb{C}^{\times}
\subset \mathbf{G}$
acts trivially on $\mathrm{Rep}(\overline{\mathsf{Q}},\alpha)$.

It is known that the stability of $x$ assures that the morphism
\[
	\begin{array}{lccc}
		\sigma_{x}\colon& 
		\mathbf{G}&
		\longrightarrow &
		\mathrm{Rep}(\overline{\mathsf{Q}},\alpha)\\
		&g&\longmapsto&g\cdot x
	\end{array}
\]
is proper.

In our case moreover the stability 
can be rephrased by the irreducibility of representations.
\begin{thm}[King \cite{Kin}]
	$x\in \mathrm{Rep}(\overline{\mathsf{Q}},\alpha)$ is stable 
	if and only if $x$ is an irreducible representation.
\end{thm}

Thus let us consider the (possibly empty) subspace
\[
	\mu^{-1}(\lambda)^{\text{irr}}:=\{
	x\in \mu^{-1}(\lambda)\mid x\text{ is irreducible}\}.
\]
Then the action of $\mathbf{G}/\mathbb{C}^{\times}$ on this space 
is proper and moreover free (see King \cite{Kin}). Thus the symplectic reduction
\[
	\mathfrak{M}^{\text{reg}}_{\lambda}(\mathsf{Q},\alpha):=
	\mu^{-1}(\lambda)^{\text{irr}}/\mathbf{G}
\]
can be seen as a complex manifold with the symplectic structure, i.e.,
a complex symplectic manifold.
We call this manifold the quiver variety too.
\begin{rem}\normalfont
	The above quiver varieties are special ones of Nakajima quiver 
	varieties which enjoy rich geometric properties and applications 
	for representation theory and theoretical physics and so on (see 
	\cite{Nak} for instance).
\end{rem}
\subsubsection{Some geometry of quiver varieties}\label{some geometry}
As we noted before, the complex symplectic manifold 
$\mathfrak{M}^{\text{reg}}_{\lambda}(\mathsf{Q},\alpha)$ is possibly 
empty. 
Thus next we see a necessary and sufficient condition for the non-emptiness
of $\mathfrak{M}^{\text{reg}}_{\lambda}(\mathsf{Q},\alpha)$ obtained by
Crawley-Boevey in \cite{C1}.

In order to explain the condition,
recall the root system of a quiver $\mathsf{Q}$
(cf. \cite{Kac}).
Let $\mathsf{Q}$ be a finite quiver.  
From the  {\em Euler form}  $$\langle \alpha,\beta\rangle 
:=\sum_{a\in \mathsf{Q}_{0}}\alpha_{a}\beta_{a}-\sum_{\rho\in \mathsf{Q}_{1}}
\alpha_{s(\rho)}\beta_{t(\rho)},$$  
a symmetric bilinear form and quadratic form are  defined by 
\begin{align*}(\alpha,\beta)&:=\langle \alpha,\beta\rangle
+\langle \beta,\alpha\rangle,\\
	      q(\alpha)
	      &:=\frac{1}{2}(\alpha,\alpha)
\end{align*}      
	      and  set
      $p(\alpha):=1-q(\alpha)$. 
      Here $\alpha,\beta\in \mathbb{Z}^{\mathsf{Q}_{0}}$.

For each vertex $a\in \mathsf{Q}_{0}$, define 
 $\epsilon_{a}\in \mathbb{Z}^{\mathsf{Q}_{0}}$ ($a\in \mathsf{Q}_{0}$) 
  so that $(\epsilon_{a})_a=1$, $(\epsilon_{a})_{b}=0$, 
$(b\in \mathsf{Q}_{0}\backslash\{a\})$.
We call $\epsilon_{a}$ a {\em fundamental root} if 
the vertex $a$ has no edge-loop, i.e., there is no arrow $\rho$ such that 
$s(\rho)=t(\rho)=a$.
Denote by $\Pi$ the set of fundamental roots.
For a fundamental root $\epsilon_{a}$, define the {\em fundamental 
reflection} $s_{a}$ by 
\[
	s_{a}(\alpha):=\alpha-(\alpha,\epsilon_{a})\epsilon_{a}
	\text{ for }\alpha\in \mathbb{Z}^{\mathsf{Q}_{0}}.
\]
The group $W\subset \mathrm{Aut\,}\mathbb{Z}^{\mathsf{Q}_{0}}$ generated
by all fundamental reflections is called the {\em Weyl group} of 
the quiver $\mathsf{Q}$. Note that the bilinear form $(\,,\,)$ is 
$W$-invariant. 
Similarly we can define the reflection $r_{a}\colon 
\mathbb{C}^{\mathsf{Q}_{0}}\rightarrow
\mathbb{C}^{\mathsf{Q}_{0}}$ by 
\[r_{a}(\lambda)_{b}:=
\lambda_{b}-(\epsilon_{a},\epsilon_{b})\lambda_{a}
\]
for $\lambda\in \mathbb{C}^{\mathsf{Q}_{0}}$ and  $a,b\in \mathsf{Q}_{0}$.
Define the set of {\em real roots} by 
\[
	\Delta^{\text{re}}:=\bigcup_{w\in W}w(\Pi).
\]

For an element $\alpha=(\alpha_{a})_{a\in \mathsf{Q}_{0}}
\in \mathbb{Z}^{\mathsf{Q}_{0}}$ the {\em support} of $\alpha$
is the set of $\epsilon_{a}$ such that $\alpha_{a}\neq 0$, and 
denoted by $\mathrm{supp\,}(\alpha)$.
We say the support of $\alpha$ is {\em connected} if 
the subquiver consisting of 
the set of vertices $a$ satisfying 
$\epsilon_{a}\in \mathrm{supp\,}(\alpha)$ and 
all arrows joining these vertices, is connected.
Define the {\em fundamental set} $F\subset \mathbb{Z}^{\mathsf{Q}_{0}}$ by
\[
	F:=
	\left\{\alpha\in (\mathbb{Z}_{\ge 0})^{\mathsf{Q}_{0}}\backslash\{0\}
        \mid (\alpha,\epsilon)\le 0\text{ for all }\epsilon\in\Pi,\,
        \text{support of }\alpha\text{ is connected}
        \right\}.
\]
Then define the set of {\em imaginary roots} by
\[
	\Delta^{\text{im}}:=\bigcup_{w\in W}w(F\cup -F).
\]
Then the {\em root system} is
\[
	\Delta:=\Delta^{\text{re}}\cup\Delta^{\text{im}}.
\]
An element $\Delta^{+}:=
\alpha\in \Delta\cap(\mathbb{Z}_{\ge 0})^{\mathsf{Q}_{0}}$
is called a {\em positive root}.
    
Now we are ready to see Crawley-Boevey's theorem.
For a fixed $\lambda=(\lambda_{a})\in \mathbb{C}^{\mathsf{Q}_{0}}$,
the set $\Sigma_{\lambda}$ consists of 
the positive roots satisfying 
\begin{enumerate}
              \item $\lambda\cdot \alpha:=
		      \sum_{a\in \mathsf{Q}_{0}}\lambda_{a}\alpha_{a}=0$,
              \item if there exists a decomposition 
                  $\alpha=\beta_{1}+\beta_{2}+\cdots,$
                  with $\beta_{i}\in \Delta^{+}$ and 
                  $\lambda\cdot \beta_{i}=0$, 
                  then 
                  $p(\alpha)>p(\beta_{1})+p(\beta_{2})+\cdots.$
\end{enumerate}

\begin{thm}[Crawley-Boevey. Theorem 1.2 in \cite{C1}]\label{CB}
	Let $\mathsf{Q}$ be a finite quiver and $\overline{\mathsf{Q}}$ the double of $\mathsf{Q}$.
        Let us fix a dimension vector 
	$\alpha\in (\mathbb{Z}_{\ge 0})^{\mathsf{Q}_{0}}$ and 
	$\lambda\in \mathbb{C}^{\mathsf{Q}_{0}}.
        $
	Then  $\mu^{-1}(\lambda)^{\text{irr}}\subset 
	\mathrm{Rep\,}(\overline{\mathsf{Q}},\alpha)$
	is nonempty 
        if and only if 
	$\alpha\in \Sigma_{\lambda}$.
	Furthermore, in this case $\mu^{-1}(\lambda)$ is 
	an irreducible algebraic variety and $\mu^{-1}(\lambda)^{\text{irr}}$ 
	is dense in $\mu^{-1}(\lambda)$.
\end{thm}
Moreover Crawley-Boevey showed the following geometric properties of 
quiver varieties.
\begin{thm}[Crawley-Boevey Corollary 1.4 in \cite{C1}]\label{CB2}
	If $\alpha\in \Sigma_{\lambda}$ then
	the quiver variety 
	$\mathfrak{M}_{\lambda}(\mathsf{Q},\alpha)$
	is a reduced and irreducible variety of dimension $2p(\alpha)$.
\end{thm}

Combining these results, we have the following non-emptiness condition
of regular parts of quiver varieties.
\begin{cor}[Crawley-Boevey \cite{C1}]
	The quiver variety 
	$\mathfrak{M}^{\text{reg}}_{\lambda}(\mathsf{Q},\alpha)$
	is non-empty if and only if $\alpha\in\Sigma_{\lambda}$.
	Furthermore in this case, it is a connected complex symplectic manifold
	of dimension $2p(\alpha)$.
\end{cor}
\subsection{Moduli spaces of stable meromorphic connections on trivial bundles}
\label{moduli}
Let us define moduli spaces of meromorphic connections on trivial bundles 
following Boalch's paper \cite{Boa1} (see also \cite{HY}).

Let
\[
	B=\mathrm{diag}(q_{1}(z^{-1})I_{n_{1}}+R_{1}z^{-1},\cdots
	q_{m}(z^{-1})I_{n_{m}}+R_{m}z^{-1})\in 
	\mathrm{GL}(n,\mathbb{C}(\!(z)\!))
\]
be an HTL normal form. The equivalent class of $B$ under formal holomorphic 
gauge transformations is 
\[
	\mathcal{O}_{B}:=\left\{X[B]\mid X\in \mathrm{GL}(n,\mathbb{C}[\![z]\!])
	\right\}.
\]
Let us consider another equivalent class of $B$ called the {\em truncated orbit}
of $B$. Let us consider the  projection 
 $$\iota\colon M(n,\mathbb{C}(\!(z)\!))\longrightarrow M(n,\mathbb{C}(\!(z)\!)/
 \mathbb{C}[\![z]\!]).$$
The map $\iota$ induces the action of $\mathrm{GL}(n,\mathbb{C}[\![z]\!])$ on 
$M(n,\mathbb{C}(\!(z)\!)/
\mathbb{C}[\![z]\!])$
from the adjoint action of that on  $M(n,\mathbb{C}(\!(z)\!))$.
 Namely 
for  $g\in \mathrm{GL}(n,\mathbb{C}[\![z]\!])$, $Z\in M(n,\mathbb{C}(\!(z)\!)/
\mathbb{C}[\![z]\!])$
define $g^{-1}Zg:=\iota (g\tilde{Z}g^{-1})$ where $\tilde{Z}\in M(n,\mathbb{C}(\!(z)\!))$
is chosen so that $\iota (\tilde{Z})=Z$.
We can see that this is independent of the choice of $\tilde{Z}$.

Then regarding $B$ as an element in $M(n,\mathbb{C}(\!(z)\!)/
\mathbb{C}[\![z]\!])$, we define the truncated orbit of $B$ by the action of 
$\mathrm{GL}(n,\mathbb{C}[\![z]\!])$ on $M(n,\mathbb{C}(\!(z)\!)/
\mathbb{C}[\![z]\!])$;
\[
	\mathcal{O}_{B}^{\text{\text{tru}}}:=\left\{g^{-1}Bg\in 
	M(n,\mathbb{C}(\!(z)\!)/
\mathbb{C}[\![z]\!])\,\middle|\, g\in \mathrm{GL}(n,\mathbb{C}[\![z]\!])
	\right\}.
\]

Let us note that 
\begin{align*}
	\iota(X[B])&=\iota(XBX^{-1})
	+\iota\left(\frac{d}{dz}X\cdot X^{-1}\right)\\
	&=\iota(XBX^{-1})\\
	&=XBX^{-1}.
\end{align*}
Thus $\iota$ gives the well-defined map
\[
	\iota\colon \mathcal{O}_{B}\longrightarrow \mathcal{O}^{\text{tru}}_{B}.
\]
Under a good situation, we can moreover show the following.
\begin{prop}
	If eigenvalues of $R_{i}$ never differ by any integer for each 
	$i=1,\ldots,m$, then 
	\[
		\iota^{-1}(\mathcal{O}^{\text{tru}}_{B})=\mathcal{O}_{B}
	\]
\end{prop}
\begin{proof}
	The Frobenius method (see \cite{W}) shows that 
	if 
	for $A\in M(n,\mathbb{C}(\!(z)\!))$ there exists $X\in 
	\mathrm{GL}(n,\mathbb{C}[\![z]\!])$ such that 
	$\iota(XAX^{-1})=B$, then there exists $X'\in 
	\mathrm{GL}(n,\mathbb{C}[\![z]\!])$ such that 
	$X'[A]=B$ under the assumption of this proposition.
\end{proof}

Let us consider a meromorphic connection $(\mathcal{O}^{n},\nabla)$ on
the trivial bundle over $\mathbb{P}^{1}$.
We write 
$\nabla_{a}\in \mathcal{O}_{B}\ (\text{resp. } \mathcal{O}_{B}^{\text{tru}})$ 
for $a\in \mathbb{P}^{1}$ if there exists $A_{a}\in M(n,\mathbb{C}(\!(z_{a})\!))$
such that 
$\nabla=d-A_{a}\,dz_{a}$ near $a$ and 
 $A_a\in \mathcal{O}_{B}\ (\text{resp. } \iota(A_a)\in \mathcal{O}^{\text{tru}}_{B})$. 
 Here $z_{a}:=\begin{cases}z-a&\text{if }a\in \mathbb{C}\\
	 1/z&\text{if }a=\infty
 \end{cases}$
 with the standard coordinate $z$ of $\mathbb{C}$.
Let $S=k_{0}a_{0}+\ldots+k_{p}a_{p}$ be an effective  divisor on $\mathbb{P}^{1}$
as before.
Define a set of meromorphic connections on $\mathbb{P}^{1}$ 
\[
	\mathrm{Triv}^{(n)}_{S}:=\left\{
		(\mathcal{O}^{n},\nabla) \middle| 
		\begin{array}{c}
			\nabla\colon \mathcal{O}^{n}\rightarrow 
			\mathcal{O}^{n}\otimes \Omega_{S}
			\end{array}
	\right\}.
\]
We say $(\mathcal{O}^{n},\nabla)\in \mathrm{Triv}^{(n)}_{S}$ 
is {\em stable} if there exists no nontrivial proper
subspace $W\subset \mathbb{C}^{n}$ 
such that the subbundle $\mathcal{W}:=W\otimes \mathcal{O}\subset \mathbb{C}^{n}\otimes 
\mathcal{O}=\mathcal{O}^{n}$ is closed under $\nabla$, i.e.,
\[
	\nabla(\mathcal{W})\subset \mathcal{W}\otimes \Omega_{S}.
\]

Let $\mathbf{B}=(B_{0},\ldots,B_{p})\in M(n,\mathbb{C}(\!(z)\!))^{p+1}$ 
be a collection of 
HTL normal forms satisfying $\mathrm{ord}(B_{i})=k_{i}$ for all $i=0,\ldots,p$.
Then the moduli space of stable meromorphic connections on trivial bundles
is 
\[
	\mathfrak{M}(\mathbf{B}):=
	\left\{
		(\mathcal{O}^{n},\nabla)\in 
		\mathrm{Triv}^{(n)}_{S}\,\middle|\,
		\begin{array}{c}
			(\mathcal{O}^{n},\nabla)\colon\text{ stable},\\
		\nabla_{a_{i}}\in \mathcal{O}^{\text{tru}}_{B_{i}}
		\text{ for all }i=0,\ldots,p
		\end{array}
	\right\}\big/\mathrm{GL}(n,\mathbb{C}).
\]
Here $\mathrm{GL}(n,\mathbb{C})=\mathrm{GL}(n,\mathcal{O}(\mathbb{P}^{1}))$
acts on $\text{Triv}^{(n)}_{S}$ as holomorphic gauge transformations.

A M\"obius transformation allows us to set $a_{0}=\infty\in \mathbb{P}^{1}$.
Then we can 
identify $(\mathcal{O}^{n},\nabla)\in \mathrm{Triv}^{n}_{S}$ with 
a meromorphic differential equation defined on $\mathbb{P}^{1}$,
\[
	\frac{d}{dz}Y=\left(\sum_{i=1}^{p}\sum_{\nu=1}^{k_{i}}
		\frac{A^{(i)}_{\nu}}{(z-a_{i})^{\nu}}
	-\sum_{2\le\nu\le k_{0}}A^{(0)}_{\nu}z^{\nu-2}
\right)Y.
\]
The stability of $(\mathcal{O}^{n},\nabla)$ corresponds to the irreducibility
of the differential equation, namely, we say the above differential equation 
is {\em irreducible} if 
there is no proper subspace of $\mathbb{C}^{n}$ other than $\{0\}$ 
which is invariant under all $A^{(i)}_{\nu}$, $i=0,\ldots,p$, $
\nu=1,\ldots,k_{i}$.
Here we set 
\[
	A_{1}^{(0)}:=-\sum_{i=1}^{p}A^{(i)}_{1}.
\]
Thus we can regard $\mathfrak{M}(\mathbf{B})$ as a moduli space of 
meromorphic differential equations on $\mathbb{P}^{1}$,
\begin{multline*}
	\mathfrak{M}(\mathbf{B})=\\
	\left\{\frac{d}{dz}Y=\left(\sum_{i=1}^{p}\sum_{\nu=1}^{k_{i}}
		\frac{A^{(i)}_{\nu}}{(z-a_{i})^{\nu}}
	-\sum_{2\le\nu\le k_{0}}A^{(0)}_{\nu}z^{\nu-2}
\right)Y\,\middle|\,\begin{array}{c}\text{irreducible},\\
 \sum_{\nu=1}^{k_{i}}\frac{A^{(i)}_{\nu}}{z^{\nu}}\in 
\mathcal{O}^{\text{tru}}_{B_{i}},\\ i=0,\ldots,p
\end{array}
	\right\}\\\bigg/ \mathrm{GL}(n,\mathbb{C}).
\end{multline*}

If we take another effective divisor $S'=k_{0}b_{0}+\cdots+k_{p}b_{p}$
with the same $k_{i}$ as $S$,
then we can identify $\mathrm{Triv}_{S}^{(n)}$
and $\mathrm{Triv}_{S'}^{(n)}$ as follows. We may assume $b_{0}=\infty$
by applying a M\"obius transformation if necessary. Then replacing $a_{i}$ with $b_{i}$ 
of an element in $\mathrm{Triv}^{(n)}_{S}$ for $i=1,\ldots,p$, we obtain an element in $\mathrm{Triv}_{S'}^{(n)}$.

Thus we may regard
\[
	\mathfrak{M}(\mathbf{B})=
	\left\{
		\mathbf{A}=(A^{(i)}(z))_{0\le i\le p}
		\in \prod_{i=0}^{p}\mathcal{O}_{B^{(i)}}^{\text{tru}} 
		\middle|\ 
\begin{array}{c}\mathbf{A}\text{ is irreducible }, \\
	\sum_{i=0}^{p}\mathrm{pr}_{\text{res}}(A^{(i)}(z))=0
\end{array}
\right\}\bigg/ \mathrm{GL}(n,\mathbb{C})
\]
which is free from locations of $a_{i}$ in $\mathbb{P}^{1}$.

\subsection{Moduli spaces of connections and
quiver varieties}
We shall give a realization of the moduli space 
$\mathfrak{M}(\mathbf{B})$ as a quiver variety.
Let us suppose that $B^{(0)},\ldots,B^{(p)}$ are written by
\[
    B^{(i)}=
    \mathrm{diag}\left(
    q^{(i)}_{1}(z^{-1})I_{n^{(i)}_{1}}+R^{(i)}_{1}z^{-1},\ldots,
    q^{(i)}_{m^{(i)}}(z^{-1})I_{n^{(i)}_{m^{(i)}}}+R^{(i)}_{m^{(i)}}z^{-1}
    \right)
\]
and choose complex numbers 
$\xi^{[i,j]}_{1},\ldots,\xi^{[i,j]}_{e_{[i,j]}}$
so that 
\[
    \prod_{k=1}^{e_{[i,j]}}(R^{(i)}_{j}-\xi^{[i,j]}_{k})=0
\]
for $i=0,\ldots,p$ and $j=1,\ldots,m^{(i)}$.
Set 
\[
	k_{i}:=-\mathrm{max}_{j=1,\ldots,m^{(i)}}\{\mathrm{ord}(q_{j}^{(i)}
	(z^{-1}))\}
\]
for each $i=0,\ldots,p$.
Set 
\[
	I_{\text{irr}}:=\{i\in\{0,\ldots,p\}\mid m^{(i)}>1\}\cup \{0\}
\]
and 
\[
	I_{\text{reg}}:=\{0,\ldots,p\}\backslash I_{\text{irr}}.
\]
Here 
$I_{\text{irr}}$ may be seen as the set of irregular singular points and $\infty$,
and $I_{\text{reg}}$ of  
regular singular points other than $\infty$.

Then let us define a quiver $\mathsf{Q}$ as follows. Set
\begin{align*}
	\mathsf{Q}_{0}^{\text{irr}}&:=\left\{[i,j]\,\middle|\,
		\begin{array}{l}
		i\in I_{\text{irr}},\\
		j=1,\ldots,m^{(i)}
\end{array}\right\},&
	\mathsf{Q}_{0}^{\text{leg}}&:=\left\{
			[i,j,k]\,\middle|\,
			\begin{array}{l}
			i=0,\ldots,p,\\
			j=1,\ldots,m^{(i)},\\
			k=1,\ldots,e_{[i,j]}-1
			\end{array}
		\right\}.\\
\end{align*}
Then the set of vertices of $\mathsf{Q}$ is the disjoint union
\[
	\mathsf{Q}_{0}:=\mathsf{Q}_{0}^{\text{irr}}\sqcup
	\mathsf{Q}_{0}^{\text{leg}}.
\]
Also set
\begin{align*}
	\mathsf{Q}_{1}^{0\to I_{\text{irr}}}&:=\left\{
		\rho^{[0,j]}_{[i,j']}\colon
		[0,j]\rightarrow [i,j']\,\middle|\,
		\begin{array}{l}
			j=1,\ldots,m^{(0)},\\
			i\in I_{\text{irr}}\backslash\{0\},\\
			j'=1,\ldots,m^{(i)}
		\end{array}
	\right\},\\
	\mathsf{Q}_{1}^{B^{(i)}}&:=\left\{
		\rho^{[k]}_{[i,j],[i,j']}\colon
		[i,j]\rightarrow [i,j']\,\middle|\,
		\begin{array}{l}
			1\le j<j'\le m^{(i)},\\
			1\le k\le d_{i}(j,j')
		\end{array}
	\right\},\\
	\mathsf{Q}_{1}^{\text{leg}^{(i)}}&:=
	\left\{
		\rho_{[i,j,k]}\colon [i,j,k]\rightarrow
		[i,j,k-1]\,\middle|\,
		\begin{array}{l}
			j=1,\ldots,m^{(i)},\\
			k=2,\ldots,e_{[i,j]}-1
		\end{array}
	\right\},\\
	\mathsf{Q}_{1}^{\text{leg}^{(i)}\to B^{(i)}}&:=\left\{
		\rho_{[i,j,1]}\colon [i,j,1]\rightarrow [i,j]\mid
		j=1,\ldots,m^{(i)}
	\right\},\\
	\mathsf{Q}^{\text{leg}^{(i)}\to 0}_{1}&:=\left\{
		\rho^{[i,1,1]}_{[0,j]}\colon
		[i,1,1]\rightarrow [0,j]\mid 
		i\in I_{\text{reg}},\,j=1,\ldots,m^{(0)}
	\right\}.
\end{align*}
Here $d_{i}(j,j'):=\mathrm{deg\,}_{\mathbb{C}[z]}(q^{(i)}_{j}(z)
-q^{(i)}_{j'}(z))-2$.

Then the set of arrows of $\mathsf{Q}$ is the disjoint union
\[
	\mathsf{Q}_{1}:=
	\mathsf{Q}_{1}^{0\to I_{\text{irr}}}
	\sqcup
	\bigsqcup_{i\in I_{\text{irr}}}\left(
		\mathsf{Q}_{1}^{B^{(i)}}\sqcup
		\mathsf{Q}_{1}^{\text{leg}^{(i)}\to B^{(i)}}\sqcup
		\mathsf{Q}_{1}^{\text{leg}^{(i)}}
	\right)
	\sqcup
	\bigsqcup_{i\in I_{\text{reg}}}\left(
		\mathsf{Q}_{1}^{\text{leg}^{(i)}\to 0}\sqcup
		\mathsf{Q}_{1}^{\text{leg}^{(i)}}
	\right).
\]

\begin{exa}\normalfont
	Let us consider the following $\mathbf{B}=(B^{(0)},B^{(1)},B^{(2)})$.
	\begin{align*}
		B^{(0)}&=
		\begin{pmatrix}
			a^{(0)}_{4}&&&\\
							       &a^{(0)}_{4}&&\\
							       &&a^{(0)}_{4}&\\
							&&&b^{(0)}_{4}
		\end{pmatrix}z^{-4}
		+
		\begin{pmatrix}
			a^{(0)}_{3}&&&\\
						 &a^{(0)}_{3}&&\\
			    &&b^{(0)}_{3}&\\
			    &&&c^{(0)}_{3}
		\end{pmatrix}z^{-3}
		\\
		&+
		\begin{pmatrix}
			a^{(0)}_{2}&&&\\
							       &b^{(0)}_{2}&&\\
							       &&c^{(0)}_{2}&\\
							       &&&d^{(0)}_{2}
		\end{pmatrix}z^{-2}
		+
		\begin{pmatrix}
			\xi^{[0,1]}_{1}&&&\\
			&\xi^{[0,2]}_{1}&&\\
		 &&\xi^{[0,3]}_{1}&\\
			&&&\xi^{[0,4]}_{1}
		\end{pmatrix}z^{-1},\\
		B^{(1)}&=
		\begin{pmatrix}
			a^{(1)}_{2}&&&\\
							       &a^{(1)}_{2}&&\\
							       &&a^{(1)}_{2}&\\
							       &&&b^{(1)}_{2}
		\end{pmatrix}
		z^{-2}+
		\begin{pmatrix}
			\xi^{[1,1]}_{1}&&&\\
				       &\xi^{[1,1]}_{2}&&\\
				       &&\xi^{[1,1]}_{3}&\\
				       &&&\xi^{[1,2]}_{1}
		\end{pmatrix}z^{-1},\\
		B^{(2)}&=
		\begin{pmatrix}
			\xi^{[2,1]}_{1}&&&\\
				       &\xi^{[2,1]}_{2}&&\\
				       &&\xi^{[2,1]}_{3}&\\
				       &&&\xi^{[2,1]}_{4}
		\end{pmatrix}z^{-1}.
	\end{align*}
	Here any distinct 
	two of $\{a^{(i)}_{j},b^{(i)}_{j},c^{(i)}_{j},d^{(i)}_{j}\}$
	stand for distinct complex numbers and
	$\xi^{[i,j]}_{k}\neq \xi^{[i,j]}_{k'}$ if $k\neq k'$. 

	Then we can associate the following quiver to this $\mathbf{B}$.
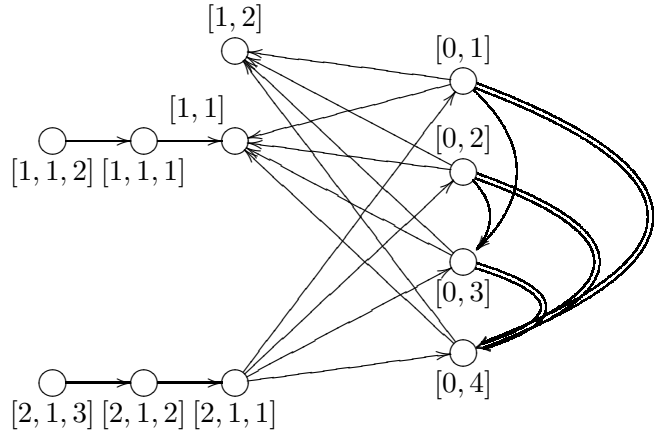
	\[
	\begin{xy}
		(0,12) *++!U{[0,3]}*\cir<5pt>{}="A",
		(0,0) *++!U{[0,4]}*\cir<5pt>{}="B",
		(0,36)*++!D{[0,1]}*\cir<5pt>{}="C",
		(0,24)*++!D{[0,2]}*\cir<5pt>{}="D",
		(-30,28)*+!R+!D{[1,1]}*\cir<5pt>{}="E",
		(-30,40)*++!D{[1,2]}*\cir<5pt>{}="F",
		(-42,28)*++!U{[1,1,1]}*\cir<5pt>{}="G",
		(-54,28)*++!U{[1,1,2]}*\cir<5pt>{}="H",
		(-30,-4)*++!U{[2,1,1]}*\cir<5pt>{}="I",
		(-42,-4)*++!U{[2,1,2]}*\cir<5pt>{}="J",
		(-54,-4)*++!U{[2,1,3]}*\cir<5pt>{}="K",
		\ar@/^50pt/@{=>}"D";"B",	
		\ar@/^70pt/@{=>}"C";"B",
		\ar@/^30pt/@{=>}"A";"B",
		\ar@/^20pt/@{->}"C";"A",
		\ar@/^10pt/@{->}"D";"A",
		\ar@{->}"H";"G",
		\ar@{->}"G";"E",
		\ar@{->}"K";"J",
		\ar@{->}"J";"I",
		\ar@{->}"I";"A",
		\ar@{->}"I";"B",
		\ar@{->}"I";"C",
		\ar@{->}"I";"D",
		\ar@{<-}"E";"A",
		\ar@{<-}"E";"B",
		\ar@{<-}"E";"C",
		\ar@{<-}"E";"D",
		\ar@{<-}"F";"A",
		\ar@{<-}"F";"B",
		\ar@{<-}"F";"C",
		\ar@{<-}"F";"D",
		\end{xy}
	\]
\end{exa}
Let $\alpha=(\alpha_{a})_{a\in \mathsf{Q}_{0}}
\in \mathbb{Z}^{\mathsf{Q}_{0}}$ be the vector, 
\[
	\alpha_{[i,j]}:=n^{(i)}_{j}\quad\text{ and }\quad
\alpha_{[i,j,k]}:=\mathrm{rank\,}
\prod_{l=1}^{k}(R^{(i)}_{j}-\xi_{l}^{[i,j]}).
\]
Also define $\lambda=(\lambda_{a})_{a\in \mathsf{Q}_{0}}\in 
\mathbb{C}^{\mathsf{Q}_{0}}$ by 
\begin{align*}
	\lambda_{[i,j]}&:=-\xi^{[i,j]}_1&&\text{ for }
	i\in I_{\text{irr}}\backslash\{0\},\,
	j=1,\ldots,m^{(i)}, \\
	\lambda_{[0,j]}&:=-\xi^{[0,j]}
	-\sum_{i\in I_{\text{reg}}}\xi^{[i,1]}_{1}&&\text{ for }
	j=1,\ldots,m^{(0)},\\
	\lambda_{[i,j,k]}&:=\xi^{[i,j]}_{k}-\xi^{[i,j]}_{k+1}&&\text{ for }
	\begin{array}{l}i=0,\ldots,p,\, j=1,\ldots,m^{(i)},\\
		k=1,\ldots,e_{[i,j]}-1.
	\end{array}
\end{align*}
Also define a sublattice of $\mathbb{Z}^{\mathsf{Q}_{0}}$, 
\[
	\mathcal{L}=\left\{\beta\in 
\mathbb{Z}^{\mathsf{Q}_{0}}\,\middle|\,
	\sum_{j=1}^{m^{(0)}}\beta_{[0,j]}=\sum_{j=1}^{m^{(i)}}\beta_{[i,j]}
\text{ for all }i\in I_{\text{irr}}\backslash\{0\}
\right\}.
\]
Set $\mathcal{L}^{+}=\mathcal{L}\cap (\mathbb{Z}_{\ge 0})^{\mathsf{Q}_{0}}$.
\subsubsection{$\mathfrak{M}(\mathbf{B})$
and  a quiver variety}
Now we shall give an identification
of $\mathfrak{M}(\mathbf{B})$ with 
a subspace of the quiver variety
$\mathfrak{M}_{\lambda}(\mathsf{Q},\alpha)$. Before seeing this, we introduce 
$\mathcal{L}$-irreducible representations in 
$\mu^{-1}(\lambda)$ which are 
defined by a weaker condition than 
the irreducibility.

\begin{df}[$\mathcal{L}$-irreducible]
    \normalfont
    If $x\in \mu^{-1}(\lambda)$ has 
    no nontrivial proper subrepresentation
    $\{0\}\neq y\subsetneqq x$ in 
    $\mu^{-1}(\lambda)$
    with 
    $\mathbf{dim\,}y\in \mathcal{L}$,
    then $x$ is said to be {\em $\mathcal{L}$-irreducible}.
\end{df}
Then we have the following bijection 
from $\mathfrak{M}(\mathbf{B})$ onto 
a subset of the quiver variety 
$\mathfrak{M}_{\lambda}(\mathsf{Q},\alpha)$.
\begin{thm}[Theorem 5.14 in \cite{H2}]\label{quiverreal}
	There exists a bijection
	\[
		\Phi_{\mathbf{B}}
		\colon \mathfrak{M}(
		\mathbf{B})\longrightarrow
	\mathfrak{M}_{\lambda}(\mathsf{Q},\alpha)^{\text{dif}}\]
	where
	\[
		\mathfrak{M}_{\lambda}(\mathsf{Q},\alpha)^{\text{dif}}:=
		\left\{
			x\in\mu^{-1}(\lambda)\ \middle|\  \begin{array}{c}
				x\text{ is $\mathcal{L}$-irreducible},\\
			\mathrm{det}
			\left(
			x_{\rho^{[0,j]}_{[i,j']}}\right)_{\substack{1\le j\le m^{(0)}\\1\le j'\le m^{(i)}}}\neq 0,
			i\in 
			I_{\text{irr}}
			\backslash\{0\}
		\end{array}
		\right\}
		\bigg/\mathbf{G}.
	\]
\end{thm}

\begin{rem}\label{invariance}\normalfont
Let us see that invariance of $\mathfrak{M}(\mathbf{B})$ under some changes of 
irregular parts of HTL normal forms in $\mathbf{B}$.
Let $\mathbf{B}$ be the collection of HTL normal forms as above.	
Let us consider another collection of HTL normal forms 
$\bar{\mathbf{B}}=(\bar{B}^{(0)},\ldots,\bar{B}^{(p)})$ of the forms
\[
	\bar{B}^{(i)}=\mathrm{diag}\left(
	\bar{q}^{(i)}_{1}(z^{-1})I_{n_{1}^{(i)}}+\bar{R}_{1}^{(i)}z^{-1},
	\ldots,
	\bar{q}^{(i)}_{m^{(i)}}(z^{-1})I_{n^{(i)}_{m^{(i)}}}+\bar{R}_{m^{(i)}}^{(i)}z^{-1}
	\right).
\]
Here we note that  
HTL normal forms $B^{(i)}$ and $\bar{B}^{(i)}$ share the same $n^{(i)}_{j}$,
the sizes of diagonal blocks for $i=0,\ldots,p, j=1,\ldots,m^{(i)}$, 
Define $\bar{d}_{i}(j,j'):=\mathrm{deg}_{\mathbb{C}[z]}(\bar{q}^{(i)}_{j}(z)-
\bar{q}^{(i)}_{j'}(z^{-1}))-2$.
Suppose that 
\begin{align*}
	\bar{d}_{i}(j,j')&=d_{i}(j,j'),&
	C_{\bar{R}^{(i)}_{j}}=C_{R^{(i)}_{j}}
\end{align*}
for all $i=0,\ldots,p$ and $j,j'=1,\ldots,m^{(i)}$.
Here $C_{A}$ is the conjugacy class of $A\in M(n,\mathbb{C})$ under the 
adjoint action of $\mathrm{GL}(n,\mathbb{C})$.
Then from the above construction of the  quiver and Theorem \ref{quiverreal}, we have
the bijection
\[
	\Phi_{\bar{\mathbf{B}}}\colon \mathfrak{M}(\bar{\mathbf{B}})\longrightarrow
	\mathfrak{M}_{\lambda}(\mathsf{Q},\alpha)^{\text{dif}}
\]
with the same $\mathsf{Q}$, $\lambda$, $\alpha$ as for $\mathbf{B}$.
Thus we have an isomorphism 
\[
	\Phi_{\bar{\mathbf{B}}}^{-1}\circ\Phi_{\mathbf{B}}\colon
	\mathfrak{M}(\mathbf{B})\xrightarrow[]{\sim}
	\mathfrak{M}(\bar{\mathbf{B}}).
\]

\end{rem}

Now let us recall the construction of the above $\Phi_{\mathbf{B}}$
which 
is first obtained by Crawley-Boevey in \cite{C} when $k_{0}=\cdots =k_{p}=1$,
by Boalch in \cite{Boarx} when $k_{0}=3$ and $k_{1}=\cdots=k_{p}=1$, 
by Yamakawa and the author in \cite{HY} 
when $k_{0}\ge 1$ and $k_{1}=\cdots=k_{p}=1$ 
and finally by the author in \cite{H2} for general $k_{0},\ldots,k_{p}\in 
(\mathbb{Z}_{>0})^{p+1}$.

First we decompose truncated orbits $\mathcal{O}_{B^{(i)}}^{\text{tru}}$ as follows.
Set 
\begin{align*}
	G^{o}_{k_{i}}&:=
	\left\{I_{n}+\sum_{i=1}^{k_{i}-1}g_{i}z^{i}\in 
	\mathrm{GL}(n,\mathbb{C}[\![z]\!]/z^{k_{i}}
\mathbb{C}[\![z]\!])\right\},\\
		(\mathcal{O}^{\text{tru}}_{B^{(i)}})^{o}&:=
		\left\{g^{-1}Bg\,\middle|\,
		g\in G_{k_{i}}^{o}
		\right\}
	\subset M(n,\mathbb{C}(\!(z)\!)/
		\mathbb{C}[\![z]\!]),\\
		H_{i}&:=
		\left\{
			\mathrm{diag}(h_{1},
			\ldots,h_{m^{(i)}})\,\middle|\, 
			h_{j}\in \mathrm{GL}(n_{j}^{(i)}),\,
			j=1,\ldots,m^{(i)}
		\right\}.
\end{align*}
for $i=0,\ldots,p$. Here we note that $G^{o}_{k_{i}}$ acts on $M(n,\mathbb{C}(\!(z)\!)/
\mathbb{C}[\![z]\!])$ as well as $\mathrm{GL}(n,\mathbb{C}[\![z]\!])$.
Then we can decompose $\mathcal{O}_{B^{(i)}}^{\text{tru}}$ 
as follows.
\begin{prop}[see Lemma 2.4 in \cite{Boa1} and Proposition 4.9 in \cite{H2}]\label{decoupling}
	For each $i=0,\ldots,p$, we have the bijection 
	\[
		\begin{array}{ccc}
			\mathrm{GL}(n,\mathbb{C})\times_{H^{(i)}}
			\mathrm{Ad}_{H_{i}}
			(\mathcal{O}_{B^{(i)}}^{\text{tru}})^{o}&
			\longrightarrow &
			\mathcal{O}^{\text{tru}}_{B^{(i)}}\\
			(g,A)&\longmapsto& gAg^{-1}.
		\end{array}
	\]
	Here $\mathrm{Ad}_{H_{i}}
			(\mathcal{O}_{B^{(i)}}^{\text{tru}})^{o}:=
			\left\{h(\mathcal{O}_{B^{(i)}}^{\text{tru}})^{o}h^{-1}
			\,\middle| \, h\in H_{i}\right\}$ and 
			$\mathrm{GL}(n,\mathbb{C})\times_{H^{(i)}}
			(\mathcal{O}_{B^{(i)}}^{\text{tru}})^{o}
			:=\left(\mathrm{GL}(n,\mathbb{C})\times
			(\mathcal{O}_{B^{(i)}}^{\text{tru}})^{o}
		\right)/\sim$
		by the identification 
		$(g,A)\sim (gh^{-1},hAh^{-1})$ for $h\in H^{(i)}$.

\end{prop}
Thus it suffices to investigate
the structure of $(\mathcal{O}_{B^{(i)}}^{\text{tru}})^{o}$.

Fix $i\in I_{\text{irr}}$ and write $B^{(i)}=B^{(i)}_{1}z^{-1}+\cdots +
B_{k_{i}}^{(i)}z^{-k_{i}}$.
Let $\bigoplus_{t=1}^{m^{(i)}_{s}}V^{(i)}_{\langle s,t \rangle}$
be the decomposition of $\mathbb{C}^{n}$ as simultaneous eigen-spaces
of $(B^{(i)}_{s+1},B^{(i)}_{s+2},\ldots,B^{(i)}_{k_{i}})$ for 
$s=1,\ldots,k_{i}-1$.

Define surjections $\pi_{i}\colon J_{s}^{(i)}:=
\{1,\ldots,m^{(i)}_{s}\}\rightarrow J_{s+1}^{(i)}:=
\{1,\ldots,m^{(i)}_{s+1}\}$ so that 
$V_{\langle s, t\rangle}\subset V_{\langle s+1,\pi_{s}(t)\rangle}$.
Fix a total ordering $\prec$ of 
$J_{1}=\{1\prec 2\prec \cdots \prec m^{(i)}\}$. Then inductively define 
total orderings on $J_{s}$, $s=2,\ldots,k_{i}-1$ so that
\[
	\text{if }t_{1}\prec t_{2},
	\text{ then }\pi_{s}(t_{1})\prec \pi_{s}(t_{2}),\quad 
	t_{1},t_{2}\in J_{s}.
\]
According to the  ordering on each $J^{(i)}_{s}$, $s=1,\ldots,k_{i}-1$, 
let us define parabolic subalgebras of $M(n,\mathbb{C})$ as below,
\begin{align*}
	(\mathfrak{p}_{s}^{(i)})^{+}
	&:=\bigoplus_{\substack{t_{1},t_{2}\in J_{s}^{(i)},\\
    t_{1}\succeq t_{2}}}
    \mathrm{Hom}_{\mathbb{C}}(V^{(i)}_{\langle s,t_{1}\rangle},
    V^{(i)}_{\langle s,t_{2}\rangle}),\\
    (\mathfrak{p}^{(i)}_{s})^{-}
    &:=\bigoplus_{\substack{t_{1},t_{2}\in J^{(i)}_{s},\\
    t_{1}\preceq t_{2}}}
    \mathrm{Hom}_{\mathbb{C}}(V^{(i)}_{\langle s,t_{1}\rangle},
    V^{(i)}_{\langle s,t_{2}\rangle}),
\end{align*}
and  similarly nilpotent subalgebras
\begin{align*}
	(\mathfrak{u}_{s}^{(i)})^{+}
	&:=\bigoplus_{\substack{t_{1},t_{2}\in J_{s}^{(i)},\\
    t_{1}\succ t_{2}}}
    \mathrm{Hom}_{\mathbb{C}}(V^{(i)}_{\langle s,t_{1}\rangle},
    V^{(i)}_{\langle s,t_{2}\rangle}),\\
    (\mathfrak{u}^{(i)}_{s})^{-}
    &:=\bigoplus_{\substack{t_{1},t_{2}\in J^{(i)}_{s},\\
    t_{1}\prec t_{2}}}
    \mathrm{Hom}_{\mathbb{C}}(V^{(i)}_{\langle s,t_{1}\rangle},
    V^{(i)}_{\langle s,t_{2}\rangle}),
\end{align*}
for $s=1,\ldots,k_{i}-1$.
Then let us define subsets of $G_{k_{i}}^{o}$,
\begin{align*}
	&\mathcal{P}^{\pm}_{k_{i}}:=\left\{
	\sum_{s=0}^{k_{i}-1}P_{s}z^{s}\in G^{o}_{k_{i}}\,\bigg|\, P_{s}\in 
	(\mathfrak{p}_{s+1}^{(i)})^{\pm}, s=0,\ldots,k_{i}-1
    \right\},\\
    &\mathcal{U}^{\pm}_{k_{i}}:=\left\{
    \sum_{s=0}^{k_{i}-1}U_{s}z^{s}\in G^{o}_{k_{i}}\,\bigg|\, U_{s}\in 
    (\mathfrak{u}_{s+1}^{(i)})^{\pm}, s=0,\ldots,k_{i}-1
    \right\}.
\end{align*}
Also define 
\begin{align*}
	\mathfrak{U}^{\pm}_{k_{i}}&:=
	\left\{
		\sum_{s=1}^{k_{i}-1}U_{s}z^{s}\,\bigg|\,
		U_{s}\in (\mathfrak{u}^{(i)}_{s+1})^{\pm},\,
		s=0,\ldots,k_{i}-1
	\right\},\\
	(\mathfrak{U}^{\mp}_{k_{i}})^{*}&:=
    \left\{
	    \sum_{s=1}^{k_{i}-1}U_{s}z^{-s-1}\,\bigg|\,
	    U_{s}\in (\mathfrak{u}^{(i)}_{s+1})^{\pm},\,
	    s=0,\ldots,k_{i}-1
    \right\}.
\end{align*}
Here we put $(\mathfrak{p}^{(i)}_{k_{i}})^{\pm}:=M(n,\mathbb{C})$ and 
$(\mathfrak{u}^{(i)}_{k_{i}})^{\pm}:=\{0\}$.

Then we have the following decomposition of $G_{k_{i}}^{o}$.
\begin{prop}[Lemma 3.5 in \cite{HY}]
	Take $i\in I_{\text{irr}}$.
	For any $g\in G^{o}_{k_{i}}$, there uniquely exist 
	$u_{-}\in \mathcal{U}_{k_{i}}^{-}$ and $p_{+}\in 
	\mathcal{P}^{+}_{k_{i}}$
	such that $g=u_{-}p_{+}$.
\end{prop}

For $A\in (\mathcal{O}_{B^{(i)}}^{\text{tru}})^{o}$, take $g\in G_{k_{i}}^{o}$
so that $g^{-1}Ag=B^{(i)}$ and decompose $g=u_{-}p_{+}$ as above.
Define 
\begin{equation}\label{triangle}
	Q:=u_{-}-I_{n},\quad \quad P:=u_{-}^{-1}A|_{(\mathfrak{U}_{k_{i}}^{-})^{*}}.
\end{equation}
Notice that we can show that 
these $P$ and $Q$ are independent of the choice of $g$.
Conversely Theorem 3.6 in \cite{HY} 
tells us that $A-\mathrm{pr}_{\text{res}}(A)$ is uniquely 
determined by these $P$ and 
$Q$.

Now we are ready to define the map $\Phi_{\mathbf{B}}$.
Take $\mathbf{A}=(A^{(i)}(z))_{0\le i\le p}\in \mathfrak{M}(\mathbf{B})$ 
and define $x=\Phi_{\mathbf{B}}(\mathbf{A})$ as follows.
As we saw in Proposition \ref{decoupling}, 
choose $g_{i}\in \mathrm{GL}(n,\mathbb{C})$ and $\tilde{A}^{(i)}(z)
\in (\mathcal{O}_{B^{(i)}}^{\text{tru}})^{o}$ so that 
$g_{i}\tilde{A}^{(i)}(z)g_{i}^{-1}
=A^{(i)}(z)$ for $i\in I_{\text{irr}}$.
Then define 
\begin{align*}
	&\text{for }\rho^{[0,j]}_{i,j'}\in \mathsf{Q}_{1}^{0\to I_{\text{irr}}},
	&\begin{cases}
	x_{\rho^{[0,j]}_{i,j'}}:=(g_{i}^{-1})_{[i,j'],
	[0,j]},\\
x_{(\rho^{[0,j]}_{i,j'})^{*}}:=(g_{i}
\tilde{A}^{(i)}_{1})_{[0,j],[i,j']},
	\end{cases}\\
	&\text{for }\rho^{[k]}_{[i,j],[i,j']}\in \mathsf{Q}_{1}^{B^{(i)}},
\ i\in I_{\text{irr}},
&\begin{cases}
x_{\rho^{[k]}_{[i,j],[i,j']}}:=(P^{(i)})^{[k]}_{[i,j] ,
[i,j']},\\
x_{(\rho^{[k]}_{[i,j],[i,j']})*}:=(Q^{(i)})^{[k]}_{[i,j'],
[i,j]}.
\end{cases}
\end{align*}
Here  $\tilde{A}^{(i)}_{1}:=\mathrm{pr}_{\text{res}}(\tilde{A}^{(i)}(z))$ and 
$X_{[i,j],
[i',j']}$ denotes $\mathrm{Hom}_{\mathbb{C}}(V^{(i')}_{\langle 1,j'
\rangle},V^{(i)}_{\langle 1,j\rangle})$ component of $X\in M(n,\mathbb{C})$.
Furthermore, $P^{(i)}=\sum_{k=1}^{k_{i}-1}(P^{(i)})^{[k]}z^{-k-1}$ and 
$Q^{(i)}=\sum_{k=1}^{k_{i}-1}(Q^{(i)})^{[k]}z^{k}$ are defined from 
$\tilde{A}^{(i)}(z)$ by the above equations (\ref{triangle}).

Also define
\begin{align*}
	&x_{\rho_{[i,j,k]}}:\mathrm{Im}\prod_{l=1}^{k}
	\left((\tilde{A}^{(i)}_{1})_{j,j}-
\xi^{[i,j]}_{l}I_{n^{(i)}_{j}}\right)\hookrightarrow 
\mathrm{Im}\prod_{l=1}^{k-1}\left((\tilde{A}^{(i)}_{1})_{j,j}-
\xi^{[i,j]}_{l}I_{n^{(i)}_{j}}\right),\\
&x_{(\rho_{[i,j,k]})^{*}}:=\left((\tilde{A}^{(i)}_{1})_{j,j}-
\xi^{[i,j]}_{k}I_{n^{(i)}_{j}}
\right)\Big|_{\mathrm{Im}\prod_{l=1}^{k-1}\left((\tilde{A}^{(i)}_{1})_{j,j}-
\xi^{[i,j]}_{l}I_{n^{(i)}_{j}}\right)},
\end{align*}
for $i=0,\ldots,p$, $j=1,\ldots,m^{(i)}$ and $k=2,\ldots,e_{[i,j]}-1$.
Here $X_{j,j}$ denote 
$\mathrm{End}_{\mathbb{C}}(V^{(i)}_{\langle 1,j\rangle})$-components of $X$
for $j=1,\ldots,m^{(i)}$.
For $i\in I_{\text{irr}}$ and $j=1,\ldots,m^{(i)}$,
\begin{align*}
	&x_{\rho_{[i,j,1]}}:\mathrm{Im}
	\left((\tilde{A}^{(i)}_{1})_{j,j}-
\xi^{[i,j]}_{1}I_{n^{(i)}_{j}}\right)\hookrightarrow 
V^{(i)}_{\langle 1,j\rangle},\\
&x_{(\rho_{[i,j,1]})^{*}}:=(\tilde{A}^{(i)}_{1})_{j,j}-
\xi^{[i,j]}_{1}I_{n^{(i)}_{j}}.
\end{align*}
For $i\in I_{\text{reg}}$ and $j=1,\ldots,m^{(0)}$, 
\begin{align*}
	&x_{\rho^{[i,1,1]}_{[0,j]}}:\mathrm{Im}
	\left(\tilde{A}^{(i)}_{1}-
	\xi^{[i,1]}_{1}I_{n^{(i)}_{1}}\right)\hookrightarrow \mathbb{C}^{n}
	\xrightarrow[ \mathbb{C}^{n}=
	\bigoplus_{l=1}^{m^{(0)}}V^{(0)}_{\langle 1,l\rangle}]
{\text{projection along }}
V^{(0)}_{\langle 1,j\rangle},\\
&x_{(\rho^{[i,1,1]}_{[0,j]})^{*}}:=\left(\tilde{A}^{(i)}_{1}-
\xi^{[i,1]}_{1}I_{n^{(i)}_{1}}\right)\Big|_{V^{(0)}_{\langle 1,j\rangle}}.
\end{align*}
\subsubsection{Open embedding of $\mathfrak{M}(\mathbf{B})$ into 
a quiver variety}
Now we notice that 
\[
	\mathfrak{M}_{\lambda}(\mathsf{Q},\alpha)^{\text{dif}}\not\subset
\mathfrak{M}^{\text{reg}}_{\lambda}(
\mathsf{Q},\alpha)
\]
 in general,
 since the $\mathcal{L}$-irreducibility is weaker 
than the ordinary irreducibility. 
Thus it seems to be possible that
$\mathfrak{M}(\mathbf{B})\cong 
\mathfrak{M}_{\lambda}(\mathsf{Q},\alpha)^{\text{dif}}$ has singularities.

To consider this problem, we 
introduce an operation on $\mathfrak{M}(\mathbf{B})$ called {\em addition}.
\begin{df}[addition]\normalfont
	For a collection of complex numbers 
	\[\mathbf{v}=(v_{1},\ldots,v_{p})\in \mathbb{C}^{p},\]
	the {\em addition} translates
	a differential equation
	\[
		\frac{d}{dz}Y=\left(\sum_{i=1}^{p}\sum_{\nu=1}^{k_{i}}
		\frac{A^{(i)}_{\nu}}{(z-a_{i})^{\nu}}
	-\sum_{2\le\nu\le k_{0}}A^{(0)}_{\nu}z^{\nu-2}
\right)Y\in \mathfrak{M}(\mathbf{B})\]
	to
	\[
		\frac{d}{dz}Y=\left(\sum_{i=1}^{p}\sum_{\nu=1}^{k_{i}}
			\left(\frac{A^{(i)}_{\nu}}{(z-a_{i})^{\nu}}
			+\frac{v_{i}I_{n}}{z-a_{i}}\right)
	-\sum_{2\le\nu\le k_{0}}A^{(0)}_{\nu}z^{\nu-2}
\right)Y\in \mathfrak{M}(\mathbf{B+\mathbf{v}}).
	\]
Here 
$\mathbf{B}+\mathbf{v}:=(B_{i}+v_{i}I_{n}z^{-1})_{i=0,\ldots,p}$ 
with $v_{0}:=-\sum_{i=1}^{p}v_{i}$.
\end{df}
Thus for $\mathbf{v}\in\mathbb{C}^{p}$ the addition defines the bijection 
\[
	\mathrm{Add}_{\mathbf{v}}\colon 
	\mathfrak{M}_{\lambda}(\mathsf{Q},\alpha)^{\text{dif}}
	\longrightarrow
	\mathfrak{M}_{\lambda+\bar{\mathbf{v}}}(\mathsf{Q},\alpha)^{\text{dif}}
\]
where $\bar{\mathbf{v}}=(\bar{v}_{a})_{a\in\mathsf{Q}_{0}}$ is defined as follows,
\begin{align*}
	\bar{v}_{[i,j]}&:=v_{i}\quad\text{ for all }[i,j]\in 
	\mathsf{Q}_{0}^{\text{irr}} 
	\text{ and }i\in I_{\text{irr}}\backslash\{0\},\\
	\bar{v}_{[0,j]}&:=v_{0}+\sum_{k\in I_{\text{reg}}}v_{k}\quad
	\text{ for all }j=1,\ldots,m^{(0)},\\
	\bar{v}_{[i,j,k]}&:=0\quad\text{ for all }[i,j,k]\in
	\mathsf{Q}_{0}^{\text{leg}}.
\end{align*}

Then we can find a nice $\mathbf{v}$ such that $\mathrm{Add}_\mathbf{v}$
sends $\mathfrak{M}_{\lambda}(\mathsf{Q},\alpha)^{\text{dif}}$ into 
the quiver variety $\mathfrak{M}_{\lambda+\bar{\mathbf{v}}}^{\text{reg}}
(\mathsf{Q},\alpha)$.

\begin{thm}\label{embed}
	Let $\mathfrak{M}_{\lambda}(\mathsf{Q},\alpha)^{\text{dif}}$ be 
	as above. Then there exists $\mathbf{v}\in \mathbb{C}^{p}$ such 
	that 
	\begin{multline*}
		\mathfrak{M}_{\lambda+\bar{\mathbf{v}}}
		(\mathsf{Q},\alpha)^{\text{dif}}=\\
		\left\{
			x\in\mu^{-1}(\lambda+\bar{\mathbf{v}})\ \middle|\  \begin{array}{c}
				x\text{ is irreducible},\\
			\mathrm{det}
			\left(
			x_{\rho^{[0,j]}_{[i,j']}}\right)_{\substack{1\le j\le m^{(0)}\\1\le j'\le m^{(i)}}}\neq 0,
			i\in 
			I_{\text{irr}}
			\backslash\{0\}
		\end{array}
		\right\}
		\bigg/\mathbf{G}
		\subset
		\mathfrak{M}_{\lambda+\bar{\mathbf{v}}}^{\text{reg}}
		(\mathsf{Q},\alpha)
	\end{multline*}
\end{thm}
\begin{proof}
	Lemma 5.9 in \cite{H2} shows that there exists 
	$\mathbf{v}\in \mathbb{C}^{p}$
	such that $\lambda'=\lambda+\bar{\mathbf{v}}$ satisfies the following.
	If $\beta\in (\mathbb{Z}_{\ge 0})^{\mathsf{Q}_{0}}$ satisfies that 
	$\beta\le \alpha$, i.e., $\beta_{a}\le \alpha_{a}$ for all $a
	\in \mathsf{Q}_{0}$, and $\lambda'\cdot \beta=0$, then
	$\beta\in \mathcal{L}$. Thus any subrepresentation $y$ of 
	$x\in \mathfrak{M}_{\lambda'}(\mathsf{Q},\alpha)$ satisfies 
	that $\mathbf{dim\,}y\in \mathcal{L}$, that is, the $\mathcal{L}$-
	irreducibility implies the irreducibility.
\end{proof}
Thus we have an open embedding of $\mathfrak{M}(\mathbf{B})$ into 
the regular part of a quiver variety.
\begin{thm}\label{embedding}
Let us take $\mathbf{B}=
	 (B^{(i)})_{0\le i\le p}$, a collection of HTL normal forms,
	 the quiver $\mathsf{Q}$,
	 $\alpha\in (\mathbb{Z}_{\ge 0})^{\mathsf{Q}_{0}}$ and 
	 $\lambda\in \mathbb{C}^{\mathsf{Q}_{0}}$ as above. 
	 Then there exists 
	 $\mathbf{v}\in\mathbb{C}^{p}$
	 and an injection
	 \[
		 \Phi\colon \mathfrak{M}(\mathbf{B})
		 \hookrightarrow
		 \mathfrak{M}^{\text{reg}}_{\lambda+\bar{\mathbf{v}}}(\mathsf{Q},\alpha)
	 \]
	 such that 
	 \[
		 \Phi(\mathfrak{M}(\mathbf{B}))=
		 \left\{x\in 
			 \mathfrak{M}^{\text{reg}}_{\lambda+
			 \bar{\mathbf{v}}}(\mathsf{Q},
			 \alpha)\,\middle|\,
			 \mathrm{det}\left(
				 x_{\rho^{[0,j]}_{[i,j']}}
			 \right)_{\substack{1\le j\le m^{(0)}\\
			 1\le j'\le m^{(i)}}}\neq 0,\
		 i\in I_{\text{irr}}\backslash\{0\}
		 \right\}.
	 \]
	 In particular if $I_{\text{irr}}=\{0\}$, then $\mathbf{v}=\mathbf{0}$
	 and $\Phi$ is bijective.
\end{thm}

\begin{cor}
	If $\mathfrak{M}(\mathbf{B})\neq \emptyset$, then 
	it can be seen as 
	a connected symplectic complex manifold of dimension $2p(\alpha)$.
\end{cor}
\begin{proof}
	If $\mathfrak{M}(\mathbf{B})\neq \emptyset$, then 
	$\mathfrak{M}_{\lambda}^{\text{reg}}(\mathsf{Q},\alpha)\neq 
	\emptyset$ and 
	thus $\mathfrak{M}_{\lambda+\bar{\mathbf{v}}}^{\text{reg}}
	(\mathsf{Q},\alpha)\neq 
	\emptyset$ with $\mathbf{v}\in\mathbb{C}^{p}$ chosen as in
	the previous theorem.
	Thus by Theorems \ref{CB} and \ref{CB2},
	it suffices to check the connectedness.

	Theorem \ref{CB2} says that $\mu^{-1}(\lambda+\bar{
	\mathbf{v}})$ is an irreducible variety.
	Let us recall that 
	\[
		\mu^{-1}(\lambda+\bar{\mathbf{v}})^{\text{sta}}:=
	\left\{
		x\in \mu^{-1}(\lambda+\bar{\mathbf{v}})\,\middle|\,
		x\text{ is stable under 
		}\mathbf{G}
\right\}
\]
is an open subset of $\mu^{-1}(\lambda+\bar{\mathbf{v}})$ (see Proposition
5.15 in \cite{Mu} for instance).
Since $\mu^{-1}(\lambda+\bar{\mathbf{v}})^{\text{sta}}=\mu^{-1}(\lambda+\bar{\mathbf{v}})^{\text{irr}}$,
\[\mu^{-1}(\lambda+\bar{\mathbf{v}})^{\text{det}}:=
	\left\{
			x\in\mu^{-1}(\lambda+\bar{\mathbf{v}})^{\text{irr}}\ 
			\middle|\  
			\mathrm{det}
			\left(
			x_{\rho^{[0,j]}_{[i,j']}}\right)_{\substack{1\le j\le m^{(0)}\\1\le j'\le m^{(i)}}}\neq 0,
			i\in 
			I_{\text{irr}}
			\backslash\{0\}
		\right\}
	\]
	is also an open subset of $\mu^{-1}(\lambda+\bar{\mathbf{v}})$.
	Since open subsets of an irreducible topological space are connected,
	$\mu^{-1}(\lambda+\bar{\mathbf{v}})^{\text{det}}$ is connected.
	Moreover   
	$\mathfrak{M}^{\text{reg}}_{\lambda+
	\bar{\mathbf{v}}}(\mathsf{Q},\alpha)^{\text{dif}}$ is the 
	image of the continuous projection 
	from $\mu^{-1}(\lambda+\bar{\mathbf{v}})^{\text
	{det}}$, thus it is connected not only as an algebraic variety but also
	as an analytic space by GAGA.
\end{proof}

\subsubsection{Non-emptiness of $\mathfrak{M}(\mathbf{B})$}
We close this subsection by giving a necessary and sufficient condition for $\mathfrak{M}(\mathbf
{B})\neq \emptyset$.
Define a set ${\Sigma}_{\lambda}^{\text{dif}}$ consists of 
$\beta\in \mathcal{L}^{+}$ satisfying
\begin{enumerate}
		\item $\beta$ is a positive root of $\mathsf{Q}$ and 
			$\beta\cdot \lambda=0$,
		\item for any decomposition $\beta=\beta_{1}+\cdots+
			\beta_{r}$ where $\beta_{i}\in 
			\mathcal{L}^{+}$ are positive roots of 
			$\mathsf{Q}$ satisfying $\beta_{i}\cdot
			\lambda=0$, we have 
			\[
				p(\beta)>
				p(\beta_{1})+\cdots +p(\beta_{r}).
			\]
	\end{enumerate}

	\begin{thm}[Non-emptiness of moduli spaces. Theorem 0.9 in 
		\cite{H2}]\label{DSproblem}
The moduli space $\mathfrak{M}(\mathbf{B})\neq \emptyset$
if and only if $\alpha\in \Sigma_{\lambda}^{\text{dif}}$.
\end{thm}

Let us recall the 
{\em spectral type} which is already appeared in Section 1.2 in \cite{KNS}.
Consider the inductive limit
\[
	\mathbb{Z}^{\infty}:=\lim_{\longrightarrow}\mathbb{Z}^{n}
\]
defined by inclusions $\phi_{i,i+1}\colon 
\mathbb{Z}^{i}\ni (a_{1},\ldots,a_{i})
\mapsto (a_{1},\ldots,a_{i},0)\in \mathbb{Z}^{i+1}$ for $i=1,2, \ldots$.

\begin{df}[spectral type and index of rigidity]\label{spectype}\normalfont
	The {spectral type} of $\mathbf{B}$ is the pair 
	\[
		\left(\mathbf{m}_{\alpha},\left(d_{i}(j,j')
			\right)_{\substack{i=0,\ldots,p\\
	1\le j<j'\le m^{(i)}}}\right)
\]
where $\mathbf{m}_{\alpha}=\left((m_{[i,j,1]},\ldots,
		m_{[i,j,e_{[i,j]}]})\right)_{
		\substack{0\le i\le p\\1\le j\le m^{(i)}}}\in 
		\bigoplus_{i=0}^{p}\bigoplus_{j=1}^{m^{(i)}}
	\mathbb{Z}^{\infty}$ which satisfies 
	$\sum_{j=1}^{m^{(0)}}\sum_{k=1}^{e_{[0,j]}}m_{[0,j,k]}=\cdots
	=\sum_{j=1}^{m^{(p)}}\sum_{k=1}^{e_{[p,j]}}m_{[p,j,k]}$
	is defined by 
	\[
		m_{[i,j,k]}:=\alpha_{[i,j,k-1]}-\alpha_{[i,j,k]}
	\]
	where 
	\[
		\alpha_{[i,j,0]}=\begin{cases}
			\alpha_{[i,j]}&\text{ if }i\in I_{\text{irr}},\\
			\sum_{k=1}^{m^{(0)}}\alpha_{[0,k]}
			&\text{ if }i\in I_{\text{reg}}
		\end{cases}
\]
	and $\alpha_{[i,j,e_{[i,j]}]}=0$. 
Sometimes we write $\mathbf{m}_{\alpha}=(\mathbf{m}_{\alpha},d_{i}(j,j'))$ for short.

The {\em index of rigidity} of $\mathbf{m}_{\alpha}$ is defined by
\[
	\mathrm{idx}\mathbf{m}:=2q(\alpha).
\]
\end{df}

Here we note that we do not distinguish $\mathbf{m}_{\alpha}$ and 
\[
	\left(\left((m_{[\sigma(i),s(j),t(1)]},\ldots,
		m_{[\sigma(i),s(j),t(e_{[\sigma(i),s(j)]})]})\right)_{
		\substack{0\le i\le p\\1\le j\le m^{(i)}}},
		\left(d_{\sigma(i)}(s(j),s(j'))\right)_{\substack{i=0,\ldots,p\\
	1\le j<j'\le m^{(i)}}}
	\right)
\]
for any permutations 
$\sigma\in \mathfrak{S}_{p+1}$, $s\in \mathfrak{S}_{m^{(i)}}$ 
and $t\in \mathfrak{S}_{e_{[i,j]}}$ for $i=0,\ldots,p$, $j=1,\ldots, m^{(i)}$.

For convenience we introduce the following notation for $\mathbf{m}$.
    The each number $d_{i}(j,j')+1$ is expressed by the 
    number of parentheses $(\ )$ between the sequences 
    $m_{[i,j,1]},m_{[i,j,2]},\ldots$ and $m_{[i,j',1]}, m_{[i,j',2]},\ldots$. 
    For instance, if  
    $$\mathbf{m}_{\beta}=\cdots m_{[i,j,1]}m_{[i,j,2]}\ldots 
    m_{[i,j,l_{i,j}]}))( (m_{[i,j',1]}m_{[i,j',2]}\cdots,$$ 
    then we can see the double parenthesis $))\,( ($ 
    between $m_{[i,j,1]}\ldots,$ and $m_{[i,j',1]}\ldots$. 
    This means $d_{i}(j,j')=1$.
    
    For example, put   
    $p=1$, $(m^{(0)},m^{(1)})=(2,3)$, $(e_{[0,1]},e_{[0,2]},e_{[1,1]},e_{[1,2]},e_{[1,3]})=(1,2,1,1,2)$ 
    and 
    $(d_{0}(1,2),d_{1}(1,2),d_{1}(2,3),d_{1}(1,3))=(0,0,1,1)$. 
    
    Then $\mathbf{m}=( (m_{[i,j,1]},\ldots,
    m_{[i,j,l_{i,j}]}))_{\substack{0\le i\le p\\1\le j\le k_{i}}}$ 
    is written by
    \[
	    (m_{[0,1,1]})(m_{[0,2,1]}m_{[0,2,2]}),\,
	    ( (m_{[1,1,1]})(m_{[1,2,1]}))( (m_{[1,3,1]}m_{[1,3,2]})).
    \]

    \subsection{Integrable deformation}\label{Integrable family}
Let us introduce {\em integrable admissible families} 
of connections following 
Boalch \cite{Boa1} and Yamakawa \cite{Y2}.

Let $\mathbb{T}$ be a contractible complex manifold and 
$a_{i}\colon\mathbb{T} \rightarrow \mathbb{P}^{1}\times 
\mathbb{T}$, $i=0,\ldots,p$,
holomorphic sections of the fiber bundle 
$\pi \colon \mathbb{P}^{1}\times \mathbb{T}\rightarrow \mathbb{T}$.
Moreover assume that 
\[
	a_{i}(t)\neq a_{j}(t)\text{ if }i\neq j
\]
in each fiber $\mathbb{P}^{1}_{t}:=\mathbb{P}^{1}\times\{t\}$.
Moreover we fix a standard coordinate $z\colon \mathbb{P}^{1}_{t}\cong \mathbb{C}\cup \{\infty\}$
so that $a_{0}(s)=\infty$ and $d_{\mathbb{T}}z=0$ on the trivial bundle $\mathbb{P}^{1}\times 
\mathbb{T}\rightarrow \mathbb{T}$.
Let us set 
\[
	z_{i}\colon \mathbb{P}^{1}\times \mathbb{T}\rightarrow \mathbb{T};
	\quad (z,t)\mapsto \begin{cases}
		1/z&(i=0)\\
		z-a_{i}(t)&(i\neq 0)
	\end{cases}
\]
for $i=0,\ldots,p$.
Let us consider a family $\mathbf{B}(t)=(B^{(i)}(t))_{i=0,\ldots,p}$
of collections of HTL normal forms of the forms
\[
	B^{(i)}(t)=\mathrm{diag}\left(
	q^{(i)}_{1}(t,z_{i}^{-1})I_{n_{1}^{(i)}}+{R}_{1}^{(i)}(t)z_{i}^{-1},
	\ldots,
	q^{(i)}_{m^{(i)}}(t,z_{i}^{-1})I_{n^{(i)}_{m^{(i)}}}+R_{m^{(i)}}^{(i)}(t)z_{i}^{-1}
	\right).
\]
Here all mappings $\mathbb{T}\ni t\mapsto q^{(i)}_{j}(t,z)\in\mathbb{C}[z]$ 
and $\mathbb{T}\ni t\mapsto R^{(i)}_{j}(t)\in M(n^{(i)}_{j},\mathbb{C})$
depend smoothly on $t\in \mathbb{T}$.
Define $d_{i}(t;j,j'):=\mathrm{deg}_{\mathbb{C}[z]}(q^{(i)}_{j}(t,z)-
\bar{q}^{(i)}_{j'}(t,z^{-1}))-2$.
We say that $\mathbf{B}(t)$ is an {\em admissible family\footnote{This is a little stronger 
	condition than that in \cite{Y2}.}} of the collections 
of HTL normal forms if 
$d_{i}(t;j,j')$ and $R^{(i)}_{j}(t)$ are independent of $t$ for all $i=0,\ldots,p$ and 
$j,j'=1,\ldots,m^{(i)}$.

Let $(\mathbf{B}(t))_{t\in \mathbb{T}}$ be an admissible family of collections 
of HTL normal forms.
Then as we saw in Remark \ref{invariance}, 
we can find quiver $\mathsf{Q}$, $\alpha\in \mathbb{Z}^{\mathsf{Q}_0}$
and $\lambda\in \mathbb{C}^{\mathsf{Q}_{0}}$ independently of $t\in \mathbb{T}$  such that  we have 
isomorphisms
\[
	\Phi_{\mathbf{B}(t)}\colon \mathfrak{M}(\mathbf{B}(t))\xrightarrow[]{\sim}
	\mathfrak{M}_{\lambda}(\mathsf{Q},\alpha)^{\text{dif}}
\]
for all $t\in \mathbb{T}$. 
We further say that the admissible family $(\mathbf{B}(t))_{t\in 
	\mathbb{T}}$ is {\em non-resonant} if eigenvalues of $R^{(i)}_{j}(t)$
	never differ by any integer for each $i=0,\ldots,p$ and $j=1,\ldots,m^{(i)}$,
	which is equivalent to 
	the condition,
\[
	\lambda_{[i,j,k]}\not\in \mathbb{Z}\backslash\{0\}\text{ for all }
	[i,j,k]\in \mathsf{Q}_{0}^{\text{leg}}.
\]

\begin{df}[admissible family]\normalfont
	Then the family 
	$\left((\mathcal{O}_{\mathbb{P}^{1}_{t}}^{n},\nabla_{t})\right)_{t\in\mathbb{T}}$
	of meromorphic connections
	is called an {\em admissible family} with  $(\mathbf{B}(t))_{t\in \mathbb{T}}$
	if the followings are satisfied:
	\begin{enumerate}
		\item the admissible family $(\mathbf{B}(t))_{t\in \mathbb{T}}$ is 
			non-resonant.
		\item We have $(\mathcal{O}_{\mathbb{P}^{1}_{t}}^{n},\nabla_{t})\in 
			\mathfrak{M}(\mathbf{B}(t))$ for all $t\in \mathbb{T}$.
		\item For each $i=0,\ldots,p$ and fixed $t\in \mathbb{T}$, 
			let us write $\nabla_{t}=d-A_{i}(t,z_{i})\,dz_{i}$,
			$A_{i}(t,z_{i})\in M(n,\mathbb{C}(\!(z_{i})\!))$ near $z_{i}=0$.
		Then there exists a holomorphic map
			$\widehat{g_{i}}\colon 
			\mathbb{T}\rightarrow 
			\mathrm{GL}(n,\mathbb{C}[\![z_{i}]\!])$
			such that 
			\[
				A_{i}(t,z_{i})=\widehat{g}_{i}(t)[B^{(i)}(t)].
			\]
	\end{enumerate}
	As we see above, we can define the triple $(\mathsf{Q},\lambda,\alpha)$
	from $(\mathbf{B}(t))_{t\in \mathbb{T}}$. We call this triple 
	the {\em spectral data} of the admissible family  $\left((\mathcal{O}_{\mathbb{P}^{1}_{t}}^{n},\nabla_{t})\right)_{t\in\mathbb{T}}$
	with 
	$(\mathbf{B}(t))_{t\in \mathbb{T}}$.
	We call the  
	number
	$2p(\alpha)=\mathrm{dim}
	(\mathfrak{M}_{\lambda}(\mathsf{Q},\alpha))=\mathrm{dim}(\mathfrak{M}(\mathbf{B}(t)))$,
	the {\em dimension} of the admissible family.
\end{df}

\begin{df}[integrable family]\normalfont
	Let $\left((\mathcal{O}_{\mathbb{P}^{1}_t}^{n},\nabla_{t})\right)_{t\in \mathbb{T}}$
	be an admissible family with $(\mathbf{B}(t))_{t\in \mathbb{T}}$. 
	If there exists a flat 
	meromorphic connection $\widehat{\nabla}$  on $\mathcal{O}_{
		\mathbb{P}^{1}\times \mathbb{T}}^{n}$ with 
		poles on $\bigcup_{i=0}^{p}a_{i}(\mathbb{T})$ such that
	$\widehat{\nabla}|_{\mathbb{P}^{1}_{t}}=\nabla_{t}$,
	then we say that  the family $\left((\mathcal{O}_{\mathbb{P}^{1}_t}^{n},\nabla_{t})\right)
	_{t\in \mathbb{T}}$ is {\em integrable}.
	In this case such $(\mathcal{O}_{
		\mathbb{P}^{1}\times \mathbb{T}}^{n},\widehat{\nabla})$ is called a
	{\em flat extension} of $\left((\mathcal{O}_{\mathbb{P}^{1}_t}^{n},\nabla_{t})\right)_{t\in \mathbb{T}}$.
	\end{df}
\section{Middle convolutions, Weyl groups and integrable deformations}
\label{middleconvolutionandreflection}
In the previous section, we saw that moduli spaces of stable meromorphic 
connections are realized as quiver varieties. 
As it is known, quiver varieties 
have Weyl group symmetries generated by 
reflection functors (see \cite{CH} and \cite{Nak2}).
Similarly on the moduli space side, we also have the symmetries  
generated by middle convolutions.
In this section, we see the relationship between middle convolutions 
and Weyl groups of quivers and give a classification of their symmetries 
in certain lower dimensional cases. And we see the symmetries of integrable 
families as an application.
\subsection{A review of middle convolutions}
Let us give a review of middle convolutions on differential 
equations with irregular singular points.
The middle convolution is originally defined by Katz in 
\cite{Katz} and 
reformulated as an operation on Fuchsian systems by 
Dettweiler-Reiter \cite{DR2}, see also \cite{DR}  and V\"olklein's paper \cite{V}.
There are several studies to generalize the middle convolution
to non-Fuchsian differential equations, 
see \cite{A},\cite{Kaw},\cite{T},\cite{Y1}
for example.
Among them we shall give a review of middle convolutions following  \cite{Y1}.

Let us take $(\mathcal{O}^{n},\nabla)\in \mathfrak{M}(\mathbf{B})$ and write
\[
	\nabla=d-\left(\sum_{i=1}^{p}\sum_{\nu=1}^{k_{i}}
		\frac{A^{(i)}_{\nu}}{(z-a_{i})^{\nu}}
	-\sum_{2\le\nu\le k_{0}}A^{(0)}_{\nu}z^{\nu-2}
\right)dz.
\]
Set \[
	\mathbf{A}=(\sum_{j=1}^{k_{i}}A^{(i)}_{j}z^{-j})_{0\le i\le p}
\in \prod_{i=0}^{p}\mathcal{O}_{B^{(i)}}\]
where $A_{1}^{(0)}:=-\sum_{i=1}^{p}
A^{(i)}_{1}$.  

Let us construct a 5-tuple
$(V,W,T,Q,P)$ consisting of $\mathbb{C}$-vector spaces $V$, $W$
and 
$T\in \mathrm{End}_{\mathbb{C}}(W)$,
$Q\in \mathrm{Hom}_{\mathbb{C}}(W,V)$, 
$P\in \mathrm{Hom}_{\mathbb{C}}(V,W)$.
Set $V=\mathbb{C}^{n}$ and $\widehat{W}_{i}=V^{\oplus k_{i}}$ for 
 $i=0,\ldots,p$.
Then define
\begin{align*}
	\widehat{Q}_{i}&:=(A^{(i)}_{k_{i}},A^{(i)}_{k_{i}-1},\ldots,A^{(i)}_{1})
	\in \mathrm{Hom}_{\mathbb{C}}(\widehat{W}_{i},V),\\
	\widehat{P}_{i}&:=\begin{pmatrix}0\\\vdots\\0\\
		\mathrm{Id}_{V}
	\end{pmatrix}
	\in \mathrm{Hom}_{\mathbb{C}}(V,\widehat{W}_{i}),\
	\widehat{N}_{i}:=
	\begin{pmatrix}
		0&\mathrm{Id}_{V}&&0\\
		&0&\ddots&\\
		&&\ddots&\mathrm{Id}_{V}\\
		0&&&0
	\end{pmatrix}
	\in \mathrm{End}_{\mathbb{C}}(\widehat{W}_{i}).
\end{align*}
Setting  
\begin{align*}
	\widehat{W}&:=\bigoplus_{i=0}^{p}\widehat{W}_{i},\\
\widehat{T}&:=(\widehat{N}_{i})_{0\le i\le p}\in \bigoplus_{i=0}^{p}
\mathrm{End}_{\mathbb{C}}(\widehat{W}_{i})\subset\mathrm{End}_{\mathbb{C}}
(\widehat{W}),\\
\widehat{Q}&:=(\widehat{Q}_{i})_{0\le i\le p}\in 
\bigoplus_{i=0}^{p}\mathrm{Hom}_{\mathbb{C}}(\widehat{W}_{i},V)
=\mathrm{Hom}_{\mathbb{C}}(\widehat{W},V),\\
\widehat{P}&:=(\widehat{P}_{i})_{0\le i\le p}\in 
\bigoplus_{i=0}^{p}\mathrm{Hom}_{\mathbb{C}}(V,\widehat{W}_{i})
=\mathrm{Hom}_{\mathbb{C}}(V,\widehat{W}),
\end{align*}
we have a 5-tuple $(V,\widehat{W},\widehat{T},\widehat{Q},\widehat{P})$.
Further setting
\[
	\widehat{A}_{i}:=
\begin{pmatrix}
	A^{(i)}_{k_{i}}&A^{(i)}_{k_{i}-1}&\cdots&A^{(i)}_{1}\\
	&A^{(i)}_{k_{i}}&\ddots&\vdots\\
	&&\ddots&A^{(i)}_{k_{i}-1}\\
	0&&&A^{(i)}_{k_{i}}
\end{pmatrix}
\in \mathrm{End}_{\mathbb{C}}(\widehat{W}_{i}),
\]
we define 
$W_{i}:=\widehat{W}_{i}/\mathrm{Ker}\widehat{A}_{i}$ and 
$W:=\bigoplus_{i=0}^{p}W_{i}$.
Then $T,Q,P$ are the maps induced from $\widehat{T},\widehat{Q},
\widehat{P}$ respectively.
\begin{df}\normalfont
	The 5-tuple $(V,W,T,Q,P)$ given above is called the 
	{\em canonical datum} for $\mathbf{A}\in
	\prod_{i=0}^{p}\mathcal{O}_{B^{(i)}}$.
\end{df}

Fix $t\in \{0,\ldots,p\}$, take a polynomial 
$q_{t}(z^{-1})=\sum_{j=1}^{k_{t}}q^{(t)}_{j}z^{-j}\in 
z^{-1}\mathbb{C}[z^{-1}]$ and define an operation, 
called {\em addition} which already appeared in some special cases before. 
For an element $\mathbf{A}=
(A_{i}(z^{-1}))_{0\le i\le p}\in 
\prod_{i=0}^{p}\mathcal{O}_{B^{(i)}}$,
we define $\mathrm{Add}^{(t)}_{q_{t}(z^{-1})}(\mathbf{A})
:=(A'_{i}(z^{-1}))_{0\le i\le p}$
by
\[
	A'_{i}(z^{-1}):=
	\begin{cases}
		A_{i}(z^{-1})&\text{if }i\neq t,\\
		A_{t}(z^{-1})-q_{t}(z^{-1})&
		\text{if }i=t.
	\end{cases}
\]
Then $
\mathrm{Add}^{(t)}_{q_{t}(x^{-1})}(\mathbf{A})\in \prod_{i=0}^{p}\mathcal{O}
_{(B')^{(i)}}$
where 
\[
	(B')^{(i)}:=
	\begin{cases}
		B^{(i)}&\text{if }i\neq t,\\
		B^{(t)}-q_{t}(z^{-1})&\text{if }i=t.
	\end{cases}
\]

Set 
\[\mathcal{J}_{i}:=\{[i,j]\mid j=1,\ldots,m^{(i)}\}\quad\text{ for }
	i=0,\ldots,p
\]
and  
\[
	\mathcal{J}:=\prod_{i=0}^{p}\mathcal{J}_{i}.
\]
Then let us define
\[
	\mathrm{Add}_{\mathbf{i}}:=\prod_{i=0}^{p}
	\mathrm{Add}^{(i)}_{q^{(i)}_{j_{i}}(z^{-1})+\xi^{[i,j_{i}]}_{1}z^{-1}},
\]  
for $\mathbf{i}=([i,j_{i}])_{0\le i\le p}\in \mathcal{J}$.

Suppose that we can 
choose $\mathbf{i}\in \mathcal{J}$ so that $\xi_{\mathbf{i}}:=
\sum_{i=0}^{p}\xi^{[i,j_{i}]}_{1}\neq 0$.
Let $(V,W,T,Q,P)$  be the canonical datum of 
$\mathrm{Add}_{\mathbf{i}}(\mathbf{A})$. 
Following Example 3 in \cite{Y1}, we construct a new 5-tuple
$(V',W,T,Q',P')$ as follows.
Note that $QP=-\xi_{\mathbf{i}}\mathrm{Id}_{V}$.
Thus $Q$ and $P$ are surjective and injective respectively.
Let us set $V'=\mathrm{Coker\,}P$ and $Q'\colon W\rightarrow V'$,
the natural projection. Then we have the split exact sequence
\[
	0\longrightarrow V\stackrel{P}{\longrightarrow}W
	\stackrel{Q'}{\longrightarrow}V'\longrightarrow 0.
\]
Note that $(-\xi_{\mathbf{i}}^{-1}Q)P=\mathrm{Id}_{V}$.
Let $P'\colon V'\rightarrow W$ be the injection such that 
$Q'(\xi_{\mathbf{i}}^{-1}P')=\mathrm{Id}_{V'}$.
Then we have a 5-tuple
$(V',W,T,Q',P')$.

Next we set $Q'_{i}$ (resp. $P'_{i}$) to be the 
$\mathrm{Hom}_{\mathbb{C}}(W_{i},V)$
(resp. $\mathrm{Hom}_{\mathbb{C}}(V,W_{i})$) component of $Q'$ (resp. $P'$).
Also set $N_{i}$ to be the $\mathrm{End}_{\mathbb{C}}(W_{i})$-component
of $T$.
Define 
\[
	(A')^{(i)}_{j}:=Q'_{i}N_{i}^{j-1}P'_{i}
\]
and $\mathbf{A}'=(A'_{i}(z^{-1}))_{0\le i\le p}$ where 
$A'_{i}(z^{-1})=\sum_{j=1}^{k_{i}}(A')^{(i)}_{j}z^{-j}$.
We note that $\sum_{i=0}^{p}(A')^{(i)}_{1}=Q'P'
=\xi_{\mathbf{i}}\mathrm{Id}_{V'}$.

Finally let us set 
\[
	\mathbf{A}'':=\mathrm{Add}_{\mathbf{i}}^{-1}\circ
	\mathrm{Add}^{(0)}_{2\xi_{\mathbf{i}}z^{-1}}(\mathbf{A}').
\]
Then $\mathbf{A}''=(A''_{i}(z^{-1}))_{0\le i\le p}$ satisfies that 
$\sum_{i=0}^{p}\pr_{\text{res}}A''_{i}(z^{-1})=0$.
Let us denote $\mathbf{A}''$ by 
$\mathrm{mc}_{\mathbf{i}}(\mathbf{A})$  and 
call the operator $\mathrm{mc}_{\mathbf{i}}$ the {\em middle convolution}
at $\mathbf{i}$.
We also denote the corresponding connection 
$\nabla'':=d-(\sum_{i=1}^{p}\sum_{\nu=1}^{k_{i}}\frac{(A'')^{(i)_{\nu}}}{(z-a_{i})^{\nu}}-
\sum_{2\le \nu\le k_{0}}(A'')^{(0)}_{\nu}z^{\nu-2})dz$ by 
$\mathrm{mc}_{\mathbf{i}}(\nabla)$.

Let us recall basic properties of middle convolutions.
\begin{prop}[see Yamakawa \cite{Y1}]\label{middleconv}
	Suppose we can choose $\mathbf{i}\in\mathcal{J}$ so that
	$\xi_{\mathbf{i}}\neq 0$.
	\begin{enumerate}
		\item If $\nabla$ is stable, then 
			$\mathrm{mc}_{\mathbf{i}}(\nabla)$ is stable.
		\item If $\nabla$ is stable, 
			\[
				\mathrm{mc}_{\mathbf{i}}\circ
				\mathrm{mc}_{\mathbf{i}}(\mathbf{A})
				\sim \mathbf{A},
			\]
			i.e., there exists $g\in \mathrm{GL}(n,\mathbb{C})$
			 such that $$\mathrm{mc}_{\mathbf{i}}\circ
				\mathrm{mc}_{\mathbf{i}}(\mathbf{A})
				=g\mathbf{A}g^{-1}:=(gA_{i}(z^{-1})g^{-1})_{
				0\le i\le p}.$$

		\item Let us define elements in $M( (n')^{(i)}_{j},\mathbb{C})$
			by
			\[
				(R')^{(i)}_{j}:=
				\begin{cases}
					R^{(i)}_{j}+
				(d_{i}(j,j_{i})+2)
				\xi_{\mathbf{i}}I_{n^{(i)}_{j}}&
				\text{if }i\neq 0,\\	
				R^{(0)}_{j}+
				d_{0}(j,j_{0})
				\xi_{\mathbf{i}}I_{n^{(0)}_{j}}&
				\text{if }i=0
			\end{cases}
			\]
			for all $i\in \{0,\ldots,p\}$ and 
			$j\in\{1,\ldots,m^{(i)}\}\backslash\{j_{i}\}$.
			Here $(n')^{(i)}_{j}:=n^{(i)}_{j}$.
			Further define $(R')^{(i)}_{j_{i}}
			\in M( (n')^{(i)}_{j_{i}},\mathbb{C})$ 
			for $i=1,\ldots,p$ so that 
			equations
			\begin{align*}
				\mathrm{rank\,}
				( (R')^{(i)}_{j_{i}}-\xi^{[i,j_{i}]}_{1})
				\prod_{k=2}^{l}
				( (R')^{(i)}_{j_{i}}-\xi^{[i,j_{i}]}_{k}-
				\xi_{\mathbf{i}})=
				\mathrm{rank\,}
				\prod_{k=1}^{l}
				(R^{(i)}_{j_{i}}-\xi_{k}^{[i,j_{i}]})
			\end{align*}
			hold for all $l=2,\ldots,e_{[i,j_{i}]}$.
			Similarly define  
			$(R')^{(0)}_{j_{0}}
			\in M( (n')^{(0)}_{j_{0}},\mathbb{C})$ 
			so that 
			equations
			\begin{equation*}
				\mathrm{rank\,}
				( (R')^{(0)}_{j_{0}}-\xi^{[0,j_{0}]}_{1}+2
				\xi_{\mathbf{i}})
				\prod_{k=2}^{l}
				( (R')^{(0)}_{j_{0}}-\xi^{[0,j_{0}]}_{k}+
				\xi_{\mathbf{i}})=
				\mathrm{rank\,}
				\prod_{k=1}^{l}
				(R^{(0)}_{j_{0}}-\xi_{k}^{[0,j_{0}]})
			\end{equation*}
			hold for all $l=2,\ldots,e_{[0,j_{0}]}$.

			Here we put 
			\[
				(n')^{(i)}_{j_{i}}:=
				n^{(i)}_{j_{i}}+\mathrm{dim}_{\mathbb{C}}W
				-2n.
			\]
			Finally define 
			\begin{multline*}
				(B')^{(i)}:=\\
				\mathrm{diag}\left(
				q^{(i)}_{1}(z^{-1})I_{(n')^{(i)}_{1}}+
				(R')^{(i)}_{1}z^{-1},\ldots,
				q^{(i)}_{m^{(i)}}(z^{-1})
				I_{(n')^{(i)}_{m^{(i)}}}+(R')^{(i)}_{m^{(i)}}
				z^{-1}
				\right)
			\end{multline*}
			for $i=0,\ldots,p$.
			Then $\mathrm{mc}_{\mathbf{i}}(\mathbf{A})
			\in \prod_{i=0}^{p}\mathcal{O}_{(B')^{(i)}}$.
	\end{enumerate}
\end{prop}

\begin{rem}\label{modulitransform}\normalfont
	Let us note that the description of $\mathbf{A}''=
	\mathrm{mc}_{\mathbf{i}}(\mathbf{A})$ depends on the choice
	of the coordinate systems of $W_{i}$ and $V'$
	in the canonical data. Thus $\mathrm{mc}_{\mathbf{i}}$ defines 
	the following well-defined bijection
	\[
		\mathrm{mc}_{\mathbf{i}}\colon
		\mathfrak{M}(\mathbf{B})
		\longrightarrow
		\mathfrak{M}(\mathbf{B}')
	\]
	where $\mathbf{B}'=( (B')^{(i)})_{0\le i\le p}$.
\end{rem}
\subsection{Middle convolutions on  
representations of a quiver}
For a vertex with no edge-loop in $\mathsf{Q}_{0}$, it is known that 
there exists 
a bijection
\[
	s_{a}\colon \mathfrak{M}_{\lambda}(\mathsf{Q},\alpha)\rightarrow 
	\mathfrak{M}_{r_{a}(\lambda)}(\mathsf{Q},s_{a}(\alpha))
\]
if $\lambda\neq 0$, so-called {\em reflection functor} see \cite{CH}
and \cite{Nak2}.
In this section, we shall define 
an analogy of the reflection functors for the subspace 
$\mathfrak{M}_{\lambda}(\mathsf{Q},\alpha)^{\text{dif}}\subset 
\mathfrak{M}_{\lambda}(\mathsf{Q},\alpha)$ by using middle convolutions.
We notice that  
a reflection functor
does not necessarily preserve 
the subset $\mathfrak{M}_{\lambda}(\mathsf{Q},\alpha)^{\text{dif}}$, namely 
it may happen that 
\[s_{a}\left(\mathfrak{M}_{\lambda}(\mathsf{Q},\alpha)^{\text{dif}}\right)
	\not\subset\mathfrak{M}_{r_{a}(\lambda)}(\mathsf{Q},s_{a}(\alpha)
	)^{\text{dif}}
\]
for some $a\in \mathsf{Q}_{0}$.
However 
as we saw in Remark \ref{modulitransform}, a middle convolution 
$\mathrm{mc}_{\mathbf{i}}$ can be seen as a transformation of 
moduli spaces $\mathfrak{M}(\mathbf{B})\cong
\mathfrak{M}_{\lambda}(\mathsf{Q},\alpha)^{\text{dif}}$. 
Thus it can be seen as 
a transformation of quiver varieties $\mathfrak{M}_{\lambda}(\mathsf{Q},
\alpha)^{\text{dif}}$ as below.

For $\mathbf{i}=([i,j_{i}])_{0\le i\le p}
\in \mathcal{J}$, let us define $\epsilon_{\mathbf{i}}
\in \mathbb{Z}^{\mathsf{Q}_{0}}$ by
\[
	(\epsilon_{\mathbf{i}})_{a}:=
	\begin{cases}
		1&\text{if }a=[i,j_{i}],\,i\in I_{\text{irr}},\\
		0&\text{otherwise}.
	\end{cases}
\]
We note that $\epsilon_{\mathbf{i}}$ for $\mathbf{i}\in \mathcal{J}$
are positive real roots of $\mathsf{Q}$.
Let us define 
\[
	s_{\mathbf{i}}(\beta):=\beta-(\beta,\epsilon_{\mathbf{i}})
\epsilon_{\mathbf{i}}
\]
for $\mathbf{i}\in \mathcal{J}$ and $\beta\in \mathbb{Z}^{\mathsf{Q}_{0}}$.
Also define $r_{\mathbf{i}}(\mu)$ for $\mu\in \mathbb{C}^{\mathsf{Q}_{0}}$
by 
\begin{align*}
	r_{\mathbf{i}}(\mu)_{[i,j]}&:=
		\begin{cases}
			\mu_{[i,j]}&\text{if }
			[i,j]\neq [0,j_{0}],\\
			\mu_{[0,j_{0}]}-2\mu_{\mathbf{i}}&
			\text{if }
			[i,j]=[0,j_{0}],
		\end{cases}\\
		r_{\mathbf{i}}(\mu)_{[i,j,k]}&:=
		\begin{cases}
			\mu_{[i,j,k]}&
			\text{if }
			[i,j,k]\neq [i,j_{i},1],\\
			\mu_{[i,j_{i},1]}+\mu_{\mathbf{i}}&
			\text{if }[i,j,k]=[i,j_{i},1].
		\end{cases}
\end{align*}

Then Proposition \ref{middleconv} tells us the following.
\begin{thm}
	Let us consider $\mathfrak{M}(\mathbf{B})\neq \emptyset$ and 
	the corresponding quiver variety $\mathfrak{M}_{\lambda}(
	\mathsf{Q},\alpha)^{\text{dif}}$ under the bijection in 
	Theorem \ref{quiverreal}.
Suppose that we can take $\mathbf{i}=([i,j_{i}])\in \mathcal{J}$ so that 
$\lambda_{\mathbf{i}}:=\sum_{i\in I_{\text{irr}}}\lambda_{[i,j_{i}]}
=-\xi_{\mathbf{i}}\neq 0$.  
Then there exists a bijection
\[
	s_{\mathbf{i}}\colon 
	\mathfrak{M}_{\lambda}(\mathsf{Q},\alpha)^{\text{dif}}
	\longrightarrow
	\mathfrak{M}_{r_{\mathbf{i}}(\lambda')}
	(\mathsf{Q},s_{\mathbf{i}}(\alpha))^{\text{dif}}
\]
\end{thm}
\begin{proof}
	We retain the notation in Proposition \ref{middleconv} and 
	put $\mathbf{B}'=( (B')^{(i)})_{0\le 
i\le p}$.
This theorem directly follows from Proposition \ref{middleconv} 
if we define the above map by 
$\Phi_{\mathbf{B}'}\circ\mathrm{mc}_{\mathbf{i}}\circ 
\Phi_{\mathbf{B}}^{-1}$ where $\Phi_{\mathbf{B}}$ and $\Phi_{\mathbf{B}'}$
are in Theorem \ref{quiverreal}. 
\end{proof}

\begin{rem}\normalfont
	For each $[i,j,k]\in \mathsf{Q}_{0}^{\text{leg}}$, 
	the ordinary reflection functor of quiver varieties
gives a bijection
\[
	s_{[i,j,k]}\colon \mathfrak{M}_{\lambda}(\mathsf{Q},\alpha)^{\text{dif}}
	\longrightarrow \mathfrak{M}_{r_{[i,j,k]}(\lambda)}(
	\mathsf{Q},s_{[i,j,k]}(\alpha))^{\text{dif}}
\]
if $\lambda_{[i,j,k]}\neq 0$.
\end{rem}
\begin{rem}\normalfont
	Let $(\mathbf{B}(t))_{t\in\mathbb{T}}$ be an admissible family of HTL normal forms.
	Let us choose $(\mathsf{Q},\lambda,\alpha)$ as in Section \ref{Integrable family}.
	Then Remark \ref{invariance} and Proposition \ref{middleconv} show that 
	for each $\mathbf{i}\in\mathcal{J}$,
	middle convolution $\mathrm{mc}_{\mathbf{i}}$ for $\mathfrak{M}(\mathbf{B}(t))$
	induces the reflection $s_{\mathbf{i}}$ for $\alpha$ and $\lambda$ of 
	$\mathsf{Q}$ independently of $t\in \mathbb{T}$.
\end{rem}
Let us define an analogue of fundamental set of the root lattice $\mathbb{Z}^{
\mathsf{Q}_{0}}$,
\[
	\tilde{F}:=
	\left\{
		\beta\in\mathcal{L}^{+}\backslash\{0\}\ \middle|\ 
		\begin{array}{c}
			(\beta,\epsilon_{a})\le 0\text{ for all }a\in 
			\mathcal{J}\cup \mathsf{Q}_{0}^{\text{leg}}\\
			\text{support of }\beta\text{ is connected}
		\end{array}
	\right\}
\]
called {\em $\mathcal{L}$-fundamental set}.
Then we can see that $\tilde{F}$ can be seen as a fundamental domain under 
the action of the group 
\[
	W^{\mathrm{mc}}:=\left\langle s_{\mathbf{i}},\,
s_{[i,j,k]}\mid \mathbf{i}\in \mathcal{J},\,[i,j,k]\in\mathsf{Q}_{0}
^{\text{leg}}\right\rangle.
\]
\begin{thm}\label{reduction}
	For $\mathfrak{M}_{\lambda}(\mathsf{Q},\alpha)^{\text{dif}}\neq 
	\emptyset$, there exists $w\in W^{\mathrm{mc}}$ such that 
	\[
		w\left(\mathfrak{M}_{\lambda}(\mathsf{Q},
		\alpha)^{\text{dif}}\right)=\mathfrak{M}_{\lambda'}(
		\mathsf{Q},\alpha')^{\text{dif}}
	\]
	with 
	\[
		\begin{cases}
		\alpha'\in \tilde{F}&\text{if }q(\alpha)\le 0,\\
			\alpha'=\epsilon_{\mathbf{i}}\text{ for some }
			\mathbf{i}\in \mathcal{J}&\text{otherwise. }
		\end{cases}
	\]
\end{thm}
\begin{proof}
	See Lemma 7.2, Theorem 7.9 and Theorem 7.10 in \cite{H2}.
\end{proof}

We introduce a condition for $\lambda$ which will be used in the latter section.
\begin{df}\normalfont
	For $\mathfrak{M}_{\lambda}(\mathsf{Q},\alpha)^{\text{dif}}\neq 
	\emptyset$, we say $\lambda$ is {\em fractional} if 
	\[
		\lambda'_{\mathbf{i}}:=
		\sum_{i\in I_{\text{irr}}}\lambda'_{[i,j_{i}]}\notin\mathbb{Z}
	\]
	for all $\mathbf{i}\in \mathcal{J}$ and $\lambda'\in 
	\{r(\lambda)\mid r\in \langle r_{[i,j,k]}\mid [i,j,k]\in 
	\mathsf{Q}_{0}^{\text{leg}}\rangle\}.$
	
	Moreover if there exists a sequence $a_{1},a_{2},\ldots,a_{l}\in 
	\mathcal{J}\cup \mathsf{Q}_{0}^{\text{leg}}$ such that 
	\[
		r_{a_k}\circ r_{a_{k-1}}\circ\cdots \circ r_{a_1}(\lambda)
		\text{ are fractional}
	\]
	for all $k=1,\ldots,l$ and $w=s_{a_{l}}s_{a_{l-1}}\cdots s_{1}$
	where $w\in W^{\mathrm{mc}}$ is chosen as in Theorem \ref{reduction},
	then we say that $\lambda$ has a {\em fractional reduction}.
\end{df}
\begin{rem}\normalfont
	Let $\mathbf{B}=(B^{(i)})_{0\le i\le p}$ be a collection
	 of HTL normal forms
	 \[
    B^{(i)}=
    \mathrm{diag}\left(
    q^{(i)}_{1}(z^{-1})I_{n^{(i)}_{1}}+R^{(i)}_{1}z^{-1},\ldots,
    q^{(i)}_{m^{(i)}}(z^{-1})I_{n^{(i)}_{m^{(i)}}}+R^{(i)}_{m^{(i)}}z^{-1}
    \right)
\]
such that $\mathfrak{M}(\mathbf{B})\neq \emptyset$.
Then $\lambda$ of $\mathfrak{M}(\mathbf{B})\cong \mathfrak{M}_{\lambda}(
\mathsf{Q},\alpha)^{\text{dif}}$ is fractional if and only if 
$\sum_{i=0}^{p}\xi_{i}\not\in \mathbb{Z}$ where $\xi_{i}$ is an 
arbitrary eigenvalue 
of $\mathrm{pr}_{\text{res}}(B^{(i)})$ for each $i=0,\ldots,p$.

\end{rem}
\subsection{The lattice $\mathcal{L}$ as a Kac-Moody root lattice}
As we saw in Theorem \ref{DSproblem}, 
if $\mathfrak{M}(\mathbf{B})\cong 
\mathfrak{M}_\lambda(\mathsf{Q},\alpha)^{\text{dif}}\neq \emptyset$,
then $\alpha$ must be in 
$\mathcal{L}\cap \Delta$ where $\Delta$ is the set of 
roots in $\mathbb{Z}^{\mathsf{Q}_{0}}$.
This inclines us to see $\mathcal{L}\cap \Delta$
as an analogy of the set of roots of the lattice 
$\mathcal{L}$ which may not be a true Kac-Moody root lattice.

\begin{df}[symmetric Kac-Moody root lattice]\normalfont
We call a $\mathbb{Z}$-lattice 
$L:=\bigoplus_{i\in I}\mathbb{Z}\alpha_{i}$ with a finite index set $I$
a {\em symmetric Kac-Moody root lattice}, when $L$ has the following 
bilinear form 
\begin{align*}
	(\alpha_{i},\alpha_{i})&=2\quad\quad (i\in I),\\
	(\alpha_{i},\alpha_{j})&=(\alpha_{j},\alpha_{i})\in 
	\mathbb{Z}_{\le 0}\quad\quad (i,j\in I,\ i\neq j).
\end{align*}
For each $\alpha_{i},\ (i\in I)$ which is 
called a {\em simple root}, we can define the {\em simple reflection} by 
\[
	s_{\alpha_{i}}(\beta):=\beta-(\beta,\alpha_{i})\alpha_{i}
\]
for $\beta\in L$. The {\em Weyl group} $W\in \mathrm{Aut}_{\mathbb{Z}}(L)$
is the group generated by all simple reflections $s_{\alpha_{i}}$, $i\in I$.

We can attach $L$ to a diagram, called the {\em Dynkin diagram},
regarding simple roots  as vertices and 
connecting $\alpha_{i},\alpha_{j}$ by $|(\alpha_{i},\alpha_{j})|$ 
edges if $i\neq j$. 

\end{df}
Notions of real roots, fundamental set and imaginary roots and so on
are also defined in the same way as   
we saw in \S \ref{some geometry}.
Then for our quiver $\mathsf{Q}$, the $\mathbb{Z}^{\mathsf{Q}_{0}}$ is 
a symmetric Kac-Moody root lattice.

It can be checked that $\mathcal{L}$ is generated by $\{\epsilon_{a}\mid 
a\in \mathcal{J}\cup\mathsf{Q}_{0}^{\text{leg}}\}$ over $\mathbb{Z}$ and 
$W^{\mathrm{mc}}=\langle s_{a}\mid a\in 
\mathcal{J}\cup\mathsf{Q}_{0}^{\text{leg}}\rangle$ acts on $\mathcal{L}$.
This may lead us to believe that $\mathcal{L}$ can be seen as a root 
lattice with the set of simple roots $\{\epsilon_{a}\mid 
a\in \mathcal{J}\cup\mathsf{Q}_{0}^{\text{leg}}\}$ and the Weyl group
$W^{\text{mc}}$.
However elements in $\{\epsilon_{a}\mid 
a\in \mathcal{J}\cup\mathsf{Q}_{0}^{\text{leg}}\}$ are not independent 
over $\mathbb{Z}$ in general.
Thus we shall introduce a new lattice $\widehat{\mathcal{L}}$ of which 
$\mathcal{L}$ can be seen as a quotient.
Let us note that 
\begin{align}
	(\epsilon_{\mathbf{i}},\epsilon_{\mathbf{i}'})&=
	2-\sum_{\substack{0\le i\le p\\j_{i}\neq j'_{i}}}
	(d_{i}(j_{i},j'_{i})+2),\label{equ5}\\
	(\epsilon_{\mathbf{i}},\epsilon_{[i,j,k]})&=
	\begin{cases}
		-1&\text{if }j=j_{i}\text{ and }k=1,\label{equ6}\\
		0&\text{otherwise},
	\end{cases}\\
	(\epsilon_{[i,j,k]},\epsilon_{[i',j',k']})&=
	\begin{cases}
		2&\text{if }[i,j,k]=[i',j',k'],\\
		-1&\text{if }(i,j)=(i',j')\text{ and }
		|k-k'|=1,\\
		0&\text{otherwise}
	\end{cases}\label{equ7}
\end{align}
for $\mathbf{i},\mathbf{i}'\in 
\mathcal{J}$ and $[i,j,k],[i',j',k']
\in \mathsf{Q}_{0}^{\text{leg}}$.
Thus we consider a new lattice $\widehat{\mathcal{L}}$
generated by the set of indeterminate
\[
	\mathcal{C}=\left\{c_{a}\mid a\in \mathcal{J}\cup
	\mathsf{Q}_{0}^{\text{leg}}\right\},
\]
and define a symmetric bilinear form $(\,,\,)$ on $\widehat{\mathcal{L}}$
in accordance with equations $(\ref{equ5}),(\ref{equ6})$ and 
$(\ref{equ7})$.
Then $\widehat{\mathcal{L}}$ becomes a symmetric Kac-Moody root lattice and
we have a projection
\[
	\Xi\colon \widehat{\mathcal{L}}\longrightarrow
	\mathcal{L}
\]	
where for $\gamma=\sum_{c\in\mathcal{C}}\gamma_{c}c\in \widehat{\mathcal{L}}$,
the image $\Xi(\gamma)=(\beta_{a})_{a\in \mathsf{Q}_{0}}$ is given by
\begin{align*}
	\beta_{[i,j]}&=\sum_{\left\{\mathbf{i}=([i,j_{i}])\in \mathcal{J}\mid 
	j_{i}=j\right\}}\gamma_{c_{\mathbf{i}}},\\
	\beta_{[i,j,k]}&=\gamma_{c_{[i,j,k]}}.
\end{align*}
\begin{prop}[Theorem 3.6 in \cite{H}]\label{quotientmap}
	We have the following.
	\begin{enumerate}
		\item The map $\Xi$ is an isometry, that is, 
			$(\gamma,\gamma')=(\Xi(\gamma),\Xi(\gamma'))$
			for any $\gamma,\gamma'\in \widehat{\mathcal{L}}$.
		\item The map $\Xi$ is injective   
			if and only if 
			\[
				\#\{i\in \{0,\ldots,p\}\mid 
				m^{(i)}>1,\,i=0,\ldots,p\}\le 1.
			\]
		\item The map $\Xi$ is $W^{\text{mc}}$-equivariant, that is, 
			for $\gamma\in \widehat{\mathcal{L}}$ and $a\in \mathcal{J}
			\cup \mathsf{Q}_{0}^{\text{leg}}$, we have 
			\[
				\Xi(s_{a}(\gamma))=s_{a}(\Xi(\gamma)).
			\]
	\end{enumerate}
\end{prop}
This proposition tells us that
$\widehat{\mathcal{L}}$ is a ``lift'' of 
$\mathcal{L}$
to a Kac-Moody root lattice with 
the Weyl group $W^{\mathrm{mc}}$.

The kernel of $\Xi$ is a big space in general.
Thus if we consider the inverse image of an element $\beta\in \mathcal{L}$,
it is convenient to restrict $\Xi$ to some smaller space as follows.
Fix $\beta\in \mathcal{L}$ and 
set  $\mathcal{J}_{\beta}:=\{([i,j_{i}])\in \mathcal{J}\mid 
\beta_{[i,j_{i}]}\neq 0\text{ for all }i\in I_{\text{irr}}\}$
and $(\mathsf{Q}_{0}^{\text{leg}})_{\beta}:=\mathsf{Q}_{0}^{\text{leg}}\cap
\supp(\beta)$. Then define
\[
	(\mathcal{J}\cup
\mathsf{Q}_{0}^{\text{leg}})_{\beta}\\
:=\mathcal{J}_{\beta}\cup (\mathsf{Q}_{0}^{\text{leg}})_{\beta}
\]
and 
a sublattice and subgroup
\begin{align*}
	\widehat{\mathcal{L}}_{\beta}&:=
\sum_{\{a\in (\mathcal{J}\cup
		\mathsf{Q}_{0}^{\text{leg}})_{\beta}\}}\mathbb{Z}c_{a},\\
		W_{\beta}^{\mathrm{mc}}&:=\langle s_{a}\mid a\in (\mathcal{J}\cup
		\mathsf{Q}_{0}^{\text{leg}})_{\beta}\rangle.
\end{align*}
Denote the set of all positive elements in $\widehat{\mathcal{L}}_{\beta}$
by $\widehat{\mathcal{L}}^{+}_{\beta}$.
We write the restriction of $\Xi$ on $\widehat{\mathcal{L}}_{\beta}$
by $\Xi_{\beta}$.

\subsubsection{Finiteness of spectral types}
As we saw in Theorem \ref{reduction},
quiver varieties $\mathfrak{M}_{\lambda}
(\mathsf{Q},\alpha)^{\text{dif}}$ 
with $\alpha\in \tilde{F}$ are 
fundamental elements under the action
of $W^{\mathrm{mc}}$.
We shall see that a kind of finiteness 
of the set $\tilde{F}$.
First let us introduce the shape
of $\beta\in \mathcal{L}$.
\begin{df}[shape]\normalfont 
	Fix a Kac-Moody root lattice $L=\bigoplus_{i\in I}\mathbb{Z}\alpha_{i}$ 
	and $\alpha=\sum_{i\in I}m_{i}\alpha_{i}\in L$.
For the Dynkin diagram of the support of $\alpha$, we attach each coefficient $m_{i}$ of $\alpha$ to the vertex corresponding 
to $\alpha_{i}$, then  we obtain the diagram with the coefficients, which
we call the {\em shape} of $\alpha$.
\end{df}
For example, if $\alpha=m_{1}\alpha_{i_1}+m_{2}\alpha_{i_2}
+m_{3}\alpha_{i_3}\in L$ with the diagram of the support
$
\begin{xy}
\ar@{-} *++!U{\alpha_{i_1}}*\cir<4pt>{};(7,0) *++!U{\alpha_{i_2}}*\cir<4pt>{}="A",
\ar@{-} "A"; (14,0) *++!U{\alpha_{i_3}}*\cir<4pt>{}
\end{xy}$,
the diagram with coefficients is 
$
\begin{xy}
    \ar@{-} *++!D{m_{1}}*++!U{\alpha_{i_1}}*\cir<4pt>{}
      ;(7,0) *++!D{m_{2}}*++!U{\alpha_{i_2}}*\cir<4pt>{}="A",
    \ar@{-} "A"; (14,0) *++!D{m_{3}}*++!U{\alpha_{i_3}}*\cir<4pt>{}
\end{xy}
$.

By using this we define  shapes of 
elements in $\mathcal{L}$ as follows.
\begin{df}\normalfont
	For $\beta\in \mathcal{L}$,
	the {\em shape} of $\beta$ is 
	the set of shapes of 
	elements in $\Xi_{\beta}^{-1}(\beta)\subset \widehat{\mathcal{L}}_{\beta}$.
\end{df}
\begin{exa}\label{exmp}\normalfont
	For example, suppose $p=1$, $m^{(0)}=m^{(1)}=2$, $e_{[i,j]}=1\ (i=0,1$ and $j=1,2)$, $d_0(1,2)=d_1(1,2)=0$.
	Consider $\beta=\epsilon_{[0,1]}+\epsilon_{[0,2]}+\epsilon_{[1,1]}+
	\epsilon_{[1,2]}$.
Then the shape of $\beta$ is 
\begin{equation*}
\begin{xy}
    \ar@{=}(0,-4)*+!R+!D{1-a}*\cir<4pt>{};(7,3) *++!D{1-a}*\cir<4pt>{},
    \ar@{=}(0,3)*+!R+!D{a}*\cir<4pt>{};(7,-4) *+!L+!D{a}*\cir<4pt>{}
\end{xy}
\quad(a\in \mathbb{Z}),
\end{equation*}
where we simply denote $\{x_{a}\mid a\in \mathbb{Z}\}$ by $x_{a}\ (a\in \mathbb{Z})$. 

Suppose $p=0$, $m^{(0)}=4$, $d_0(i,j)=1$ for $1\le i<j\le 4$ and $e_{0,\nu}=1$
for $1\le\nu\le 4$.  If $m_{0,j,1}=1$ for $1\le j\le 4$,
the shape of $\beta=\sum_{\nu=1}^{4}\epsilon_{[0,\nu]}$ is  
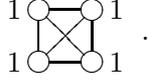
\begin{equation*}
\begin{xy}
    \ar@{-}(0,-4)*++!R{1}*\cir<4pt>{}="A";(7,3) *++!L{1}*\cir<4pt>{}="B",
    \ar@{-}(0,3)*++!R{1}*\cir<4pt>{}="C";(7,-4) *++!L{1}*\cir<4pt>{}="D",
    \ar@{-} "A";"C",
    \ar@{-} "A";"D",
    \ar@{-} "B";"C",
    \ar@{-} "B";"D"
\end{xy}.
\end{equation*}
\end{exa}

We say that $\beta\in \mathbb{Z}_{\ge 0}^{\mathsf{Q}_{0}}$ is \textit{reduced} if 
it never happens that there 
exists $i\in \{1,\ldots,p\}$ such that 
$\#\{j\mid \beta_{[i,j]}\neq 0\}=1$ and 
$e_{[i,j_{i}]}=1$ where $j_{i}\in \{j\mid \beta_{[i,j]}\neq 0\}$.
\begin{rem}\normalfont
	If $\mathfrak{M}(\mathbf{B})\cong \mathfrak{M}_{\lambda}(\mathsf{Q},
	\alpha)^{\text{dif}}\neq \emptyset$, then we may assume $\alpha$ is 
	reduced. For if there exists $i\in \{1,\ldots,p\}$ such that 
	$\{j_{i}\}=
\{j\mid \beta_{[i,j]}\neq 0\}$ and 
$e_{[i,j_{i}]}=1$, then the HTL normal form $B^{(i)}$ is a scalar matrix,
i.e., $B^{(i)}=b^{(i)}(z^{-1})I_{n}$, $b^{(i)}(z)\in \mathbb{C}[z]$.
Thus applying $\mathrm{Add}^{(i)}_{b^{(i)}(z^{-1})}$, the singular point $a_{i}$
becomes regular.
\end{rem}

Let us consider the set of all nonempty moduli 
spaces $\mathfrak{M}(\mathbf{B})$. 
Set 
\begin{align*}
	\mathrm{Ht}^{(n)}&:=\left\{(B_{i})\in \bigoplus^{\infty}
		M(n,\mathbb{C}[z^{-1}])\ \middle|\ 
	\text{ all }B_{i}\text{ are HTL normal forms}\right\}\\
	\mathrm{Ht}&:=\bigcup_{n=1}^{\infty}\mathrm{Ht}^{(n)}.
\end{align*}
\begin{df}[fundamental spectral type]\normalfont
	Let $\mathbf{m}$ be a spectral type.
	We say that $\mathbf{m}$ is {\em effective}
	if there exists $\mathbf{B}\in \mathrm{Ht}$ such that 
	$\mathfrak{M}(\mathbf{B})\cong 
	\mathfrak{M}_{\lambda}(\mathsf{Q},\alpha)^{\text{dif}}\neq 
	\emptyset$ and $\mathbf{m}=\mathbf{m}_{\alpha}$.
	A spectral type $\mathbf{m}=\mathbf{m}_{\alpha}$
	is said to be  
	{\em basic} if $\alpha\in \tilde{F}$.
	Also we say that $\mathbf{m}$ is {\em reduced} 
	if $\alpha$ is reduced.
	We say that $\mathbf{m}$ is {\em fundamental} if 
	$\mathbf{m}$ is effective, basic and reduced.
	By the {\em shape} of $\mathbf{m}$, we mean the shape of $\alpha$.
\end{df}
Then we can show the following finiteness of basic spectral types.
\begin{thm}[Theorem 8 in \cite{HO}]\label{finitespec}
	Let us fix an integer $q\in 2\mathbb{Z}_{\le 0}$. Then 
	there exist only finite number of  fundamental  
	spectral types $\mathbf{m}$ satisfying
	$\mathrm{idx}\mathbf{m}=q$.  
\end{thm}

Let us see the cases $q=0$ and $-2$ for example.
The first case is $q=0$.
\begin{thm}[Theorem 9 in \cite{HO}]\label{tame}
	Shapes of fundamental spectral types $\mathbf{m}$ 
	satisfying $\mathrm{idx}\mathbf{m}=0$ 
	are  one of the following.
    
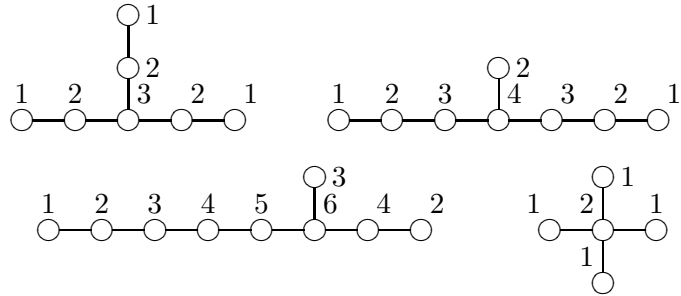
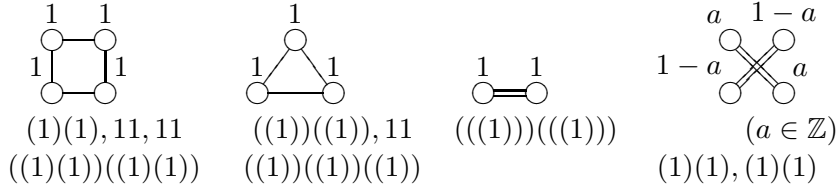
\begin{gather*}
    \begin{xy}
        *++!D{1}*\cir<4pt>{}="A";
        \ar@{-} "A"; (7,0) *++!D{2}*{}*\cir<4pt>{}="B",
        \ar@{-} "B";(14,0) *+!L+!D{3}*\cir<4pt>{}="C",
        \ar@{-} "C";(21,0) *+!L+!D{2}*\cir<4pt>{}="D",
        \ar@{-} "D";(28,0) *+!L+!D{1}*\cir<4pt>{}
	\ar@{-} "C";(14,7) *++!L{2}*\cir<4pt>{}="E",
        \ar@{-} "E";(14,14) *++!L{1}*\cir<4pt>{}
    \end{xy}
    \quad
    \begin{xy}
        *++!D{1}*\cir<4pt>{}="A";
        \ar@{-} "A"; (7,0) *++!D{2}*\cir<4pt>{}="B",
        \ar@{-} "B";(14,0) *++!D{3}*\cir<4pt>{}="C",
        \ar@{-} "C";(21,0) *+!L+!D{4}*\cir<4pt>{}="D",
        \ar@{-} "D";(28,0) *+!L+!D{3}*\cir<4pt>{}="E",
        \ar@{-} "D";(21,7) *++!L{2}*\cir<4pt>{},
        \ar@{-} "E";(35,0) *+!L+!D{2}*\cir<4pt>{}="F",
        \ar@{-} "F";(42,0) *+!L+!D{1}*\cir<4pt>{}
    \end{xy}\allowdisplaybreaks\\
    \begin{xy}
        *++!D{2}*\cir<4pt>{}="A";
        \ar@{-} "A"; (7,0) *++!D{3}*\cir<4pt>{}="B",
        \ar@{-} "B";(14,0) *++!D{4}*\cir<4pt>{}="C",
        \ar@{-} "C";(21,0) *++!D{5}*\cir<4pt>{}="D",
        \ar@{-} "D";(28,0) *+!L+!D{6}*\cir<4pt>{}="E",
        \ar@{-} "E";(35,0) *+!L+!D{4}*\cir<4pt>{}="F",
        \ar@{-} "F";(42,0) *+!L+!D{2}*\cir<4pt>{},
        \ar@{-} "A";(-7,0) *++!D{1}*\cir<4pt>{}="F",
        \ar@{-} "E";(28,7) *++!L{3}*\cir<4pt>{},
    \end{xy}\quad
    \begin{xy}
        \ar@{-} *+!R+!D{2}*\cir<4pt>{}="A";(7,0) *++!D{1}*\cir<4pt>{},
        \ar@{-} "A"; (0,7) *++!L{1}*\cir<4pt>{},
        \ar@{-} "A"; (-7,0) *+!R+!D{1}*\cir<4pt>{},
	\ar@{-} "A"; (0,-7) *+!R+!D{1}*\cir<4pt>{}
    \end{xy}\allowdisplaybreaks\\
    \begin{xy}
        \ar@{-} *+!R+!D{1}*\cir<4pt>{}="A"; (7,0) *+!L+!D{1}*\cir<4pt>{}="B",
        \ar@{-} (0,7) *++!D{1}*\cir<4pt>{}="C";(7,7) *++!D{1}*\cir<4pt>{}="D",
        \ar@{-} "A";"C",
        \ar@{-} "B";"D"
        \ar@{} (7,-5) *{(1)(1),11,11}
        \ar@{} (7,-10) *{( (1)(1))( (1)(1))}
    \end{xy}\quad
    \begin{xy}
        \ar@{-}*++!D{1}*\cir<4pt>{}="A";(10,0)*++!D{1}*\cir<4pt>{}="B",
        \ar@{-}"A";(5,7)*++!D{1}*\cir<4pt>{}="C",
        \ar@{-}"B";"C"
        \ar@{} (10,-5) *{( (1))( (1)),11}
        \ar@{} (10,-10) *{( (1))( (1))( (1))}
    \end{xy}\quad
    \begin{xy}
        \ar@{=}*++!D{1}*\cir<4pt>{};(7,0)*++!D{1}*\cir<4pt>{}
        \ar@{} (7,-5) *{( ( (1)))( ( (1)))}
    \end{xy}\ 
    \begin{xy}
        \ar@{=}(0,0)*+!R+!D{1-a}*\cir<4pt>{};(7,7) *++!D{1-a}*\cir<4pt>{},
        \ar@{=}(0,7)*+!R+!D{a}*\cir<4pt>{};(7,0) *+!L+!D{a}*\cir<4pt>{},
        \ar@{} (8,-5) *{(a\in \mathbb{Z})}
        \ar@{} (1,-10) *{(1)(1),(1)(1)}
    \end{xy}
    \end{gather*}
     We simply write sets $\{x_{a}\mid a\in \mathbb{Z}\}$ and $\{x\}$ by $x_{a}\,(a\in\mathbb{Z})$ and $x$, respectively.
    For  the first $4$ star shaped graphs, corresponding spectral types are 
    given in Remark \ref{starspectral} below.
\end{thm}
If $\mathfrak{M}_{\lambda}(\mathsf{Q},\alpha)^{\text{dif}}\neq \emptyset$ and  $\alpha\in \tilde{F}$ with $q(\alpha)=0$,  
	then by the above list of shapes of $\alpha$, 
	we can check that $\alpha$ is invariant under 
	$W^{\text{mc}}_{\alpha}$, i.e., $w(\alpha)=\alpha$ for 
	any $w\in W_{\alpha}^{\mathrm{mc}}$. 
	Then
	\[	
		s_{a}\colon
		\mathfrak{M}_{\lambda}(\mathsf{Q},\alpha)^{\text{dif}}
		\longrightarrow
		\mathfrak{M}_{r_{a}(\lambda)}(\mathsf{Q},\alpha)^{\text{dif}}
	\]
	for each 
	$a\in (\mathcal{J}\cup \mathsf{Q}_{0}^{\text{leg}})_{\alpha}$
	defines a 
	$W_{\alpha}^{\mathrm{mc}}$-action on the parameter space 
	\[ \begin{array}{lccc}
			r_{a}\colon&
			\sum_{a\in 
		(\mathcal{J}\cup \mathsf{Q}_{0}^{\text{leg}})_{\alpha}}\mathbb{C}c_{a}&\longrightarrow &	\sum_{a\in 
		(\mathcal{J}\cup \mathsf{Q}_{0}^{\text{leg}})_{\alpha}}\mathbb{C}c_{a}\\
	&\lambda&\longmapsto&r_{a}(\lambda)\end{array},
	\]
	see also Proposition 3.7 in \cite{H}.
	Here if $\lambda_{a}=0$, i.e., $s_{a}$ on $\mathfrak{M}_{\lambda}(\mathsf{Q},\alpha)^{\text{dif}}$ is not well-defined,
	we formally set $s_{a}=\mathrm{id}$ and $r_{a}=\mathrm{id}$.
	By the above theorem, $W_{\alpha}^{\mathrm{mc}}$ is isomorphic to 
	one of the Weyl groups of the 
	following types,
	\[
		E^{(1)}_{8},\ 
		E^{(1)}_{7},\
		E^{(1)}_{6},\
		D^{(1)}_{4},\
		A^{^(1)}_{3},\ 
		A^{(1)}_{2},\
		A^{(1)}_{1},\
		A^{(1)}_{1}\times
		A^{(1)}_{1}.
	\]

	\begin{rem}\normalfont\label{starspectral}
    In the above list of shapes, 
    we omit the spectral types for star-shaped diagrams. 
    For these cases spectral types are obtained as follows.
    Consider a shape
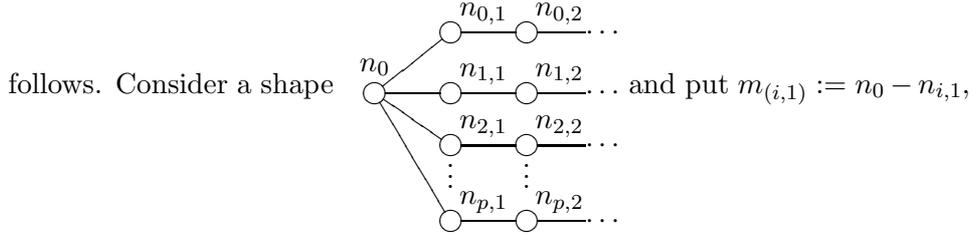
\begin{xy}
\ar@{-}               *++!D{\text{$n_0$}}  *\cir<4pt>{}="O";
             (10,0)   *+!L!D{\text{$n_{1,1}$}} *\cir<4pt>{}="A",
\ar@{-} "A"; (20,0)   *+!L!D{\text{$n_{1,2}$}} *\cir<4pt>{}="B",
\ar@{-} "B"; (30,0)   *{\cdots}, 
\ar@{-} "O"; (10,-7)  *+!L!D{\text{$n_{2,1}$}} *\cir<4pt>{}="C",
\ar@{-} "C"; (20,-7)  *+!L!D{\text{$n_{2,2}$}} *\cir<4pt>{}="E",
\ar@{-} "E"; (30,-7)  *{\cdots}
\ar@{-} "O"; (10,8)   *+!L!D{\text{$n_{0,1}$}} *\cir<4pt>{}="D",
\ar@{-} "D"; (20,8)   *+!L!D{\text{$n_{0,2}$}} *\cir<4pt>{}="F",
\ar@{-} "F"; (30,8)   *{\cdots}
\ar@{} (10,-10) *{\vdots}
\ar@{} (20,-10) *{\vdots}
\ar@{-} "O"; (10,-17) *+!L!D{\text{$n_{p,1}$}} *\cir<4pt>{}="G",
\ar@{-} "G"; (20,-17) *+!L!D{\text{$n_{p,2}$}} *\cir<4pt>{}="H",
\ar@{-} "H"; (30,-17) *+!U{\phantom*}*{\cdots}
\end{xy}
and put $m_{(i,1)}:=n_0-n_{i,1}$, $m_{(i,j+1)}:=n_{i,j}-n_{i,j+1}$, 
$m_{(i,0)}:=\sum_{\substack{0\le k\le p\\ k\neq i}}n_{k,1}-n_{0}$ 
and $m_{(0)}:=\sum_{i=0}^{p}n_{i,1}-n_{0}$. 
Then the shape corresponds to the following $5$ types.
\begin{align*}
 &m_{(0,1)}m_{(0,2)}\ldots,\,m_{(1,1)}m_{(1,2)}\ldots,\,\ldots,\,m_{(p,1)}m_{(p,2)}\ldots,\\
 &m_{(0)}n_{0},\,(m_{(0,2)}m_{(0,3)}\ldots)\ldots(m_{(p,2)}m_{(p,3)}\ldots),\\
 &m_{(i,0)}m_{(i,1)}\ldots,\,(m_{(0,2)}m_{(0,3)}\ldots)\ldots(m_{(i-1,2)}\ldots)(m_{(i+1,2)}\ldots)\ldots,\\
 &((m_{(i,1)}m_{(i,2)}\ldots))((m_{(0,2)}m_{(0,3)}\ldots)\ldots(m_{(i-2,2)}\ldots)(m_{(i+1,2)}\ldots)\ldots),\\
 &((n_0))((m_{(0,2)}m_{(0,3)}\ldots)\ldots(m_{(p,2)}m_{(p,3)}\ldots)).
\end{align*}
\end{rem}

Next let us see the case $q=-2$.
\begin{thm}[Theorem 10 in \cite{HO}]\label{class}
	Shapes of fundamental spectral types $\mathbf{m}$ satisfying 
	$\mathrm{idx\,}\mathbf{m}=-2$ are 
	one of the following.
    \begin{gather*}
    \begin{xy}
        \ar@{-} (0,3.5) *+!R+!D{a}*{\cdot}*\cir<4pt>{}="A";(7,3.5)*+!L+!D{1-a}*{\cdot}*\cir<4pt>{}="B",
        \ar@{-} (0,-3.5) *+!R+!D{1-a}*{\cdot}*\cir<4pt>{}="C";(7,-3.5) *+!L+!D{a}*{\cdot}*\cir<4pt>{}="D",
        \ar@3{-} "A";"D",
        \ar@3{-} "B";"C"
        \ar@{} (3,-8) *{( (1))( (1)), (1)(1)}
	\ar@{} (3,-13) *{W^{\text{inv}}=\emptyset}
    \end{xy}\hspace{-.3cm}
     (a\in \mathbb{Z})\qquad
    \begin{xy}
        \ar@{=} (0,5)*++!D{1-a}*{\cdot}*\cir<4pt>{}="A";(0,-5) *+!R+!D{1-a}*{\cdot}*\cir<4pt>{}="B",
        \ar@{-} "A"; (7,0) *++!D{1}*\cir<4pt>{}="C",
        \ar@{-} "B";"C", 
        \ar@{=} (14,5)*++!D{a}*{\cdot}*\cir<4pt>{}="D";(14,-5) *+!L+!D{a}*{\cdot}*\cir<4pt>{}="E",
        \ar@{-} "D";"C",
        \ar@{-} "E";"C",
        \ar@{} (5,-10) *{(1)(1),(1)(1),11}
	\ar@{} (5,-15) *{W^{\text{inv}}=A_{1}}
    \end{xy}
     (a\in \mathbb{Z})\allowdisplaybreaks\\
    \begin{xy}
        \ar@{=} (0,-6)*+!R+!D{2-a}*{\cdot}*\cir<4pt>{}="A";(10,-6) *+!L+!D{2-a}*{\cdot}*\cir<4pt>{}="B",
        \ar@{-} (0,2)*+!R+!D{1}*\cir<4pt>{}="C";(5,2)*+!R+!D{a}*\cir<4pt>{}="D",
        \ar@{-} "D"; (10,2) *+!L+!D{1}*\cir<4pt>{}="E",
        \ar@{=} "D";(5,9) *++!D{a-1}*{\cdot}*\cir<4pt>{},
        \ar@{-} "A";"C",
        \ar@{-} "B";"E"
        \ar@{} (5,-11) *{(1)(11),(1)(11)}
	\ar@{} (5,-16) *{W^{\text{inv}}=A_{3}}
    \end{xy}
    (a\in\mathbb{Z})\qquad
    \begin{xy}
        \ar@{=}*+!R+!D{2-a}*\cir<4pt>{};(7,0)*++!D{2-a}*{\cdot}*\cir<4pt>{}="A",
        \ar@{-} "A";(14,0)*++!D{1}*\cir<4pt>{}="B",
        \ar@{-} "B";(21,0) *++!D{a}*{\cdot}*\cir<4pt>{}="C",
        \ar@{=} "C";(28,0) *++!D{a}*\cir<4pt>{}
        \ar@{} (14,-5) *{(2)(2),(2)(11)}
	\ar@{} (14,-10) *{W^{\text{inv}}=A_1\times A_{1}\times A_{1}}
    \end{xy}
    \ \ (a\in\mathbb{Z})\allowdisplaybreaks\\
    \begin{xy}
        (0,5)*++!D{a}*\cir<4pt>{}="A", (7,5) *++!D{b}*\cir<4pt>{}="B",
        (14,5)*+!L+!D{2-a-b}*\cir<4pt>{}="C", 
        (0,-5)*+!R+!U{1-a}*{\cdot}*\cir<4pt>{}="A'", (7,-5) *++!U{1-b}*{\cdot}*\cir<4pt>{}="B'",
        (14,-5)*+!L+!U{a+b-1}*{\cdot}*\cir<4pt>{}="C'",
        \ar@{=} "A";"B'",
        \ar@{=} "A";"C'",
        \ar@{=} "B";"A'",
        \ar@{=} "B";"C'",
        \ar@{=} "C";"A'",
        \ar@{=} "C";"B'"
        \ar@{} (7,-13) *{(1)(1)(1),(2)(1)}
	\ar@{} (7,-18) *{W^{\text{inv}}=A_{1}\times A_{1}\times A_{1}}
    \end{xy}
    \hspace{-5mm}(a,\,b\in\mathbb{Z})
\quad
    \begin{xy}
        \ar@{=}*+!R+!D{a-1}*{\cdot}*\cir<4pt>{};(7,0)*++!D{a}*\cir<4pt>{}="A",
        \ar@{-} "A";(14,0)*+!R+!D{2}*\cir<4pt>{}="B",
        \ar@{-} "B";(21,0) *++!D{3-a}*\cir<4pt>{}="C",
        \ar@{=} "C";(28,0) *+!L+!D{2-a}*{\cdot}*\cir<4pt>{},
        \ar@{-} "B"; (14,7) *++!D{1}*\cir<4pt>{},
        \ar@{} ;(21,-5) *{(a\in\mathbb{Z})}
        \ar@{} (14,-10) *{(2)(2),(1)(111)}
	\ar@{} (14, -15)*{W^{\text{inv}}=D_{4}}
    \end{xy}\allowdisplaybreaks\\
     \begin{xy}
        \ar@3{-} *++!D{1}*{\cdot}*\cir<4pt>{};(7,0)*++!D{1}*{\cdot}*\cir<4pt>{}.
        \ar@{} (4,-5) *{( ( ( (1))))( ( ( (1))))}
	\ar@{} (4,-10) *{W^{\text{inv}}=\emptyset}
    \end{xy}\quad
    \begin{xy}
        \ar@{=}*++!D{2}*\cir<4pt>{};(7,0) *++!D{2}*{\cdot}*\cir<4pt>{}="A",
        \ar@{-} "A";(14,0) *++!D{1}*\cir<4pt>{}
        \ar@{} (7,-5) *{( ( (2)))( ( (11)))}
	\ar@{} (7,-10) *{W^{\text{inv}}=A_{1}\times A_{1}}
    \end{xy}\quad
    \begin{xy}
        \ar@{=} *++!D{1}*\cir<4pt>{};(7,0) *++!D{1}*{\cdot}*\cir<4pt>{}="A",
        \ar@{=} "A";(14,0) *++!D{1}*\cir<4pt>{}
        \ar@{} (7,-5) *{( ( (1)(1)))( ( (1)))}
	\ar@{} (7,-10) *{W^{\text{inv}}=A_{1}\times A_{1}}
    \end{xy}\quad
    \begin{xy}
        \ar@{=}*++!D{1}*{\cdot}*\cir<4pt>{}="A";(10,0) *++!D{1}*{\cdot}*\cir<4pt>{}="B",
        \ar@{-}"A";(5,7)*++!D{1}*\cir<4pt>{}="C",
        \ar@{-}"B";"C" 
        \ar@{} (5,-5) *{( ( (1)))( ( (1))),11}
	\ar@{} (5,-10) *{W^{\text{inv}}=A_{1}}
    \end{xy}\allowdisplaybreaks\\
    \begin{xy}
        *+!R+!D{2}*{\cdot}*\cir<4pt>{}="A",
        \ar@{-} "A"; (6.66,2.17) *++!D{1}*\cir<4pt>{},
        \ar@{-} "A";(4.11,-5.66) *+!L+!D{1}*\cir<4pt>{},
        \ar@{-} "A";(0,7) *++!D{1}*\cir<4pt>{}, 
        \ar@{-} "A"; (-6.66,2.17) *++!D{1}*\cir<4pt>{},
        \ar@{-} "A";(-4.11,-5.66) *+!R+!D{1}*\cir<4pt>{},
	\ar@{} (0,-10) *{W^{\text{inv}}=(A_{1})^{5}}
    \end{xy}\ 
    \begin{xy}
        \ar@{-}*++!D{2}*\cir<4pt>{};(7,0) *+!R+!D{4}*\cir<4pt>{}="A",
        \ar@{-}"A"; (14,0)*++!D{2}*{\cdot}*\cir<4pt>{}="B",
        \ar@{-}"B";(21,0) *++!D{1}*\cir<4pt>{},
        \ar@{-} "A";(7,7)*++!D{2}*\cir<4pt>{},
        \ar@{-} "A";(7,-7)*+!R+!D{2}*\cir<4pt>{}
	\ar@{} (12,-10) *{W^{\text{inv}}=D_{4}\times A_{1}}
    \end{xy}\ 
    \begin{xy}
        \ar@{-}*++!D{2}*\cir<4pt>{}="F";(-7,0)*++!D{1}*\cir<4pt>{},
        \ar@{-}"F";(7,0) *+!R+!D{3}*\cir<4pt>{}="A",
        \ar@{-}"A"; (14,0)*++!D{2}*\cir<4pt>{}="B",
        \ar@{-}"B";(21,0) *++!D{1}*\cir<4pt>{},
        \ar@{-} "A";(7,7)*++!D{1}*{\cdot}*\cir<4pt>{},
        \ar@{-} "A";(7,-7)*+!R+!D{1}*{\cdot}*\cir<4pt>{}
	\ar@{} (9,-10) *{W^{\text{inv}}=A_{5}}
    \end{xy}\ 
    \begin{xy}
        \ar@{-}*++!D{1}*{\cdot}*\cir<4pt>{};(7,0) *+!R+!D{4}*\cir<4pt>{}="A",
        \ar@{-}"A"; (14,0)*++!D{3}*\cir<4pt>{}="B",
        \ar@{-}"B";(21,0) *++!D{2}*\cir<4pt>{}="D",
        \ar@{-} "D";(28,0) *++!D{1}*\cir<4pt>{},
        \ar@{-} "A";(7,7)*++!D{2}*\cir<4pt>{}="C",
        \ar@{-} "A";(7,-7)*+!R+!D{2}*\cir<4pt>{}
	\ar@{} (9,-11) *{W^{\text{inv}}=D_{6}}
    \end{xy}\allowdisplaybreaks\\
    \begin{xy}
        *++!D{2}*{\cdot}*\cir<4pt>{}="A";
        \ar@{-} "A"; (7,0) *++!D{4}*\cir<4pt>{}="B",
        \ar@{-} "B";(14,0) *+!L+!D{6}*\cir<4pt>{}="C",
        \ar@{-} "C";(21,0) *++!D{4}*\cir<4pt>{}="D",
        \ar@{-} "D";(28,0) *++!D{2}*\cir<4pt>{}
        \ar@{-} "C";(14,7) *+!L+!D{4}*\cir<4pt>{}="E",
        \ar@{-} "E";(14,14) *++!L{2}*\cir<4pt>{},
        \ar@{-} "A";(-7,0) *++!D{1}*\cir<4pt>{}
	\ar@{} (9,-4) *{W^{\text{inv}}=E_{6}\times A_{1}}
    \end{xy}\quad
    \begin{xy}
        *++!D{1}*\cir<4pt>{}="A";
        \ar@{-} "A"; (7,0) *++!D{2}*\cir<4pt>{}="B",
        \ar@{-} "B";(14,0) *++!D{3}*\cir<4pt>{}="C",
        \ar@{-} "C";(21,0) *+!L+!D{4}*\cir<4pt>{}="D",
        \ar@{-} "D";(28,0) *++!D{3}*\cir<4pt>{}="E",
        \ar@{-} "D";(21,7) *+!L+!D{2}*{\cdot}*\cir<4pt>{}="G",
        \ar@{-} "E";(35,0) *++!D{2}*\cir<4pt>{}="F",
        \ar@{-} "F";(42,0) *++!D{1}*\cir<4pt>{}
        \ar@{-} "G";(21,14) *++!L{1}*\cir<4pt>{}
	\ar@{} (15,-4) *{W^{\text{inv}}=A_{7}\times A_{1}}
    \end{xy}\\
    \begin{xy}
        *++!D{2}*{\cdot}*\cir<4pt>{}="A";
        \ar@{-} "A";(-7,0) *++!D{1}*\cir<4pt>{},
        \ar@{-} "A"; (7,0) *++!D{4}*\cir<4pt>{}="B",
        \ar@{-} "B";(14,0) *++!D{6}*\cir<4pt>{}="C",
        \ar@{-} "C";(21,0) *+!L+!D{8}*\cir<4pt>{}="D",
        \ar@{-} "D";(28,0) *++!D{6}*\cir<4pt>{}="E",
        \ar@{-} "D";(21,7) *++!D{4}*\cir<4pt>{}="G",
        \ar@{-} "E";(35,0) *++!D{4}*\cir<4pt>{}="F",
        \ar@{-} "F";(42,0) *++!D{2}*\cir<4pt>{}
	\ar@{} (20,-4) *{W^{\text{inv}}=E_{7}\times A_{1}}
    \end{xy}\ 
    \begin{xy}
        *++!D{2}*\cir<4pt>{}="A";
        \ar@{-} "A"; (7,0) *++!D{3}*\cir<4pt>{}="B",
        \ar@{-} "B";(14,0) *++!D{4}*\cir<4pt>{}="C",
        \ar@{-} "C";(21,0) *++!D{5}*\cir<4pt>{}="D",
        \ar@{-} "D";(28,0) *+!L+!D{6}*\cir<4pt>{}="E",
        \ar@{-} "E";(35,0) *++!D{4}*\cir<4pt>{}="F",
        \ar@{-} "F";(42,0) *++!D{2}*{\cdot}*\cir<4pt>{}="G",
        \ar@{-} "A";(-7,0) *++!D{1}*\cir<4pt>{},
        \ar@{-} "G";(49,0) *++!D{1}*\cir<4pt>{},
        \ar@{-} "E";(28,7) *++!D{3}*\cir<4pt>{},
	\ar@{} (20,-4) *{W^{\text{inv}}=D_{8}\times A_{1}}
    \end{xy}\\
    \begin{xy}
        *++!D{4}*\cir<4pt>{}="A";
        \ar@{-} "A"; (7,0) *++!D{6}*\cir<4pt>{}="B",
        \ar@{-} "B";(14,0) *++!D{8}*\cir<4pt>{}="C",
        \ar@{-} "C";(21,0) *++!D{10}*\cir<4pt>{}="D",
        \ar@{-} "D";(28,0) *+!L+!D{12}*\cir<4pt>{}="E",
        \ar@{-} "E";(35,0) *!L++!D{8}*\cir<4pt>{}="F",
        \ar@{-} "F";(42,0) *++!D{4}*\cir<4pt>{},
        \ar@{-} "A";(-7,0) *++!D{2}*{\cdot}*\cir<4pt>{}="F",
        \ar@{-} "E";(28,7) *++!D{6}*\cir<4pt>{},
        \ar@{-} "F";(-14,0) *++!D{1}*\cir<4pt>{},
	\ar@{} (20,-4) *{W^{\text{inv}}=E_{8}\times A_{1}}
    \end{xy}\ 
    \begin{xy}
        *++!D{1} *\cir<4pt>{}="A";
        \ar@{-} "A";(7,0) *++!D{2}*\cir<4pt>{}="A";
        \ar@{-} "A";(14,0) *++!D{3}*\cir<4pt>{}="A";
        \ar@{-} "A";(21,0) *++!D{4}*\cir<4pt>{}="A";
        \ar@{-} "A";(28,0) *+!L+!D{5}*\cir<4pt>{}="C";
        \ar@{-} "C";(35,0) *++!D{4}*\cir<4pt>{}="A";
        \ar@{-} "A";(42,0) *++!D{3}*\cir<4pt>{}="A";
        \ar@{-} "A";(49,0) *++!D{2}*\cir<4pt>{}="A";
        \ar@{-} "A";(56,0) *++!D{1}*\cir<4pt>{}
        \ar@{-} "C";(28,7) *++!D{2}*{\cdot}*\cir<4pt>{}
	\ar@{} (20,-4) *{W^{\text{inv}}=A_{9}}
    \end{xy}\\ 
    \begin{xy}
        *++!D{1} *\cir<4pt>{}="A";
        \ar@{-} "A";(7,0) *++!D{2}*\cir<4pt>{}="A";
        \ar@{-} "A";(14,0) *++!D{3}*\cir<4pt>{}="A";
        \ar@{-} "A";(21,0) *++!D{4}*\cir<4pt>{}="A";
        \ar@{-} "A";(28,0) *++!D{5}*\cir<4pt>{}="A";
        \ar@{-} "A";(35,0) *++!D{6}*\cir<4pt>{}="A";
        \ar@{-} "A";(42,0) *++!D{7}*\cir<4pt>{}="A";
        \ar@{-} "A";(49,0) *+!L+!D{8}*\cir<4pt>{}="C";
        \ar@{-} "C";(49,7) *++!D{4}*\cir<4pt>{};
        \ar@{-} "C";(56,0) *++!D{5}*\cir<4pt>{}="C";
        \ar@{-} "C";(63,0) *++!D{2}*{\cdot}*\cir<4pt>{}
	\ar@{} (20,-4) *{W^{\text{inv}}=D_{10}}
    \end{xy}\
    \begin{xy}
        *++!D{1}*{\cdot}*\cir<4pt>{}="A",
        \ar@{-} "A";(7,0) *++!D{4}*\cir<4pt>{}="B",
        \ar@{-} "B";(14,0) *++!D{7}*\cir<4pt>{}="D",
        \ar@{-} "D";(21,0) *+!L+!D{10}*\cir<4pt>{}="G",
        \ar@{-} "G";(28,0) *++!D{8}*\cir<4pt>{}="H",
        \ar@{-} "H";(35,0) *++!D{6}*\cir<4pt>{}="I",
        \ar@{-} "I";(42,0) *++!D{4}*\cir<4pt>{}="J",
        \ar@{-} "J";(49,0) *++!D{2}*\cir<4pt>{},
        \ar@{-} "G";(21,7) *++!D{5}*\cir<4pt>{}
	\ar@{} (20,-4) *{W^{\text{inv}}=E_{8}}
    \end{xy}\allowdisplaybreaks\\
    \begin{xy}
        *++!D{1}*{\cdot}*\cir<4pt>{}="A";
        \ar@{-} "A"  ;(7,0) *++!D{3}*\cir<4pt>{}="A";
        \ar@{-} "A";(14,0) *+!L+!D{5}*\cir<4pt>{}="B";
        \ar@{-} "B";(14,7) *+!L+!D{3}*\cir<4pt>{}="C";
        \ar@{-} "C";(14,14)*++!L{1}*{\cdot}*\cir<4pt>{};
        \ar@{-} "B";(21,0) *++!D{4}*\cir<4pt>{}="A";
        \ar@{-} "A";(28,0) *++!D{3}*\cir<4pt>{}="A";
        \ar@{-} "A";(35,0) *++!D{2}*\cir<4pt>{}="A";
        \ar@{-} "A";(42,0) *++!D{1}*\cir<4pt>{};
	\ar@{} (20,-4) *{W^{\text{inv}}=D_{7}}
    \end{xy}\quad
    \begin{xy}
        \ar@{-} (0,3.5)*++!D{2}*{\cdot}*\cir<4pt>{}="A";(5,7) *++!D{2}*\cir<4pt>{}="B",
        \ar@{-}"A";(5,0) *+!L+!D{2}*\cir<4pt>{}="C",
        \ar@{-}"C";"B",
        \ar@{-} "A";(-7,3.5)*++!D{1}*\cir<4pt>{}
        \ar@{} (0,-5) *{( (2))( (2))( (11))}
        \ar@{} (0,-10) *{( (2))( (11)),22}
        \ar@{} (0,-15) *{( (2))( (2)),211}
	\ar@{} (0,-20) *{W^{\text{inv}}=A_{2}\times A_{1}}
    \end{xy}\quad
    \begin{xy}
        \ar@{-} (0,3.5) *++!D{2}*\cir<4pt>{}="A"; (5,7) *++!D{2}*\cir<4pt>{}="B",
        \ar@{-}"A";(5,0) *+!L+!D{1}*{\cdot}*\cir<4pt>{}="C",
        \ar@{-}"C";"B",
        \ar@{-} "A";(-7,3.5)*++!D{1}*\cir<4pt>{},
        \ar@{-} "B";(12,7)*++!D{1}*\cir<4pt>{}
        \ar@{} (0,-5) *{( (11))( (11))( (1))}
        \ar@{} (0,-10) *{( (11))( (1)),111}
        \ar@{} (0,-15) *{( (11))( (11)),31}
	\ar@{} (0,-20) *{W^{\text{inv}}=A_{4}}
    \end{xy}\\
    \begin{xy}
        \ar@{-}*+!R+!D{2}*{\cdot}*\cir<4pt>{}="A";(7,0)*+!L+!D{2}*\cir<4pt>{}="B",
        \ar@{-}"A";(0,7)*++!D{2}*\cir<4pt>{}="C",
        \ar@{-}"B";(7,7)*++!D{2}*\cir<4pt>{}="D",
        \ar@{-}"C";"D",
        \ar@{-}(-7,0)*++!D{1}*\cir<4pt>{};"A"
        \ar@{} (0,-5) *{( (2)(2))( (2)(11))}
        \ar@{} (0,-10) *{ (2)(2),22,211}
        \ar@{} (0,-15) *{ (2)(11),22,22}
	\ar@{} (0,-20) *{W^{\text{inv}}=A_{3}\times A_{1}}
    \end{xy}\allowdisplaybreaks\
    \begin{xy}
        \ar@{-}*+!R+!D{2}*\cir<4pt>{}="A";(7,0)*+!L+!D{2}*\cir<4pt>{}="B",
        \ar@{-}"A";(0,7)*++!D{2}*\cir<4pt>{}="C",
        \ar@{-}"B";(7,7)*++!D{1}*{\cdot}*\cir<4pt>{}="D",
        \ar@{-}"C";"D",
        \ar@{-}(14,0)*++!D{1}*\cir<4pt>{};"B",
        \ar@{-}(-7,7)*++!D{1}*\cir<4pt>{};"C"
        \ar@{} (2,-5) *{( (11)(11))( (2)(1))}
        \ar@{} (2,-10) *{ (11)(11),22,31}
        \ar@{} (2,-15) *{(2)(1),111,111}
	\ar@{} (2,-20) *{W^{\text{inv}}=A_{5}}
    \end{xy}\ 
    \begin{xy}
        \ar@{-}*+!R+!D{2}*\cir<4pt>{}="A";(7,0)*+!L+!D{1}*{\cdot}*\cir<4pt>{}="B",
        \ar@{-}"A";(0,7)*++!D{2}*\cir<4pt>{}="C",
        \ar@{-}"B";(7,7)*++!D{1}*{\cdot}*\cir<4pt>{}="D",
        \ar@{-}"C";"D",
        \ar@{-}(-7,0)*++!D{1}*\cir<4pt>{};"A",
        \ar@{-}(-7,7)*++!D{1}*\cir<4pt>{};"C",
        \ar@{} (0,-5) *{( (11)(1))( (11)(1))}
        \ar@{} (0,-10) *{( (11)(1)),21,111}
	\ar@{} (0,-15) *{W^{\text{inv}}=A_{4}}
    \end{xy}\ 
    \begin{xy}
        \ar@{-}*+!R+!D{3}*\cir<4pt>{}="A";(7,0)*+!L+!D{2}*\cir<4pt>{}="B",
        \ar@{-}"A";(0,7)*++!D{2}*\cir<4pt>{}="C",
        \ar@{-}"B";(7,7)*++!D{1}*{\cdot}*\cir<4pt>{}="D",
        \ar@{-}"C";"D",
        \ar@{-}(-7,0)*++!D{2}*\cir<4pt>{}="E";"A",
        \ar@{-}(-14,0)*++!D{1}*\cir<4pt>{};"E"
        \ar@{} (-2,-5) *{( (1)(111))( (2)(2))}
        \ar@{} (-2,-10) *{ (1)(111),22,22}
        \ar@{} (-2,-15) *{ (2)(2),31,1111}
	\ar@{} (-2,-20) *{W^{\text{inv}}=D_{5}}
    \end{xy}\\ 
    \begin{xy}
        (-5,10)\ar@{-}*++!R{1}*{\cdot}*\cir<4pt>{}="A"; (5,10) *++!L{1}*\cir<4pt>{}="B",
        \ar@{-}"A"; (-5,0) *++!R{1}*\cir<4pt>{}="C",
        \ar@{-}"C"; (5,0) *++!L{1}*{\cdot}*\cir<4pt>{}="D",
        \ar@{-}"A"; (0,5) *+!L!D{1}*\cir<4pt>{}="E",
        \ar@{-}"E";"D",
        \ar@{-}"B";"D",
        \ar@{} (0,-5) *{( (1) (1))( (1)(1)(1))}
        \ar@{} (0,-10) *{ (1)(1),11,11,11}
        \ar@{} (0,-15) *{ (1)(1)(1),21,21}
	\ar@{} (0,-20) *{W^{\text{inv}}=(A_{1})^3}
    \end{xy}\
    \begin{xy}
        \ar@{-}*+!R+!D{1}*{\cdot}*\cir<4pt>{}="A";(7,0)*+!L+!D{1}*\cir<4pt>{}="B",
        \ar@{-}"A";(0,7) *++!D{1}*\cir<4pt>{}="C",
        \ar@{-} "C";(7,7) *++!D{1}*{\cdot}*\cir<4pt>{}="D",
        \ar@{-} "A";"D",
        \ar@{-} "B";"D"
        \ar@{} (5,-5) *{( (1))((1))( (1)(1))}
        \ar@{} (5,-10) *{( (1))( (1)),11,11}
	\ar@{} (5,-15) *{((1)(1)) ( (1)),21}
	\ar@{} (5,-20) *{W^{\text{inv}}=A_{1}\times A_{1}}
    \end{xy}
    \end{gather*}
    Here we simply denote the sets $\{x_{a}\mid a\in \mathbb{Z}\}$ and $\{y\}$ 
    by $x_{a}\ (a\in \mathbb{Z})$ and $y$, respectively. 
    For the spectral types of the star shaped graphs, see Remark 
    \ref{starspectral}.
    Here for fundamental spectral types $\mathbf{m}=\mathbf{m}_{\alpha}$
    \[
	    W^{\text{inv}}:=\langle s_{a}\mid 
	    s_{a}(\alpha)=\alpha, 
	    a\in (\mathcal{J}\cup \mathsf{Q}_{0}^{\text{leg}})_{\beta}
	    \rangle\subset W_{\alpha}^{\mathrm{mc}}.
    \]
    Plain circles in the Dynkin diagrams correspond to simple roots
    $c_{a}$ such that $s_{a}(\beta)=\beta$ and dotted circles 
    correspond to $s_{a}(\beta)\neq \beta$.
\end{thm}

As well as the case $q=0$, in the case of $q=-2$, 
$\mathfrak{M}_{\lambda}(\mathsf{Q},\alpha)^{\text{dif}}\neq \emptyset$ 
with $\alpha\in \tilde{F}$ such that $q(\alpha)=-2$
has a $W^{\text{inv}}$-action on the parameter space 
\[ \begin{array}{lccc}
			r_{a}\colon&
			\sum_{a\in 
		(\mathcal{J}\cup \mathsf{Q}_{0}^{\text{leg}})_{\alpha}}\mathbb{C}c_{a}&\longrightarrow &	\sum_{a\in 
		(\mathcal{J}\cup \mathsf{Q}_{0}^{\text{leg}})_{\alpha}}\mathbb{C}c_{a}\\
	&\lambda&\longmapsto&r_{a}(\lambda)\end{array}.
	\]
\subsection{Integrable deformations and middle convolutions}
We shall discuss symmetries of integrable deformations coming from 
symmetries of $\mathfrak{M}_{\lambda}(\mathsf{Q},\alpha)^{
\text{dif}}\cong \mathfrak{M}(\mathbf{B})$. When $\mathsf{Q}$ is a 
simply-laced quiver and $I_{\text{irr}}=\{0\}$, Boalch gave a formulation of integrable deformations 
as Hamiltonian systems over quiver varieties and discuss their symmetries
in \cite{Boa1}. The following results might be seen as 
a generalization of his work.

The theorem below  connects
the $W^{\mathrm{mc}}$-action
on $\mathfrak{M}(\mathbf{B})$ and  on integrable deformations.
The following theorem is obtained by Haraoka-Filipuk in \cite{HF} for
Fuchsain cases, 
Boalch in \cite{Bo} for simply-laced $\mathsf{Q}$ with $I_{\text{irr}}=\{0\}$
and Yamakawa 
in \cite{Y2} for general $\mathsf{Q}$.
\begin{thm}[Yamakawa. Corollary 3.17 in \cite{Y2}. cf. Haraoka-Filipuk \cite{HF} and  Boalch  \cite{Bo}]\label{integrablemiddleconv}
	Let $(\mathbf{B}(t))_{t\in\mathbb{T}}$ be an non-resonant admissible family of 
	collections of HTL normal forms which satisfies that  
	that $(B^{(0)}(t))_{\text{irr}}\equiv 0$ and $\mathrm{pr}_{\text{res}}(B^{(0)}(t))$
	is invertible.
	Let $\left((\mathcal{O}_{\mathbb{P}^{1}_{t}}^{n}\nabla_{t})
	\right)_{t\in \mathbb{T}}$ 
	be an admissible integrable family with $(\mathbf{B}(t))_{t\in
		\mathbb{T}}$ and the spectral data 
$(\mathsf{Q},\lambda,\alpha)$.
	Then for each $\mathbf{i}\in \mathcal{J}$, there exists 
	an admissible integrable deformation $
	\left((\mathcal{O}^{n}_{\mathbb{P}^{1}_{t}}\nabla^{\mathbf{i}}_{t})
	\right)_{t
	\in \mathbb{T}}$ with 
	the spectral data $(\mathsf{Q},r_{\mathbf{i}}(\lambda),
	s_{\mathbf{i}}(\alpha))$ such that $\nabla^{\mathbf{i}}_{t}\cong
	\mathrm{mc}_{\mathbf{i}}(\nabla_{t})$ for all $t\in \mathbb{T}$.
\end{thm}
\begin{proof}
	See Corollary 3.17 and (ii) in Remark 3.18 in \cite{Y2}.
\end{proof}
\begin{rem}\normalfont
	In the above theorem, we only discussed middle convolutions.
	Similarly  the relation 
	between Fourier-Laplace transformations and integrable deformations
	are also important and found in \cite{Bo}, \cite{Har}, \cite{Wod}
	and \cite{Yam}.
\end{rem}
\begin{df}\normalfont
	We say that 
	an admissible integral family $\left(
	(\mathcal{O}_{\mathbb{P}^{1}_{t}}^{n},\nabla_{t})
	\right)_{t\in \mathbb{T}}$
	with $(\mathsf{Q},\lambda,\alpha)$ is {\em fundamental} when
	$\alpha\in \tilde{F}\cap \Sigma_{\lambda}^{\text{dif}}$ 
	and $\alpha$ is reduced. 
\end{df}

Then Theorem \ref{reduction}  and Theorem \ref{integrablemiddleconv} show 
the following.
\begin{thm}
	Let $(\mathbf{B}(t))_{t\in\mathbb{T}}$ be as in Theorem \ref{integrablemiddleconv}.
	Let $\left(
	(\mathcal{O}^{n},\nabla_{t})
	\right)_{t\in \mathbb{T}}$ be an admissible integrable family 
	with $(\mathbf{B}(t))_{t\in\mathbb{T}}$ and 
	the spectral data $(\mathsf{Q},\lambda,\alpha)$ where 
	$\alpha\in \Sigma_{\lambda}^{\text{dif}}$ and $q(\alpha)\le 0$.
	Suppose that $\lambda$ is fractional and moreover has a fractional
	reduction.
	Then $\left(
	(\mathcal{O}_{\mathbb{P}^{1}_{t}},\nabla_{t})
	\right)_{t\in \mathbb{T}}$ can be reduced to a fundamental
	admissible integral deformation by a finite iteration of 
	middle convolutions and additions.
\end{thm}

For an admissible integrable family $\left(
(\mathcal{O}^{n}_{\mathbb{P}^{1}_{t}},\nabla_{t})
\right)_{t\in \mathbb{T}}$
with a spectral data 
$(\mathsf{Q},\lambda,\alpha)$,  we call $\mathbf{m}_{\alpha}$ the 
{\em spectral type} and 
also $\lambda$ the {\em spectral parameter}.
\begin{thm}
	Let us fix an integer $d\in 2\mathbb{Z}_{> 0}$
	There exists only finite spectral types of fundamental
	admissible integrable
	deformations of dimension $d$.
\end{thm}
\begin{proof}
	This directly follows from Theorems \ref{finitespec} and 
	\ref{integrablemiddleconv}.
\end{proof}

We have the classification of spectral types of admissible deformations 
of dimension $d=2$ and $4$.
\begin{thm}\label{integrableWeyl}
	Spectral types of fundamental admissible integrable 
	deformations of dimension 
	$d=2, 4$ are listed in Theorem \ref{tame} (resp. Theorem \ref{class})
	for 
	$d=2$ (resp. $d=4$).
	Moreover generic spectral parameters have $W^{\mathrm{mc}}$-actions 
	(resp. $W_{\text{inv}}$-action) for $d=2$ (resp. $d=4$).
	Here we say that a spectral parameter $\lambda\in 
	\mathbb{C}^{\mathsf{Q}_{0}}$ is {\em generic}
	when for any sequence $a_{1},\ldots,a_{l}\in \mathcal{J}
	\cup \mathsf{Q}_{0}^{\text{leg}}$, $\lambda'=r_{a_{l}}\circ 
	\cdots \circ r_{a_{2}}\circ r_{a_{1}}(\lambda)$ satisfies 
	non-resonance condition.
\end{thm}

In Theorems \ref{tame} and \ref{class}, we saw the cases where several 
different spectral types correspond to a same shape. For example,
\[ \begin{xy}
        (-5,10)\ar@{-}*++!R{1}*{\cdot}*\cir<4pt>{}="A"; (5,10) *++!L{1}*\cir<4pt>{}="B",
        \ar@{-}"A"; (-5,0) *++!R{1}*\cir<4pt>{}="C",
        \ar@{-}"C"; (5,0) *++!L{1}*{\cdot}*\cir<4pt>{}="D",
        \ar@{-}"A"; (0,5) *+!L!D{1}*\cir<4pt>{}="E",
        \ar@{-}"E";"D",
        \ar@{-}"B";"D",
        \ar@{} (0,-5) *{( (1) (1))( (1)(1)(1))}
        \ar@{} (0,-10) *{ (1)(1),11,11,11}
        \ar@{} (0,-15) *{ (1)(1)(1),21,21}
    \end{xy}.
\]
Then we identify two spectral types of dimension $2$ or $4$ 
if they have the same shape. This identification might be justified by 
Boalch's simply-laced isomonodromy theory \cite{Bo}. 
By Theorems \ref{tame} and \ref{class}, 
we can check that the identification happens only when the corresponding 
quiver $\mathsf{Q}$ in the spectral data $(\mathsf{Q},\lambda,\alpha)$ 
is simply-laced and $I_{\text{irr}}=\{0\}$. 
Thus integrable deformations with spectral types of the same shape
are isomorphic by Theorem 1.1 in \cite{Bo} in the sense of the paper \cite{Bo}.

In \cite{KNS}, Kawakami, Nakamura and Sakai considered  
isomonodromic deformations of linear differential equations 
which obtained by the confluent process 
from Fuchsian differential equations with $4$ accessory parameters 
classified by Oshima in \cite{O}.
And they gave explicit Hamiltonian equations of the isomonodromic deformations
after Sakai's computation in the Fuchsian cases (see \cite{Sak}).
Then under the above identification of spectral types, 
Theorem \ref{class}
shows that  the list of 
spectral types appeared in their paper \cite{KNS}
is the complete list of fundamental spectral types of dimension 4.

\begin{thm}
	Under the above identification of spectral types, 
	if we exclude the spectral types corresponding to 
	differential equations which  have only 3 regular singular points
	and no other singularities, then the 
	list of spectral types appeared in Section 1.3 of \cite{KNS} 
	is the complete list of spectral
	types of fundamental integrable deformations of dimension 4.
\end{thm}
\begin{proof}
	Compare the list of spectral types in \cite{KNS} and 
	the shapes listed in Theorem \ref{class}.
\end{proof}

Moreover Theorem \ref{integrableWeyl} assures that 
integrable deformations considered in \cite{KNS} have $W^{\text{inv}}$-
symmetries listed in Theorem \ref{class}.

\end{document}